\documentclass[11pt]{article}
\usepackage{geometry}
\usepackage{graphicx}
\usepackage{algorithm}
\usepackage{algpseudocode} 
\usepackage{amsmath,amssymb,amsthm,bm,mathtools,amsfonts,mathrsfs}
\usepackage{enumitem}
\usepackage[numbers,sort&compress]{natbib}
\usepackage[T1]{fontenc}
\usepackage{lmodern}
\usepackage[table]{xcolor}
\usepackage{titlesec}
\usepackage{subcaption}
\usepackage{float}
\usepackage{textcomp}
\usepackage{dsfont}
\usepackage{tcolorbox}
\usepackage{booktabs}
\usepackage{makecell}
\usepackage{multirow}
\usepackage{epstopdf}
\usepackage{threeparttable}

\usepackage{hyperref}
\hypersetup{colorlinks=true, linkcolor=magenta, citecolor=blue, urlcolor=blue}
\usepackage{thmtools}
\usepackage{cleveref}

\newtheorem{theorem}{Theorem}[section]
\newtheorem{lemma}[theorem]{Lemma}
\newtheorem{corollary}[theorem]{Corollary}

\newtheorem{assumption}{Assumption}[section]
\newtheorem{definition}{Definition}[section]
\newtheorem{remark}{Remark}[section]

\usepackage{todonotes}
\usepackage{etoolbox}
\makeatletter
\patchcmd{\@addmarginpar}{\ifodd\c@page}{\ifodd\c@page\@tempcnta\m@ne}{}{}
\makeatother
\reversemarginpar
\usepackage{mathtools}
\allowdisplaybreaks
\usepackage{bbm}
\usepackage{bbding}
\usepackage{amssymb}

\def\grad{\nabla}

\def\cB{\mathcal{B}}
\def\cC{\mathcal{C}}

\def\cF{\mathcal{F}}
\def\cG{\mathcal{G}}

\def\cK{\mathcal{K}}
\def\cL{\mathcal{L}}
\def\cM{\mathcal{M}}
\def\cN{\mathcal{N}}
\def\cO{\mathcal{O}}

\def\cS{\mathcal{S}}
\def\cT{\mathcal{T}}

\def\cX{\mathcal{X}}
\def\cY{\mathcal{Y}}
\def\cZ{\mathcal{Z}}

\def\smskip{\smallskip}

\def\texitem#1{\par\smskip\noindent\hangindent 25pt
               \hbox to 25pt {\hss #1 ~}\ignorespaces}

\def\norm#1{\|#1\|}

\newcommand{\BEAS}{\begin{eqnarray*}}
\newcommand{\EEAS}{\end{eqnarray*}}
\newcommand{\BEA}{\begin{eqnarray}}
\newcommand{\EEA}{\end{eqnarray}}
\newcommand{\BEQ}{\begin{eqnarray}}
\newcommand{\EEQ}{\end{eqnarray}}
\newcommand{\BIT}{\begin{itemize}}
\newcommand{\EIT}{\end{itemize}}
\newcommand{\BNUM}{\begin{enumerate}}
\newcommand{\ENUM}{\end{enumerate}}

\newcommand{\BA}{\begin{array}}
\newcommand{\EA}{\end{array}}

\newcommand{\reals}{\mathbb{R}}
\newcommand{\integers}{\mathbb{Z}}

\newcommand{\Rank}{\mathop{\bf rank}}

\def\fprod#1{\left\langle#1\right\rangle}

\newcommand{\dom}{\mathop{\bf dom}}

\usepackage{xcolor}

\newif\ifpagenumbering
\pagenumberingtrue

\pagenumberingfalse
\newsavebox{\theorembox}
\newsavebox{\lemmabox}
\newsavebox{\defnbox}
\newsavebox{\assbox}
\savebox{\theorembox}{\noindent\bf Theorem}
\savebox{\lemmabox}{\noindent\bf Lemma}
\savebox{\defnbox}{\noindent\bf Definition}

\usepackage{soul}
\usepackage[normalem]{ulem}
\DeclareUnicodeCharacter{2212}{~}

\usepackage{hyperref}
\usepackage{cleveref}
\crefname{assumption}{Assumption}{Assumptions}
\crefname{theorem}{Theorem}{Theorems}
\crefname{lemma}{Lemma}{Lemmas}

\usepackage{relsize}

\usepackage{pifont}
\newcommand{\cmark}{\ding{51}}
\newcommand{\xmark}{\ding{55}}

\definecolor{darkgreen}{rgb}{0.2,0.8,0.1}
\def\jw#1{\textcolor{black}{#1}}

\def\sa#1{\textcolor{black}{#1}}
\def\rev#1{\textcolor{black}{#1}}

\def\rapdb{\texttt{rAPDB}}
\def\rapdbada{\texttt{rAPDB-ada}}
\def\apdbxy{\texttt{APDB-xy}}
\def\apdbyx{\texttt{APDB-yx}}
\def\rapdbxy{\texttt{rAPDB-xy}}
\def\rapdbyx{\texttt{rAPDB-yx}}

\newcommand{\R}{\mathbb R}
\newcommand{\dist}{\operatorname{dist}}

\newcommand{\feas}{\mathcal X_{\rm feas}}
\newcommand{\Kcone}{\mathcal K}



\title{Restarted Accelerated Primal-Dual Algorithms with\\ Adaptive Stepsizes for Convex Nonlinear Conic Programs}
\author{
Necdet Serhat Aybat\thanks{The authors contributed equally to this work.}\\
\small{Department of Industrial and Manufacturing Engineering}\\
\small{Penn State University}\\
{\small \texttt{nsa10@psu.edu}}
\and
Jinxin Wang\footnotemark[1]\\
\small{Department of Statistics}\\
\small{University of Chicago}\\
{\small\texttt{wangjinxin68@gmail.com}}
}
\date{}

\begin{document}

\maketitle

\begin{abstract}
We propose restarted accelerated primal-dual algorithms with (non-monotone) backtracking (rAPDB) for convex nonlinear conic programs, with quadratically constrained quadratic programs (QCQPs) as a special case. Unlike linear and quadratic programs, these problems give rise to convex-concave minimax reformulations with \emph{non-bilinear} coupling terms; therefore, the existing primal-dual methods for \textit{bilinear} couplings are not 
applicable. To address this challenge, we build on the accelerated primal-dual method with \sa{adaptive} stepsize search\sa{---as it adapts to the local curvature---}and develop both fixed-frequency and adaptive restart schemes, incorporating both monotone and non-monotone adaptive step-size search strategies. The resulting algorithms require only first-order information and matrix-vector products, making them suitable for large-scale and GPU-accelerated implementation. Under metric subregularity of the KKT mapping, we prove a quadratic growth property for a self-centered smoothed duality gap and establish global linear convergence of the proposed restarted methods. We also establish sufficient conditions under which the metric subregularity holds even for general \emph{nonconvex} problems over convex polyhedral cones. These results are new and may be of independent interest. Numerical experiments on random QCQPs and kernel matrix learning instances show that the proposed methods, especially with non-monotone adaptive stepsizes and GPU acceleration, achieve strong practical performance.
\end{abstract}

\section{Introduction}
Convex nonlinear conic programs (NCPs), which include linear programs (LPs), quadratic programs (QPs), quadratically constrained quadratic programs (QCQPs), second-order cone programs (SOCPs), and semidefinite programs (SDPs) as special cases, are a fundamental class of optimization problems and arise widely in signal processing \cite{luo2010semidefinite}, machine learning \cite{ye2007learning,komiyama2018nonconvex,wang2022solving,gabriel2019computational}, and decision-making under uncertainty \cite{phan1982quadratically,ben2001lectures,alizadeh2003second,chen2024exponential}.
Although conic programs can be solved reliably to high accuracy by
interior-point-type methods, as implemented in solvers such as Gurobi and
MOSEK, these approaches typically rely on repeated factorizations of large
linear systems. Sparse direct solvers exploit the sparsity of these systems,
but their efficiency depends heavily on the sparsity pattern \cite{davis2016survey}. 
For large-scale instances, the resulting
factors can be substantially denser than the original problem data, leading
to significant memory and computational costs \cite{lu2026practical}.
Furthermore, due to the sequential nature of factorization, it is challenging for factorization-based routines to take advantage of parallel and distributed computing, \sa{e.g., factorization-based methods cannot fully exploit the parallel computing capability of GPUs.}

By contrast, first-order primal-dual (FOPD) methods based on matrix-vector multiplication can leverage the sparsity structure well and have low-memory overhead. Furthermore, these methods scale quite well on modern computing architectures including GPUs and distributed computing. Indeed, recent developments 
\sa{in the design and analysis of FOPD} methods for large-scale LPs \cite{applegate2023faster,fercoq2023quadratic,applegate2021practical,chen2025relationships} and QPs \cite{lu2026practical,huang2025restarted} have shown strong empirical gains \sa{in comparison to} 
simplex- or interior-point-type methods; see also \cite{lu2025overview} for an overview of recent GPU-based first-order methods. In particular, Applegate et al. in \cite{applegate2023faster} studied 
\sa{various FOPD} methods, 
\sa{including the} alternating direction method of multipliers (ADMM), the primal-dual hybrid gradient method (PDHG), and the extragradient method (EGM), for solving the \emph{bilinear} minimax reformulation of large-scale LPs. \sa{By employing systematic restarts of the method and by exploiting the} sharpness of a normalized duality gap, 
\sa{these aforementioned} methods achieve fast linear convergence rate, which suggests that 
\sa{the class of FOPD methods is rich enough and methods from this class} can reach moderate-to-high accuracy in practice within a reasonable time budget. In a concurrent work \cite{fercoq2023quadratic}, Fercoq considered more general bilinear minimax problems (including LPs) and proposed the quadratic growth of a smoothed duality gap to obtain linear convergence rate of restarted PDHG. Building on the PDHG framework, Applegate et al.~\cite{applegate2026pdlp}
incorporated several practical enhancements, such as diagonal preconditioning,
presolving, adaptive stepsizes, adaptive restarting, and feasibility polishing,
to achieve significant speedups. Going beyond LPs, Lu et al.~\cite{lu2026practical} extended accelerated PDHG with restarts
from solving LPs to solving QPs, and obtained linear convergence rate and optimal iteration complexity among a wide class of primal-dual methods. Empirically, they showed that PDHG with restarts and acceleration outperforms other operator-splitting-based first-order methods including OSQP \cite{stellato2018osqp} and SCS \cite{ocpb:16,o2021operator}. In addition, Huang et al. in \cite{huang2025restarted} proposed a restarted primal-dual hybrid conjugate gradient (PDHCG) method for solving QPs, where the conjugate gradient method is employed to solve the primal subproblems. It is shown that PDHCG achieves linear convergence rate with improved convergence constants and superior empirical performance compared with the restarted accelerated PDHG in \cite{lu2026practical}.~\jw{Another related line of work is \cite{lin2025pdcs,lin2026technical}, where Lin et al. proposed restarted PDHG-type methods with enhanced implementations for linear conic programs (LCPs), including second-order cones
and exponential cones, but with rather limited theoretical guarantees.} It is worth highlighting that all these works rely on transforming LPs, QPs, \jw{and LCPs} to minimax optimization problems with \emph{bilinear} coupling, which enables sharpness or quadratic growth properties via Hoffman's bound for linear systems. Besides, motivated by theoretical strengths of HPR over PDHG, Chen et al. in~\cite{chen2025hpr} presented a Halpern Peaceman-Rachford (HPR) method with adaptive restarts for solving large-scale LPs and then extended it to solving large-scale QPs in \cite{chen2025hprqp}. Notably, the HPR framework delivers superior performance over the PDHG framework.

In this paper, we propose efficient FOPD methods (called \rapdb{} and \rapdbada{}) for general NCPs 
\sa{achieving} \emph{linear convergence rate} under reasonable regularity assumptions.
A key difference between QPs and more general NCPs, e.g., QCQPs, is that the coupling term in the minimax formulation is 
not bilinear anymore \sa{whenever NCPs have nonlinear functional constraints. This difference renders} many classical primal-dual splitting schemes tailored for bilinear minimax problems \textit{not}
directly applicable, including the methods discussed above for solving LPs
and QPs. \rev{A notable recent work \cite{diaz2026active} establishes finite active-set identification and subsequent \textit{local} linear convergence under metric subregularity, and considers EGM for nonlinear functional constraints such as those in QCQPs.}
\rev{Indeed, the results for the meta algorithm in \cite{diaz2026active} are given for 
$\min\{f(x):\ g(x)\leq 0\}$, where $f:\reals^n\to\reals$ and $g:\reals^n\to\reals^m$ are smooth and $g$ is Lipschitz.} 
\rev{For nonlinear $g$, e.g., for QCQPs, \cite{diaz2026active} states that PPM and EGM satisfy the requirements of its meta-algorithm. However, each PPM iteration requires the exact solution of a strongly convex–strongly concave saddle-point problem. For EGM, \cite[Proposition 3.5]{diaz2026active} assumes that $f(x)+\langle y,g(x)\rangle$ has a globally Lipschitz gradient on $\mathbb R^n\times\mathbb R_+^m$. For nonlinear $g$, such a \textit{global} constant generally does not exist, as 
it can grow with $\norm{y}$. This issue could potentially be addressed by showing that the main and intermediate EGM iterates remain bounded and using a Lipschitz constant on that region. This, however, would generally yield a stepsize restriction depending on an a priori unknown trajectory-dependent Lipschitz constant.}
Similar to \cite{diaz2026active}, our
convergence guarantees for the proposed methods are also entirely nonasymptotic and, importantly, do not rely on strict complementarity; and in contrast to \cite{diaz2026active}, 
\rev{our methods choose step sizes adaptively and still} guarantee \textit{global} linear convergence under metric subregularity, while the EGM analyzed in \cite{diaz2026active} \rev{requires a constant step size that depends on some unknown problem parameters 
and achieves} local linear convergence after identification of the active set. \sa{One significant practical drawback} of \cite{diaz2026active} is that it fixes the same constant stepsize for both primal and dual updates, which can lead to unsatisfactory practical performance when the formulation is unbalanced, e.g., when the curvature in the objective function is much larger than that in the constraint functions.

Recently, Hamedani and Aybat~\cite{hamedani2021primal}
proposed an accelerated primal-dual method with (non-monotone) backtracking (APDB) for a quite general class of minimax optimization problems with non-bilinear coupling terms, which includes the minimax reformulation of convex
optimization problems with nonlinear conic constraints. Such an adaptive step-size selection rule \sa{(as it adapts to the local curvature)} avoids manual parameter tuning and usually allows for larger step sizes by exploiting local Lipschitz constants, which can be significantly smaller than the global constant. Hamedani and
Aybat~\cite{hamedani2021primal} established that suboptimality in terms of
function values and constraint violation both diminish to zero with
sublinear convergence rates of \(\cO(1/k)\) and \(\cO(1/k^2)\) for convex
optimization problems with nonlinear conic constraints when the objective is
merely convex and strongly convex, respectively. That being said, for the
majority of practical applications, sublinear convergence rate guarantees are not desirable when the goal is to compute near-optimal solutions with \textit{moderate-to-high accuracy}. In order to achieve linear convergence rate guarantees while avoiding manual step-size tuning or heuristic step-size selection rules, we propose restarted versions of APDB: \rapdb{} with fixed-frequency restarts and \rapdbada{} with adaptive restarts.

Our concrete contributions are as follows:
\begin{itemize}
\item[(i)] We propose several \textit{restarted} versions of APDB, namely
\rapdb{} and \rapdbada{}, for tackling large-scale convex NCPs, with QCQPs as a special case. These methods only rely on
matrix-vector multiplications and simple projection-type operations; thus, they can take advantage of modern computing architectures such as GPUs to
achieve significant acceleration.\footnote{Due to its alternating update
rule, APDB proposed in~\cite{hamedani2021primal} does not have symmetric
updates for primal and dual variables. Therefore, for both \rapdb{} and
\rapdbada{}, we consider two variants: one with the primal \((x)\)-update
followed by the dual \((y)\)-update, denoted by \texttt{rAPDB-xy} and
\texttt{rAPDB-xy-ada}, and the other with the dual update followed by the
primal update, denoted by \texttt{rAPDB-yx} and \texttt{rAPDB-yx-ada}.} A prominent feature of \rapdb{} and
\rapdbada{} is their adaptive step-size selection scheme with both monotone and non-monotone strategies, which often
enables larger stepsizes and faster practical convergence.

\item[(ii)] We establish sufficient conditions for the metric
subregularity of the KKT map for a broad class of \emph{nonconvex} NCPs \sa{subject to nonlinear functional constraints defined by polyhedral cones (see~\Cref{thm:jongshi}).} These results extend the related error-bound results in~\cite{facchinei2003finite} and show that the key regularity assumption
underlying our linear convergence analysis holds under natural constraint
qualifications and a tractable condition excluding nonzero critical
directions, which may be of independent interest.

\item[(iii)] Under metric subregularity of the corresponding KKT mapping, we first show a quadratic growth property of a self-centered smoothed duality gap \sa{(see~\Cref{thm: qg}).}
Using this property, we further establish global linear convergence of \rapdb{}
and \rapdbada{} \sa{(see~\Cref{thm: lr_rapdb} and \Cref{sec:rAPDB-ada})}. This enables the proposed \rapdb{}
and \rapdbada{} to reach
moderate-to-high accuracy within a reasonable time frame. To the best of
our knowledge, this is the first restarted primal-dual framework
with linear convergence guarantees for convex NCPs \sa{--in general, these problems lead to non-bilinear coupling in their minimax reformulation.}

\item[(iv)] Numerical experiments on random QCQPs and kernel matrix
learning problem instances demonstrate that the proposed schemes are practically efficient and outperform state-of-the-art solvers, including
interior-point-type methods.
\end{itemize}

To highlight our contributions, in \Cref{tab:intro-comparison} \sa{we compare  the FOPD methods closely related to ours.}

\begin{table}[t]
\centering
\begin{threeparttable}
\caption{Comparison with related work.}
\label{tab:intro-comparison}
\renewcommand{\arraystretch}{1.2}
\begin{tabular}{lcccc}
\toprule
\makecell{Problem\\ class} & Method & \makecell{\sa{Functional}\\ \sa{constraints\tnote{\dag}}} & \makecell{Adaptive\\ Stepsizes} & \makecell{Linear\\ rate} \\
\midrule
\multirow{3}{*}{LPs}
& PDHG \cite{applegate2023faster,fercoq2023quadratic} 
& \xmark & \xmark & \cmark \\
& PDLP \cite{applegate2021practical,applegate2026pdlp} 
& \xmark & \cmark\tnote{2} & \cmark \\
& HPR-LP \cite{chen2025hpr} 
& \xmark & \cmark\tnote{3} & \xmark \\
\midrule
\multirow{3}{*}{QPs}
& PDQP \cite{lu2026practical} 
& \xmark & \xmark & \cmark \\
& PDHCG \cite{huang2025restarted} 
& \xmark & \cmark\tnote{4} & \cmark \\
& HPR-QP \cite{chen2025hprqp} 
& \xmark & \cmark\tnote{3} & \xmark \\
\midrule
\multirow{4}{*}{\makecell{Conic\\ programs}}
& \rev{PPM} \& EGM~\cite{diaz2026active} 
& \cmark & \xmark & \cmark\tnote{1} \\
& APDB \cite{hamedani2021primal} 
& \cmark & \cmark & \xmark \\
& PDCS \cite{lin2025pdcs,lin2026technical} 
& \xmark & \cmark\tnote{4} & \xmark \\
& \rapdb{} and \rapdbada{} (ours) 
& \cmark & \cmark\tnote{5} & \cmark \\
\bottomrule
\end{tabular}
\begin{tablenotes}
\footnotesize
\item[$\dag$] \sa{This column summarizes whether the method can handle \textit{functional constraints} of the form $g(x)\in-\cK$ for some \textit{nonlinear} $g:\reals^n\to\reals^m$ and closed convex cone $\cK$.}
\item[1] \rev{
\cite{diaz2026active}
considers 
$g(x)\leq 0$, rather than 
\(g(x)\in-\cK\), and establishes \textit{local} linear convergence under metric
subregularity without strict complementarity.
\rev{Each iteration of PPM requires solving an SCSC saddle point
problem. EGM adopts a constant step size, the same for both primal
and dual updates. 
For nonlinear $g$, the global smoothness assumption used in the EGM result generally fails; see the discussion on page 2.
}}
\item[2] \jw{PDLP employs a heuristic line-search scheme, 
see \cite[Section 3]{applegate2026pdlp}.}
\item[3] \jw{HPR-LP and HPR-QP adjust the penalty parameter adaptively.}
\item[4] \jw{PDHCG and PDCS adopt a heuristic line-search scheme similar to that in PDLP, see \cite[Section 4]{huang2025restarted} and \cite[Section 2.1]{lin2026technical}.}
\item[5] \jw{\rapdb{} and \rapdbada{} allow for both monotone and non-monotone step-size search.}
\end{tablenotes}
\end{threeparttable}
\end{table}



\paragraph{Notation.} Throughout, given $m\in\integers_+$, let $[m]\triangleq\{1,\ldots,m\}$. We use $\mathbb{S}^n_+$ to denote the set of $n$-by-$n$ symmetric and positive semidefinite matrices and $I_n$ denotes the $n$-by-$n$ identity matrix. For a closed convex 
set $\mathcal{C} \subseteq \mathbb{R}^n$, we denote the indicator function of $\mathcal{C}$ by $\iota_{\mathcal{C}}$, i.e., 
\[
\iota_{\mathcal{C}}(x) \;\triangleq \;
\begin{cases}
0, & x \in \mathcal{C},\\
+\infty, & x \notin \mathcal{C};
\end{cases}
\]
moreover, \sa{for any $x\in\cC$, the set $\cN_{\cC}(x)\triangleq\{d\in\reals^n:\ \fprod{d,x'-x}\leq 0 \quad\forall~x'\in\cC\}$ defines the normal cone of $\cC$ at $x\in\cC$.}
The norm $\norm{\cdot}$ denotes the Euclidean norm for vectors and the spectral norm for matrices. \sa{Given $r>0$ and $\bar x\in\reals^n$, let $\cB_r(\bar x)\triangleq\{x:~\norm{x-\bar x}<r\}$ denote the open ball centered at $\bar x$ with radius $r>0$, and $\bar\cB_r(\bar x)$ denotes its closure}. Given a closed set $\mathcal{C} \subseteq \mathbb{R}^n$, the distance function from 
$x \in \mathbb{R}^n$ to the set $\mathcal{C}$ is defined by
\begin{align*}
{\rm dist}(x, \mathcal{C}) \triangleq \min_{x' \in \mathcal{C}} \| x - x'\|.
\end{align*}
For a closed convex cone $\mathcal{K} \subseteq \mathbb{R}^m$, \sa{$\cK^\circ\triangleq -\cK^*$ is the polar cone of $\mathcal{K}$,} where $\mathcal{K}^*$ is the dual cone, i.e.,
\[
\mathcal{K}^* \triangleq \{ y \in \mathbb{R}^m : \langle y,x\rangle \ge 0,\ \forall x \in \mathcal{K} \}.
\]
\jw{For a point $x \in \mathbb{R}^m$, we denote the \sa{Euclidean} projection of $x$ onto the cone $\cK$ by $\Pi_{\cK}(x)$.}
In the text, we use ``iff'' to abbreviate ``if and only if''. \jw{For two nonnegative sequences \(\{a_k\}_{k \ge 0}\) and \(\{b_k\}_{k \ge 0}\), we write \(a_k=\mathcal{O}(b_k)\) if there exist constants \(C>0\) and \(k_0\ge 0\) such that \(a_k\le Cb_k\) for all \(k\ge k_0\); \(a_k=\Omega(b_k)\) if there exist constants \(c>0\) and \(k_0\ge 0\) such that \(a_k\ge cb_k\) for all \(k\ge k_0\); and \(a_k=\Theta(b_k)\) if both \(a_k=\mathcal{O}(b_k)\) and \(a_k=\Omega(b_k)\).} \sa{For a vector-valued map $g:\reals^n\to\reals^m$ that is differentiable at $x$, $\grad g(x)\in\reals^{m\times n}$ denotes the Jacobian matrix of $g$ at $x$.}
\section{Problem Formulation and Preliminaries}
In this paper, we consider a 
class of convex optimization problems subject to nonlinear conic constraints in the following form:
\begin{equation}\label{eq:conic-problem} \tag{NCP}
-\infty<f^\star=\min_{x\in \mathcal{X}}\{f(x):\  Ax=b,\quad g(x)\in -\mathcal{K}\},
\end{equation}
where $\mathcal{X}\subseteq\mathbb{R}^n$, $\mathcal{K} \subseteq \mathbb{R}^m$, $A\in\reals^{p\times n}$, $f:\reals^n\to \reals$, and $g:\reals^n\to\reals^m$ satisfy Assumptions~\ref{as:sets} and~\ref{as:Lipschitz-grads}.
\sa{
\begin{assumption}
\label{as:sets}
    Let $\mathcal{X}\subseteq\mathbb{R}^n$ be a convex \emph{compact} set,
    and $\cK\subset\reals^m$ be a proper cone\footnote{We note that the set $\mathcal{K}$ can also be a Cartesian product of smaller dimensional cones, i.e., $\mathcal{K}=\mathcal{K}_1 \times \dots \times \mathcal{K}_\ell$ for some $\ell >1$, where the dual cone $\mathcal{K}^* = \mathcal{K}_1^* \times \dots \times \mathcal{K}_\ell^*$.}, i.e., a closed,
convex, pointed cone with nonempty interior. Moreover, let $A\in\mathbb{R}^{p\times n}$ satisfy $\Rank(A)=p<n$.
\end{assumption}}
\begin{assumption}
\label{as:Lipschitz-grads}
    Consider the problem \eqref{eq:conic-problem} with some proper cone $\cK\subset\reals^m$. Let $f:\reals^n\to \reals$ be convex on $\cX$ with modulus $\mu\geq 0$ and $g:\reals^n\to \reals^m$ be $\cK$-convex\footnote{See \cite[Section 3.6.2]{boyd2004convex} for details of $\mathcal{K}$-convexity.} on $\cX$. Suppose both $f$ and $g$ are continuously differentiable on an open set containing $\cX$ such that $\grad f$ and $\grad g$ are  Lipschitz on $\cX$, i.e., there exist constants $L_f,L_g\geq 0$ such that 
    $$\norm{\grad f(x')-\grad f(x)}\leq L_f \norm{x'-x}, \qquad \norm{\grad g(x')-\grad g(x)}\leq L_g \norm{x'-x},\quad\forall~x',x\in\cX.$$
\end{assumption}
\sa{Trivially, we can conclude that $g$ is Lipschitz on $\cX$ since 
the Jacobian $\nabla g: \mathbb{R}^n \rightarrow \mathbb{R}^{m \times n}$ is continuous and $\cX$ is compact.}
\begin{definition}
\label{def:Lipschitz-g}
Let $C_g>0$ denote the smallest Lipschitz constant of $g$ over $\mathcal{X}$, i.e., 
\begin{align} \label{eq:def_L_g}
\| g(x)-g(x')\| \le \sa{C_g}\|x - x'\|,\quad \forall\ x,x' \in \mathcal{X};
\end{align}
hence, it holds that $C_g\leq B_g$, where
\begin{align} \label{eq:def_B_g}
B_g \triangleq \max_{x \in \mathcal{X}} \| \nabla g(x) \|.
\end{align}
\end{definition}
\sa{
Let $\cY\triangleq\mathbb{R}^p \times \cK^*$ denote the dual domain. By introducing multipliers $v\in\reals^p$ and $\lambda \in \mathcal{K}^*\subset\reals^m$, we define the coupling function\footnote{\jw{For any $\lambda \in \cK^*$, the function \(x\mapsto \langle \lambda,g(x)\rangle\) is convex due to $\cK$-convexity of g.}} $\Phi: 
\cX\times\cY \rightarrow \mathbb{R}$ such that for any given 
$x\in\cX$ and $y=(v,\lambda)\in\cY$,}
\begin{align}
\label{eq:ncc-coupling}
    \Phi(x,y) \triangleq f(x) + \fprod{v,~Ax-b} + \langle\lambda,~g(x)\rangle;
\end{align} 
hence, 
\eqref{eq:conic-problem} 
can be equivalently written as the following minimax optimization problem:
\begin{align} \label{eq:minimax-pro} \tag{min-max}
\min_{x\in \mathbb{R}^n} \max_{y \in \mathbb{R}^{p} \times \mathbb{R}^{m}}\mathcal{L}(x,y),\quad\mbox{where}\quad \mathcal{L}(x,y) \triangleq 
\iota_{\mathcal{X}}(x) + \Phi(x,y) - \sa{\iota_{\cY}(y)}.
\end{align}
\begin{definition}
    Let $\mathcal{Z}^\star \subseteq \mathcal{Z} \triangleq \sa{\mathcal{X} \times \cY}$ denote the set of saddle points for \eqref{eq:minimax-pro}, i.e.,
\begin{align*}
  (x^\star,y^\star)\in\cZ^\star \iff \mathcal{L}(x^\star,y) \le  \mathcal{L}(x^\star,y^\star) \le \mathcal{L}(x,y^\star),\quad \sa{\forall\ x\in\reals^n,\quad\forall\ y\in\reals^p\times\reals^m}.
\end{align*}
Furthermore, let $\cX^\star\subset\cX$ and $\cY^\star\subset\cY$ denote the sets of primal and dual optimal solutions to \eqref{eq:conic-problem}.
\end{definition}
\begin{assumption}
\label{as:saddle-point}
     Suppose that $\cZ^\star$ is non-empty.
\end{assumption}
\begin{lemma}
\label{lem:product-saddle-set}
    Assumption~\ref{as:saddle-point} \sa{implies that} $\cX^\star\neq\emptyset$, $\cY^\star\neq\emptyset$, and strong duality holds for \eqref{eq:conic-problem}; indeed, $\cZ^\star=\cX^\star\times\cY^\star$, i.e., $(x^\star,y^\star) \in \mathcal{Z}^\star$ if and only if $x^\star\in\cX^\star$ and $y^\star\in\cY^\star$. Furthermore, \sa{whenever Assumption~\ref{as:saddle-point} holds, 
    $(x^\star,y^\star) \in \mathcal{Z}^\star$ iff $(x^\star,y^\star)$ is a KKT point.}
\end{lemma}
\begin{proof}
    This result is standard; therefore, we omit its proof.
\end{proof}
\subsection{A special case: convex QCQPs}
The problem class \eqref{eq:conic-problem} includes convex QCQPs as a special case; indeed, consider
\begin{equation}\label{eq: qcqp-pro} \tag{QCQP}
-\infty<f^\star=\min_{x\in \mathcal{X}}\Big\{f(x):\ Ax=b, \quad g_i(x)\le 0,\quad i\in[m]\Big\},
\end{equation}
where
\[
f(x)=\frac12 x^\top Q_0 x + q_0^\top x + r_0,\qquad
g_i(x)=\frac12 x^\top Q_i x + q_i^\top x + r_i,\quad i\in[m]
\]
such that 
\sa{$Q_i\in 
\mathbb{S}^n_+$ for $i\in\{0,\dots,m\}$} are positive semidefinite, 
$\mathcal{X}\subseteq\mathbb{R}^n$ is a simple convex \emph{compact} set, e.g., a box or a simplex, the matrix $A \in \mathbb{R}^{p \times n}$ has full row rank, and $f$ is convex with modulus $\mu \ge 0$, where \sa{the case with $\mu=0$ corresponds to the mere convexity of $f$.} We also assume that the feasible set \sa{$\mathcal{X}_{\rm feas}\triangleq\{x\in\cX:\ Ax=b, \; g_i(x)\le 0,\; i\in[m]\}$} is non-empty. 

For \eqref{eq: qcqp-pro},  
{\small 
$\grad g(x)=\begin{bmatrix}
    (Q_1x+q_1)^\top\\
    \vdots\\
    (Q_m x+q_m)^\top
\end{bmatrix}$} is Lipschitz 
with constant $L_g = (\sum_{i=1}^m \| Q_i \|^2)^{1/2}$ since
{\footnotesize
\begin{align*}
\left( \sum_{i\in[m]} \| Q_i \|^2 \right)^{1/2}\geq \sqrt{\lambda_{\rm max}\left(\sum_{i\in[m]} Q_i^2\right)}=\left\|
\begin{bmatrix}
Q_1, \dots, Q_m
\end{bmatrix}
\right\|;
\end{align*}}%
and $\nabla f$ is Lipschitz continuous with the constant $L_f = \|Q_0\|$; moreover, if $\mathcal{X} \subseteq \{x\in \mathbb{R}^n: \|x\| \le r\}$ for some $r >0$, then $B_g = \max_{x \in \mathcal{X}} \| \nabla g(x) \| \le \| [q_1, \dots, q_m] \| + r\cdot \sqrt{\sum_{i=1}^m \| Q_i \|^2}$.

For the problem \eqref{eq: qcqp-pro}, as a special case of the general conic formulation in \eqref{eq:conic-problem}, 
the cone is set to $\mathcal{K}=\mathbb{R}^m_+$; hence,
$\mathcal{K}^*=\mathbb{R}^m_+$. Accordingly, the dual domain becomes $\mathcal{Y}=\mathbb{R}^p\times\mathbb{R}^m_+$. The Lagrangian function
$\mathcal{L}$ defined in~\eqref{eq:minimax-pro} therefore reduces to
\begin{align*}
\mathcal{L}(x,y)=
\iota_{\mathcal{X}}(x)+f(x)+
\langle v,Ax-b\rangle+
\langle \lambda,g(x)\rangle-
\iota_{\mathbb{R}^m_+}(\lambda),
\end{align*}
where $g:\mathbb{R}^n\to\mathbb{R}^m$ is given by
$g(x)=[g_i(x)]_{i\in[m]}$. Equivalently, $\mathcal{L}(x,y)
=\iota_{\mathcal{X}}(x)+\Phi(x,y)- \iota_{\mathbb{R}^m_+}(\lambda)$,
where the coupling function $\Phi$ in~\eqref{eq:ncc-coupling} takes the following form:
\begin{align}
\Phi(x,y)&=
f(x)+v^\top(Ax-b)+\lambda^\top g(x)\nonumber\\
&=
\frac12 x^\top
\Big(Q_0+\sum_{i\in[m]}\lambda_i Q_i\Big)x+
\Big(q_0+\sum_{i\in[m]}\lambda_i q_i\Big)^\top x+
r_0+\sum_{i\in[m]}\lambda_i r_i+v^\top(Ax-b).
\label{eq:qcqp-coupling}
\end{align}

\subsection{Bregman distances}
\begin{definition}
\label{def:bregman}
Let $\varphi_{p}: \mathbb{R}^n \to\mathbb R$ and $\varphi_{d}: \mathbb{R}^p \times \mathbb{R}^m \to\mathbb R$ be differentiable
on open sets containing $\mathcal{X}$ and $\cY$ respectively. Assume that $\varphi_{p}$ and $\varphi_{d}$
\sa{are both closed and $1$-strongly convex functions}. Let $\mathbf{D}_{p}: \mathbb{R}^n \times \mathbb{R}^n \to\mathbb R_{+}$
denote the Bregman distance
associated with $\varphi_{p}$, i.e.,
\[
\mathbf{D}_{p}(x,\bar x)\triangleq
\varphi_{p}(x)-\varphi_{p}(\bar x)
-\langle \nabla\varphi_{p}(\bar x),\,x-\bar x\rangle,\quad\sa{\forall~x,\bar x\in\cX}.
\]
The Bregman distance $\mathbf{D}_{d}: (\mathbb{R}^p \times \mathbb{R}^m)\times(\mathbb{R}^p \times \mathbb{R}^m)\to\mathbb R_{+}$ associated with $\varphi_{d}$ is defined \sa{on $\cY$} analogously. \sa{Furthermore, let $\mathbf{D}:\cZ\times\cZ\to\reals_+$ such that $\mathbf{D}(z,\bar z)=\mathbf{D}_p(x,\bar x)+\mathbf{D}_d(y,\bar y)$ for any $z,\bar z\in\cZ$.}
\end{definition}

By $1$-strong convexity, it holds that $\mathbf{D}_{p}(x,\bar x)\ge \frac12\|x-\bar x\|^{2}$ for all $x\in \mathbb{R}^n$ and $\bar x\in \mathcal{X}$, and
$\mathbf{D}_{d}(y,\bar y)\ge \frac12\|y-\bar y\|^{2}$ for all $y\in \mathbb{R}^p \times \mathbb{R}^m$ and $\bar y\in \sa{\cY}$. 
Throughout the paper, we assume that both $\grad\varphi_p$ and $\grad\varphi_d$ 
\sa{are Lipschitz continuous over the problem domain} with a constant $L_\varphi>0$, i.e., 
\begin{align}\label{eq: L-varphi-bregman}
\begin{split}
\| \nabla \varphi_p(x) - \nabla \varphi_p(\bar x) \| &\le L_\varphi \| x - \bar x\|,\quad \forall x,\bar x \in \mathcal{X},\\
\| \nabla \varphi_d(y) - \nabla \varphi_d(\bar y) \| &\le L_\varphi \| y - \bar y\|, \quad \forall y,\bar y \in \sa{\cY}; 
\end{split}
\end{align}
which implies that
\begin{align}
\label{eq:Bregman-bounds}
 \sa{\frac{1}{2} \le \mathbf{D}_{p}(x,\bar x)/\|x-\bar x\|^{2}  \le \frac{L_\varphi}{2}, \quad \forall x,\bar x \in \mathcal{X}, \quad \frac{1}{2}\le \mathbf{D}_{d}(y,\bar y)/\|y-\bar y\|^{2} \le \frac{L_\varphi}{2}, \quad \forall y,\bar y \in \sa{\cY}.} 
\end{align}

\begin{remark}
    \sa{\(\varphi_p(x)=\frac12\norm{x}^2\) and \(\varphi_d(y)=\frac12\norm{y}^2\) satisfy above assumptions with \(L_\varphi = 1\). Therefore, the Euclidean setting is a special case of the Bregman setup considered in this paper.}
\end{remark}

\section{Introduction to APDB for Nonlinear Conic Programs}
\sa{The coupling term $\Phi$ in \eqref{eq:minimax-pro} formulation corresponding to \eqref{eq: qcqp-pro} and \eqref{eq:conic-problem} is given in \eqref{eq:qcqp-coupling} and \eqref{eq:ncc-coupling}, respectively; since these coupling functions are \textit{not} bilinear in general, 
one cannot use the existing primal-dual algorithms designed for minimax optimization problems with bilinear coupling terms,} e.g., \cite{chambolle2011first,chambolle2016ergodic,malitsky2018first}. One natural candidate for tackling minimax optimization with \sa{more general coupling terms} is the recently proposed accelerated primal-dual method (APD) in \cite{hamedani2021primal}, which extends the Chambolle–Pock algorithm in \cite{chambolle2016ergodic} \sa{from bilinear coupling functions to smooth convex-concave couplings. Moreover, in \cite{hamedani2021primal}, the authors also proposed APD with a backtracking step-size search (APDB). The variant in \cite[Algorithm 2.3]{hamedani2021primal} generates a monotonically decreasing sequence of primal step sizes, whereas the variant discussed in \cite[Remark 3.8]{hamedani2021primal} can adaptively determine step sizes based on local curvature information by adopting a step-size search scheme that permits a \emph{non-monotone} sequence of primal step sizes. Both APDB variants are agnostic to the global Lipschitz constants pertaining to the problem; hence, such adaptive step-size selection schemes become essential when it is difficult to estimate certain Lipschitz constants. Moreover, even if those (global) Lipschitz constants can be estimated or known,} they can be very large and lead to overly conservative step sizes.

\begin{algorithm}[htb]
\caption{\apdbxy($x^0,y^0,K$)} 
\label{alg:apdb}
\begin{algorithmic}[1]
\Require \sa{$K\in\integers_+$,} $(x^0,y^0)\in \mathcal X\times \cY$, $\mu\ge0$, $c_\alpha,c_\beta,\delta \ge 0$ with $c_\alpha+c_\beta+\delta \in (0,1)$ 
\Require $\eta\in(0,1)$, $\bar\tau$, $\gamma_0>0$, \sa{$c_{\rm nm}\in\{0,1\}$} 
\State $(x^{-1},y^{-1})\gets(x^0,y^0)$,\; $\tau_{-1}\gets\bar\tau$,\;$\tau_0\gets\bar\tau$,\; $\sigma_{-1}\gets\gamma_0\tau_0$,\; \sa{$(x_{\rm cum},y_{\rm cum})\gets (0,0)$,\; $T\gets 0$}
\For{$k=0,\dots,\sa{K-1}$}
\Loop
\State $\sigma_k \gets \gamma_k\tau_k$,\quad $\theta_k \gets \frac{\sigma_{k-1}}{\sigma_k}$,\quad $\alpha_{k+1}\gets \frac{c_\alpha}{\tau_k}$,\quad $\beta_{k+1}\gets \gamma_0 \frac{c_\beta}{\sigma_k}$ \label{algeq:alphabeta}
\State $s^k \gets (1+\theta_k)\nabla_x \Phi(x^k,y^k)
- \theta_k \nabla_x \Phi(x^{k-1},y^{k-1})$ \label{algeq:s-update}
\State $x^{k+1} \gets 
\mathop{\arg\min}_{x\in\mathcal X} \left\{ \langle s^k,x\rangle + \frac{1}{\tau_k} \textbf{D}_p(x,x^k)\right\}$ \label{algeq:x-update}
\State $y^{k+1} \gets 
\mathop{\arg\min}_{y\in \mathbb{R}^p \times \cK^*} 
\left\{
- \langle \nabla_y \Phi(x^{k+1},y^k),y\rangle
+ \frac{1}{\sigma_k} \textbf{D}_d(y,y^k) \right\}$ \label{algeq:y-update}
\If{$E_k(x^{k+1},y^{k+1}) 
\le -\frac{\delta}{\tau_k}\textbf{D}_p(x^{k+1},x^k)
- \frac{\delta}{\sigma_k}\textbf{D}_d(y^{k+1},y^k)$} 
\label{algeq:test}
    \State \textbf{break}
\Else
    \State $\tau_k \gets \eta\tau_k$
\EndIf
\EndLoop
\State \sa{$t_k\gets \sigma_k/\sigma_0$,\quad $(x_{\rm cum},y_{\rm cum})\gets(x_{\rm cum},y_{\rm cum})+t_k(x^{k+1},y^{k+1})$,\quad $T\gets T+t_k$}
\State $\gamma_{k+1}\gets \gamma_k(1+\mu\tau_k)$,\quad 
\sa{$\tau_{k+1}\gets \tau_k \sqrt{\frac{\gamma_k}{\gamma_{k+1}}\left(1+c_{\rm nm}\cdot\frac{\tau_k}{\tau_{k-1}}\right)}$} \label{algeq:tau}
\EndFor
\State \sa{$(\bar x^K,\bar y^K)\gets (x_{\rm cum},y_{\rm cum})/T$}
\State \Return \sa{$(\bar x^K,\bar y^K)$}
\end{algorithmic}
\end{algorithm}
\subsection{\texttt{APDB-xy}: $x$ update first}
\sa{We first present the \apdbxy{} method 
in \Cref{alg:apdb} --we set $\cK^*=\mathbb{R}^m_+$ when \apdbxy{} is applied to \eqref{eq: qcqp-pro} since $\cK=\reals^m_+$ for this setting. Specifically, 
\apdbxy{} consists of alternating (Gauss-Seidel type) primal and dual updates, and the primal update uses momentum averaging to achieve acceleration.} To determine suitable step sizes by estimating local Lipschitz constants, \Cref{alg:apdb} involves a backtracking step-size search; \sa{indeed, the candidate point $(x^{k+1},y^{k+1})$ is accepted if the condition in \texttt{line}~\ref{algeq:test} of \apdbxy{} holds, where the test function $E_k:\cX\times\cY\to\reals$ is defined as}
\begin{align}
E_k(x,y)
\;\triangleq\;&
\frac{1}{2\alpha_{k+1}}
\left\|
\nabla_x \Phi(x,y)
-
\nabla_x \Phi(x,y^k)
\right\|^2
-
\frac{1}{\sigma_k}
\textbf{D}_{d}(y,y^k)
\nonumber\\
&+
\frac{1}{2\beta_{k+1}}
\left\|
\nabla_x \Phi(x,y^k)
-
\nabla_x \Phi(x^k,y^k)
\right\|^2-
\left(
\frac{1}{\tau_k}
-
\theta_k(\alpha_k+\beta_k)
\right)
\textbf{D}_{p}(x,x^k).
\label{eq: backtracking}
\end{align}
Below we provide convergence results of \texttt{APDB-xy} for \eqref{eq:minimax-pro}, some of \sa{which are established in} \cite[Corollary 4.5]{hamedani2021primal} and \cite[Proof of Corollary 4.5]{hamedani2021primal}.
\begin{lemma}\label{lem:apdbxy-1}
\sa{Suppose that \cref{as:sets,as:Lipschitz-grads,as:saddle-point} hold.} 
\sa{Consider \eqref{eq:minimax-pro} with
$\mu\geq 0$. Initialized from an arbitrary initial point $(x^0,y^0)\in\cX\times\cY$, the dual iterate sequence $\{y^k\}_{k\geq 0}$ generated by \emph{\apdbxy} in \Cref{alg:apdb} for any $c_{\rm nm}\in\{0,1\}$ is} bounded. Specifically, it holds for all $k\geq 0$ that $\textbf{D}_{d}(y^\star,y^k)\le \sigma_0 \Delta(x^\star,y^\star)$ for all $(x^\star,y^\star)\in \cZ^\star$, where 
\begin{align}
\label{eq:Delta}
    \Delta(x,y) \triangleq \frac{1}{\tau_0} \mathbf{D}_p(x,x^0) + \frac{1}{\sigma_0} \mathbf{D}_d(y,y^0),\quad \forall~(x,y)\in \cX \times \cY.
\end{align} 
Moreover, \sa{let $\bar B \triangleq \inf\Big\{\| {y}^\star \| + \sqrt{2\gamma_0 {\mathbf{D}_p(x^\star,x^0) + 2 \mathbf{D}_d(y^\star,y^0) }}:\ (x^\star,y^\star)\in \cZ^\star\Big\}$. Then, it holds that $\sup_{k\geq 0}\norm{y^k}\leq \bar B$, and there exists $y^\star\in\cY^\star$ such that $\norm{y^\star}\leq\bar B$.} 
%
\end{lemma}
\begin{proof}
    See the proof of \cite[Corollary 4.5]{hamedani2021primal}. \sa{According to \cref{lem:product-saddle-set}, $\cZ^\star$ has a product form $\cX^\star\times\cY^\star$ such that $\cY^\star$ is a closed set; therefore, we can conclude the existence of $y^\star\in\cY^\star$ such that $\norm{y^\star}\leq\bar B$.}
\end{proof}
\sa{
\begin{definition}
\label{def:LxxLxy}
Let $\bar\cY\triangleq\{y\in\cY:\ \norm{y}\leq\bar B\}$. Consider $\Phi:\cX\times\cY\to\reals$ defined in \eqref{eq:ncc-coupling}, let $L_{xx}$ 
denote the Lipschitz constant of $\nabla_x \Phi(\cdot,y)$ over the set $\mathcal{X}$ uniformly for all $y\in\bar\cY$, and $L_{xy}$ denotes the Lipschitz constant of $\nabla_x \Phi(x,\cdot)$ over the set $\bar\cY$ uniformly for all $x \in \mathcal{X}$, i.e.,
\begin{align*}
    \norm{\grad_x\Phi(x',y)-\grad_x\Phi(x,y)}\leq L_{xx}\norm{x'-x},\quad\forall~x',x\in\cX,\ y\in\bar\cY,\\
    \norm{\grad_x\Phi(x,y')-\grad_x\Phi(x,y)}\leq L_{xy}\norm{y'-y},\quad\forall~y',y\in\bar\cY,\ x\in\cX.
\end{align*}
\end{definition}
\begin{remark}
    It follows from Assumption~\ref{as:Lipschitz-grads} and Lemma~\ref{lem:apdbxy-1} that $L_{xx}$ and $L_{xy}$ in Definition~\ref{def:LxxLxy} satisfy $L_{xx} = L_f+\bar B L_g$ and $L_{xy}=\norm{A}+B_g$. Specifically, for \eqref{eq: qcqp-pro}, the constants $L_{xx}$ and $L_{xy}$ have the following form:
\begin{align}
L_{xx} = \| Q_0\| + \bar{B}\cdot \left(\sum_{i\in[m]} \| Q_i \|^2\right)^{1/2}, \quad L_{xy} = \| A \| + B_g.
\end{align}
\end{remark}}%
\begin{lemma}
\label{lem:apdb-xy-bounds}
\sa{Suppose that \cref{as:sets,as:Lipschitz-grads,as:saddle-point} hold.} Consider \eqref{eq:minimax-pro}. For \emph{\apdbxy} in \Cref{alg:apdb} \sa{with any $c_{\rm nm}\in\{0,1\}$,} initialized from an arbitrary point $(x^0,y^0)\in\cX\times \cY$, define a quantity intrinsic to the algorithm:
\begin{align*}
\Psi_1 \triangleq \min \left\{ \frac{\sqrt{c_\alpha (1-\delta)}}{L_{xy} \sqrt{\gamma_0}},~\frac{\sqrt{ c_\beta( 1 - (c_\alpha + c_\beta + \delta))}}{L_{xx}} \right\},
\end{align*}
where the constants $L_{xx}$ and $L_{xy}$ are given in Definition~\ref{def:LxxLxy}. Then, $\tau_k\geq \hat\tau_k\triangleq \hat\tau\sqrt{\gamma_0/\gamma_k}$, where $\hat\tau=\eta\Psi_1$, and $\sigma_k=\gamma_k\tau_k$ for all $k\geq 0$; moreover, for any $K\ge 1$ and $(x,y)\in\cX\times\cY$, it holds that
\begin{align}
\frac{{\gamma_{K-1}}\bigl(1-(c_\alpha+c_\beta)\bigr)}{\sigma_0}
\textbf{D}_{p}(x,x^K)
+
\frac{1}{\sigma_0}
\textbf{D}_{d}(y,y^K)
+
T_K\bigl(\mathcal L(\bar x^K,y)-\mathcal L(x,\bar y^K)\bigr)
\le
\Delta(x,y),
\label{eq: 4.16apd}
\end{align}
where $T_K \triangleq \sum_{k=0}^{K-1} t_k$ such that $t_k = \frac{\sigma_k}{\sigma_0}$ for all $k\geq 0$, and
\[
\bar x^K \triangleq \frac{1}{T_K}\sum_{k=0}^{K-1} t_k x^{k+1},\quad
\bar y^K \triangleq \frac{1}{T_K}\sum_{k=0}^{K-1} t_k y^{k+1}.
\]
\end{lemma}
\begin{proof}
     The proof of this result directly follows from the arguments in \cite[Proof of Corollary 4.5]{hamedani2021primal}.
\end{proof}
\begin{remark}\label{rem:APDBxy-strong-convexity}
\rev{In~\cite{hamedani2021primal}, for the strongly convex scenario, we assumed 
$f$ is closed and convex 
with modulus $\mu>0$ and computed 
$\mathrm{prox}_{\tau_k f}(\cdot)$ every iteration; the term $\tfrac{\mu}{2}\|x-x^{k+1}\|^{2}$ in the one-step inequality~\cite[Eq.~(4.13)]{hamedani2021primal} was then supplied directly by~\cite[Lemma~7.1]{hamedani2021primal}. In \texttt{APDB-xy} as stated in Algorithm~\ref{alg:apdb} above, $f$ is instead absorbed into $\Phi$ in~\eqref{eq:ncc-coupling} and
linearized at every iteration via line~\ref{algeq:x-update}; clearly, $\Phi(\cdot,y)$ 
is strongly convex uniformly over $y\in\mathcal{Y}$ --- a property that follows automatically from Assumption~\ref{as:Lipschitz-grads} when $\mu>0$.
Despite this difference in implementation, the same one-step inequality
in \cite[Eq.~(4.13)]{hamedani2021primal} continues to hold verbatim: applying the gradient inequality for the
strongly convex function $\Phi(\cdot, y^{k+1})$ at $x^{k+1}$ in place of the proximal inequality for $f$ reproduces the same 
$\tfrac{\mu}{2}\|x-x^{k+1}\|^{2}$ term with no additional error terms --- since $\nabla_{x}\Phi$
is evaluated at $x^{k+1}$ in both the strong-convexity inequality and in the
momentum term $q_{k+1}$ used to derive line~\ref{algeq:x-update}'s optimality condition. It is
this observation that justifies citing~\cite[Proof of Corollary~4.5]{hamedani2021primal} for~\eqref{eq: 4.16apd} above.}
\end{remark}
\begin{lemma}[Monotone setting]
Under the premise of \cref{lem:apdb-xy-bounds}, consider \emph{\apdbxy} in \Cref{alg:apdb} with $c_{\rm nm}=0$.
For the case $\mu=0$, i.e., $f$ is merely convex, $\gamma_k=\gamma_0$ for $k\geq 0$; hence, $\bar\tau\geq \tau_k\geq\hat\tau$, $\sigma_k=\gamma_0\tau_k$ and $t_k\in \left[\frac{\hat\tau}{\bar\tau}, 1\right]$ for all $k\geq 0$; therefore, $K\hat \tau/\bar\tau\leq T_K\leq K$ for all $K\geq 1$. 
For the case $\mu>0$, i.e., $f$ is strongly convex, we set $\varphi_p(\cdot)=\frac{1}{2}\|\cdot\|^2$. Then, for all $k\geq 1$, $\gamma_k=\Theta(k^2)$, $\tau_k=\Theta(1/k)$, $\sigma_k=\Theta(k)$; hence, $t_k =\Theta(k)$.  Therefore, $T_K =\Theta(K^2)$ for $K\ge 1$; in particular, $T_K \ge \frac{\Gamma^2}{12\mu\sigma_0}K^2$ for $K\ge 2$ with $\Gamma=\mu\hat\tau\sqrt{\gamma_0}$. 
{Moreover, in both cases, within $K\geq 1$ iterations of \emph{\apdbxy},
the total number of backtracking condition evaluations in
\emph{\texttt{line}}~\ref{algeq:test} of \emph{\apdbxy} is bounded by
$K +\log_{1/\eta}({\bar{\tau}}/{\hat\tau})$.}
\label{lem: apdb-sublinear}
\end{lemma}
\begin{proof}
Some of the bounds stated above are new or improve the existing bounds in \cite{hamedani2021primal}, and we discuss how they can be derived using the results in~\cite{hamedani2021primal}. When $c_{\rm nm}=0$, the update rule for $\tau_k$ implies that $\tau_0\leq\bar\tau$ and $\tau_{k+1}\leq \tau_k \sqrt{\frac{\gamma_k}{\gamma_{k+1}}}$ for $k\geq 0$; hence, we immediately get $\tau_k\leq \bar\tau\sqrt{\gamma_0/\gamma_k}$ for $k\geq 0$. Combining it with the lower bound from \cref{lem:apdb-xy-bounds}, i.e., $\tau_k\geq \hat\tau\sqrt{\gamma_0/\gamma_k}$ for $k\geq 0$, we get $\tau_k=\Theta(1/\sqrt{\gamma_k})$ for $k\geq 0$. This is true for both $\mu=0$ and $\mu>0$ cases. Next, before we analyze each case separately, we make an important observation. For each \apdbxy{} iteration $k\geq 0$, let $\tau_k^\circ$ denote the initial value that $\tau_k$ is set to before the backtracking condition in {\texttt{line}}~\ref{algeq:test} is checked for the first time for iteration $k$; therefore, the number of times the test function $E_k$ is evaluated in iteration $k$ 
is equal to $1+\log_{1/\eta}\frac{\tau_{k}^\circ}{\tau_k}$.

For the case $\mu=0$, we know that $\gamma_k=\gamma_0$ for $k\geq 0$; therefore, the result in previous paragraph implies for $k\geq 0$ that $\tau_k\in[\hat\tau,~\bar\tau]$ and $\sigma_k\in[\gamma_0\hat\tau,~\sigma_0]$; hence, $t_k=\sigma_k/\sigma_0\in[\hat\tau/\bar\tau,1]$ for $k\geq 0$ which leads to the desired bounds on $T_K$ for $K\geq 1$. Moreover, when $\mu=0$, since $\tau^\circ_k$ is set to $\tau_{k-1}$ for $k\geq 0$ with $\tau_{-1}=\bar\tau$,
during the first $K\geq 1$ iterations of \apdbxy{}, the total number of times one needs to evaluate the test function is given by 
    \begin{align*}
        \sum_{k=0}^{K-1}\Big(1+\log_{1/\eta}\Big(\frac{\tau_{k-1}}{\tau_k}\Big)\Big)=K+\log_{1/\eta}\Big(\frac{\bar\tau}{\tau_{K-1}}\Big)\leq K+\log_{1/\eta}\Big(\frac{\bar\tau}{\hat\tau}\Big),
    \end{align*}
    which follows from the fact that $\tau_{-1}=\bar\tau$ and $\tau_k\geq \hat\tau$ for all $k\geq 0$ -- we also used the fact that $\tau_{k-1}/\tau_k=1/\eta^r$ for some integer $r\geq 0$.

Next, we discuss the case $\mu>0$. The earlier paper~\cite{hamedani2021primal} established $\gamma_k=\Omega(k^2)$; but, $\gamma_k=\cO(k^2)$ bound is missing in~\cite{hamedani2021primal}. Indeed, from the proof of \cite[Lemma 3.7]{hamedani2021primal}, we know that
\begin{align}
\label{eq:gamma-diff-lower}
    \gamma_{k+1}\geq \gamma_k+\Gamma\sqrt{\gamma_k},\quad\forall~k\geq 0,
\end{align}
where $\Gamma=\mu\hat\tau\sqrt{\gamma_0}$. By induction, one can show that for any $\epsilon>0$, $\gamma_k\geq\frac{\Gamma^2}{(2+\epsilon)^2}k^2$ for all $k\geq \lceil\frac{1}{\epsilon}\rceil$; thus, 
\begin{align}
\label{eq:gamma-lower-bound}
    \gamma_k\geq\frac{\Gamma^2}{9}k^2,\quad\forall~k\geq 0.    
\end{align}
Next, we derive an upper bound on $\gamma_k$ to show that $\gamma_k=\cO(k^2)$.
For $k\geq 0$, using the update rule for $\gamma_{k+1}$ and the fact that $\tau_k\leq \bar\tau\sqrt{\gamma_0/\gamma_k}$, we get
    \begin{align}
    \label{eq:gamma-diff-upper}
        \gamma_{k+1}=\gamma_k(1+\mu\tau_k)\leq \gamma_k+\mu\bar\tau\sqrt{\gamma_0\gamma_k}=\gamma_k+\bar\Gamma\sqrt{\gamma_k},
    \end{align}
    where $\bar \Gamma\triangleq \mu\bar\tau\sqrt{\gamma_0}\geq\Gamma$. Therefore, we also get $\sqrt{\gamma_{k+1}}\leq \sqrt{\gamma_k+\bar\Gamma\sqrt{\gamma_k}}\leq \sqrt{\gamma_k}+\frac{1}{2}\bar\Gamma$ for all $k\geq 0$, which implies that $\sqrt{\gamma_k}\leq \sqrt{\gamma_0}+k\bar\Gamma/2$ for $k\geq 0$, i.e.,
    \begin{align}
    \label{eq:gamma-upper-bound}
        \gamma_{k}\leq \gamma_0 \Big(1+\mu\bar\tau k/2\Big)^2,\quad\forall~k\geq 0.
    \end{align}
Therefore, combining \eqref{eq:gamma-lower-bound} and \eqref{eq:gamma-upper-bound}, we get $\gamma_k=\Theta(k^2)$ for $k\geq 0$, which immediately implies that $\tau_k=\Theta(1/k)$ for $k\geq 1$, i.e., $\frac{2\hat\tau}{2+\mu\bar\tau k}\leq \tau_k\leq \frac{\bar\tau}{\hat\tau}\frac{3}{\mu}\frac{1}{k}$ for $k\geq 0$. Furthermore, note that $\sigma_k=\gamma_k\tau_k$ and $\gamma_{k+1}=\gamma_k(1+\mu\tau_k)$ for $k\geq 0$ imply that $\sigma_k=(\gamma_{k+1}-\gamma_k)/\mu$ for $k\geq 0$; therefore, combining \eqref{eq:gamma-diff-lower} and \eqref{eq:gamma-diff-upper}, we conclude that $\frac{\Gamma}{\mu}\sqrt{\gamma_k}\leq \sigma_k\leq\frac{\bar\Gamma}{\mu}\sqrt{\gamma_k}$, i.e., $\sigma_k=\Theta(k)$ for $k\geq 0$. With this bound on $\{\sigma_k\}$, we can conclude that $t_k=\Theta(k)$ as well; hence, $T_K=\sum_{k=0}^{K-1}t_k=\Theta(K^2)$ for $K\geq 1$. \sa{The explicit bound $T_K \ge \frac{\Gamma^2}{12\mu\sigma_0}K^2$ for $K\ge 2$ follows from $\sigma_k\geq\frac{\Gamma}{\mu}\sqrt{\gamma_k}\geq \frac{\Gamma^2}{3\mu}k$ for $k\geq 0$, where we used \eqref{eq:gamma-lower-bound}; hence, $T_K=\sum_{k=0}^{K-1}\sigma_k/\sigma_0\geq \frac{\Gamma^2}{3\mu\sigma_0}K(K-1)/2\geq \frac{\Gamma^2}{12\mu\sigma_0}K^2$ for $K\ge 2$.} 
Now for the case $\mu>0$, using the same arguments we adopted for $\mu=0$ case, we can show that in $K\geq 1$ iterations of \apdbxy{}, the total number of times one needs to evaluate the test function is given by 
\begin{align*}
     1 + \log_{1/\eta}\left(\frac{\bar{\tau}}{\tau_0}\right) + \sum_{k=1}^{K-1}\Big(1+\log_{1/\eta}\Big(\frac{\tau_{k-1}}{\tau_k}\cdot\sqrt{\frac{\gamma_{k-1}}{\gamma_k}}\Big)\Big) =K+\log_{1/\eta}\Big(\frac{\bar\tau}{\tau_{K-1}} \sqrt{\frac{\gamma_0}{\gamma_{K-1}}} \Big) \le  K + \log_{1/\eta}\left(\frac{\bar{\tau}}{\hat{\tau}}\right), \end{align*}
where the last inequality comes from $\tau_{K-1} \ge \hat{\tau}\sqrt{\gamma_0/\gamma_{K-1}}$. This gives the desired result.
\end{proof}

\begin{remark}\label{remark: nonmonotonic-stepsizes}
\sa{It is also possible to adopt a non-monotone step-size search scheme within \Cref{alg:apdb} by setting the algorithm parameter $c_{\rm nm}=1$, which affects the $\tau$-update rule in \emph{\texttt{line}~\ref{algeq:tau}}, i.e.,}
\begin{align}
\label{eq:nonmonotone-tau}
\tau_{k+1}= 
\tau_k \sqrt{\frac{\gamma_k}{\gamma_{k+1}}\left(1+\frac{\tau_k}{\tau_{k-1}}\right)},\quad \forall~k\geq 0.
\end{align}
This choice allows for $\{\tau_k\}$ sequence to increase; hence, \emph{\apdbxy} can better exploit the local curvature, i.e., it can take larger steps in flatter regions due to smaller local Lipschitz constants, potentially improving the practical performance. 
\end{remark}
The following new result provides bounds on the number of backtracking condition evaluations for the non-monotone case, i.e., $c_{\rm nm}=1$.
\begin{lemma}[Non-monotone setting] Under the premise of \cref{lem:apdb-xy-bounds}, consider \emph{\apdbxy} in \Cref{alg:apdb} with $c_{\rm nm}=1$. For the case $\mu=0$, i.e., $f$ is merely convex, $\gamma_k=\gamma_0$ for $k\geq 0$; hence, $\tau_k\geq\hat\tau$, $\sigma_k=\gamma_0\tau_k$ and $t_k\geq \frac{\hat\tau}{\bar\tau}$ for all $k\geq 0$; therefore, $K\hat \tau/\bar\tau\leq T_K$ for all $K\geq 1$. For the case $\mu>0$, i.e., $f$ is strongly convex, we set $\varphi_p(\cdot)=\frac{1}{2}\|\cdot\|^2$. Then, for all $k\geq 1$, $\gamma_k=\Omega(k^2)$ and 
$\sigma_k=\Omega(k)$; hence, $t_k =\Omega(k)$.  Therefore, $T_K =\Omega(K^2)$ for $K\ge 2$; in particular, $T_K \ge \frac{\Gamma^2}{12\mu\sigma_0}K^2$ for $K\ge 2$ with $\Gamma=\mu\hat\tau\sqrt{\gamma_0}$. \sa{Furthermore, for both cases of $\mu=0$ and $\mu>0$, within $K\geq 1$ iterations of \emph{\apdbxy} with $c_{\rm nm}=1$, the total number of backtracking condition evaluations in \emph{\texttt{line}}~\ref{algeq:test} of \emph{\apdbxy} is bounded by $ (1+ \log_{1/\eta} \varphi )K + \log_{1/\eta}\left({\bar\tau}/{\hat{\tau}}\right)$, where $\varphi \triangleq (1+ \sqrt{5})/2$.}
\label{lem: apdb-sublinear-nm}
\end{lemma}
{Before we proceed with the proof, we highlight that, under the
non-monotone step-size search scheme, each iteration requires \textit{at most}
\((1+\log_{1/\eta}\varphi)\) backtracking condition evaluations on average,
which is approximately \(2\) when \(\eta=0.6\).}
\begin{proof}
Recall that \cref{lem:apdb-xy-bounds} implies that $\tau_k\geq\hat\tau\sqrt{\gamma_0/\gamma_k}$ and $\sigma_k=\gamma_k\tau_k$ for $k\geq 0$.

For the case $\mu=0$, since $\gamma_k=\gamma_0$, we get $\tau_k\geq\hat\tau$ for $k\geq 0$; moreover, we have $\tau_0\leq\bar\tau$. Thus, $t_k=\sigma_k/\sigma_0=\tau_k/\tau_0\geq \hat\tau/\bar\tau$ for $k\geq 0$ and $T_K=\sum_{k=0}^{K-1}t_k\geq K\frac{\hat\tau}{\bar\tau}$.
Recall that the number of times the test function $E_k$ is called in iteration $k$ can be explicitly stated as $1+\log_{1/\eta}\frac{\tau_{k}^\circ}{\tau_k}$. Thus, in $K\geq 1$ iterations of \apdbxy{}, the total number of times one needs to evaluate the test function is at most 
    \begin{align}
        \sum_{k=0}^{K-1}\Big(1+\log_{1/\eta}\Big(\frac{\tau_{k}^\circ}{\tau_k}\Big)\Big)&\le K+\log_{1/\eta}\left(\frac{\bar\tau}{\tau_0}\right)+\sum_{k=1}^{K-1}\log_{1/\eta}\left(\frac{\tau_{k-1}}{\tau_k}\sqrt{1+\frac{\tau_{k-1}}{\tau_{k-2}}}\right) \nonumber \\
        & = K+ \log_{1/\eta}\left(\frac{\bar\tau}{\tau_{K-1}}\right) + \frac{1}{2} \sum_{k=0}^{K-2}\log_{1/\eta}\left(1+\frac{\tau_{k}}{\tau_{k-1}}\right).\nonumber
    \end{align}
    Here, we used the following facts: since $\tau_{-1}=\bar\tau$,  $\tau_k^\circ=\bar\tau$ for $k=0$ and $\tau_k^\circ \leq \tau_{k-1}\sqrt{1+\frac{\tau_{k-1}}{\tau_{k-2}}}$ for all $k\geq 1$, which follows from \eqref{eq:nonmonotone-tau} and the fact that $\gamma_k=\gamma_0$ for $k\geq 0$ -- we also used the fact that $\tau_{k}^\circ/\tau_k=1/\eta^r$ for some integer $r\geq 0$.
    Now, let $a_k \triangleq \tau_k/\tau_{k-1}$ for $k\ge 0$; hence $a_0 = \tau_0/\bar{\tau} \le 1$. By $\tau_{k+1} \le \tau_k \sqrt{1+ \tau_k/\tau_{k-1}}$, we have $a_{k+1} \le \sqrt{1+a_k}$ for $k \ge 0$. Observe that $a_0\le 1 \le \varphi$ and if $a_k \le \varphi$ then $a_{k+1} \le \sqrt{1+a_k} \le \sqrt{1+\varphi} \le \varphi$. It follows that $\tau_k /\tau_{k-1} \le \varphi$ for all $k\ge 0$. Therefore, the total number of times one needs to evaluate the test function is at most
    \begin{align*}
       K + \log_{1/\eta}\left(\frac{\bar\tau}{\tau_{K-1}}\right) + \frac{1}{2}(K-1) \log_{1/\eta}\left(1+\varphi\right) & = K + \log_{1/\eta}\left(\frac{\bar\tau}{\tau_{K-1}}\right) + (K-1) \log_{1/\eta} \varphi\\
       & \le (1+ \log_{1/\eta} \varphi )K + \log_{1/\eta}\left({\bar\tau}/{\hat{\tau}}\right).
    \end{align*}
For the case $\mu>0$, the following bounds, $\gamma_k=\Omega(k^2)$, 
$\sigma_k=\Omega(k)$, $t_k =\Omega(k)$, and $T_K =\Omega(K^2)$, directly follows from the arguments in the proof of \cref{lem: apdb-sublinear}. Moreover, within $K\geq 1$ iterations of \apdbxy{}, the total number of times one needs to evaluate the test function is at most 
    \begin{align}
        \sum_{k=0}^{K-1}\Big(1+\log_{1/\eta}\Big(\frac{\tau_{k}^\circ}{\tau_k}\Big)\Big)&\le K+\log_{1/\eta}\left(\frac{\bar\tau}{\tau_0}\right)+\sum_{k=1}^{K-1}\log_{1/\eta}\left(\frac{\tau_{k-1}}{\tau_k}\sqrt{ \frac{\gamma_{k-1}}{\gamma_k} \left( 1+\frac{\tau_{k-1}}{\tau_{k-2}}\right) }\right)\nonumber \\
        & = K+ \log_{1/\eta}\left(\frac{\bar\tau \sqrt{\gamma_0}}{\tau_{K-1} \sqrt{\gamma_{K-1}}}\right) + \frac{1}{2} \sum_{k=1}^{K-1}\log_{1/\eta}\left(1+\frac{\tau_{k-1}}{\tau_{k-2}}\right). \label{eq:nonmono-backtracking-bound}  \end{align}
    Note that $\frac{\bar\tau \sqrt{\gamma_0}}{\tau_{K-1} \sqrt{\gamma_{K-1}}} \le \frac{\bar \tau}{\hat \tau}$. On the other hand, observe that
    \begin{align*}
        \frac{\tau_{k+1}}{\tau_k} \le \sqrt{\frac{\gamma_k}{\gamma_{k+1}}\left(1+\frac{\tau_k}{\tau_{k-1}}\right)} \le \sqrt{1+\frac{\tau_k}{\tau_{k-1}}} \quad \forall k \ge 0,
    \end{align*}
    and thus $\tau_1 / \tau_0 \le \sqrt{1+ \tau_0/\bar{\tau}} \le \sqrt{2} \le \varphi$. By induction, we know $\tau_{k+1}/ \tau_k \le \varphi$ for all $k\ge 0$. Therefore, the bound in \eqref{eq:nonmono-backtracking-bound} is at most $K + \log_{1/\eta}(\bar \tau / \hat{\tau}) + K \log_{1/\eta}\varphi$.    
\end{proof}



\subsection{\texttt{APDB-yx}: $y$ update first}
\sa{\apdbxy{} method uses the momentum term only in the $x$-update; therefore, it is natural to ask how the behavior is affected if we switch the order of $x$ and $y$ updates, and adopt momentum term in the $y$-update.} In fact, we have found that updating $y$ first can improve the algorithmic performance in practice; however, to establish the associated convergence guarantee, one theoretical obstacle is that the dual iterate \sa{sequence $\{y^k\}_{k\ge 0}$ may be unbounded --unlike \apdbxy. To resolve this issue, we assume that a dual bound is known. Such a bound is attainable whenever Slater's condition holds.} 
\begin{assumption}
There exists 
$\tilde{x} \in \rm relint(\mathcal{X})$
such that $A \tilde x = b$, and $-g(\tilde x) \in {\rm int}(\mathcal{K})$.
\label{assu: slater}
\end{assumption}
Under Slater's condition, we know that strong duality holds and the dual optimal value is attained \cite[Section 5.9]{boyd2004convex}, i.e., a dual optimal solution $ y^\star \in \cY $ exists\footnote{For \eqref{eq:conic-problem}, there is a refined (weaker) Slater's condition to ensure strong duality and existence of a dual optimal solution $y^\star$. This weaker version holds if there exists $\tilde{x} \in {\rm relint}(\mathcal{X})$, the relative interior of the set $\mathcal{X}$, such that $A \tilde x = b$, $g_i(\tilde{x}) \le 0$ for all $i\in [m]$ such that $g_i$ is affine, and $g_i(\tilde x)<0$ for the remaining indices $i\in [m]$, see, e.g., \cite[Section 5.2.3]{boyd2004convex}. However, to obtain a computable bound on $\| y^\star\|$, one can adopt the stronger version stated in \Cref{assu: slater}}. 
\sa{In the following result, given a Slater point $\tilde x$, we discuss how a dual bound $\bar B>0$ can be computed such that $\norm{y^\star}\leq \bar B$ for some $y^\star\in\cY^\star$.
Indeed, Lemma~\ref{lem:dual_sol_bounded} is an extension of \cite[Lemma 1.1]{nedic201010}.}
\begin{lemma}\label{lem:dual_sol_bounded}
\sa{Suppose that \cref{as:sets,as:Lipschitz-grads} hold.} Consider the problem in \eqref{eq:conic-problem}: $f^\star=\min_{x\in \mathcal{X}}\{f(x): \;  Ax=b,\; g(x)\in -\mathcal{K}\}$. The dual function \sa{$q:\reals^p\times\reals^m\to\reals\cup\{-\infty\}$} is defined as
\begin{align}
    \label{eq:dual-function}
\sa{q}(v,\lambda) \triangleq
\begin{cases}
\inf\limits_{x\in \mathcal{X}}
\big\{
f(x)+\langle v,Ax-b\rangle + \langle \lambda,g(x)\rangle
\big\},
& \textnormal{if }\lambda\in \mathcal{K}^*,\\[0.5em]
-\infty, & \textnormal{otherwise},
\end{cases}
\end{align}
\sa{where $\dom q\triangleq \{(v,\lambda) \in \cY:\ {q}(v,\lambda) > - \infty\}$.} 
If \cref{assu: slater} holds, then
for any dual optimal solution $(v^\star,\lambda^\star)$, it holds that
\begin{equation}\label{eq:app-lambdabound}
\|\lambda^\star\|
\le \frac{f(\tilde x)-{q}(v,\lambda)}{r^*},\quad \forall~(v,\lambda) \in \dom q,
\end{equation}
where
\[
0 < r^* \triangleq
\min\Big\{
-\langle w,g(\tilde x)\rangle:\ \|w\|=1,\ w\in \mathcal{K}^*
\Big\}.
\]
\sa{Moreover, if we further assume that $\mathcal{X}^\star \cap {\rm int}(\mathcal{X})\neq\emptyset$},
then
\[
\|v^\star\|
\le
\frac{1}{\sigma_{\min}(A)}
\left(
\max_{x\in\mathcal X}\|\nabla f(x)\|
+
C_g\,\frac{f(\tilde x)-{q}(v,\lambda)}{r^*}
\right),\quad \forall~(v,\lambda) \in \dom q;
\]
hence, 
$\cY^\star$ is bounded, \sa{i.e., there exists $\bar B>0$ such that $y^\star\in\cY^\star$ implies that $\norm{y^\star}\leq \bar B$.}
\end{lemma}

\begin{proof}
\sa{Given arbitrary $(v^\star,\lambda^\star)\in\cY^\star$, it holds for any $(v,\lambda)\in \dom q$ that}
\begin{align*}
{q}(v,\lambda)
\le
{q}(v^\star,\lambda^\star)
& =
\inf_{x\in\mathcal X}
\big\{
f(x)+\langle v^\star,Ax-b\rangle+\langle \lambda^\star,g(x)\rangle
\big\}\\
& \le f(\tilde{x})+\langle v^\star,A\tilde{x}-b\rangle+\langle \lambda^\star,g(\tilde{x})\rangle;
\end{align*}
hence, it follows that $-\langle \lambda^\star,g(\tilde x)\rangle
\le
f(\tilde x)-{q}(v,\lambda)$. Since $-g(\tilde x)\in \operatorname{int}(\mathcal K)$, there exists $r>0$ such that $-g(\tilde x)-ru\in \mathcal K$ for all $u \in \mathbb{R}^m$ with $\|u\| \le 1$. Note that if $\lambda^\star = 0$, the bound in \eqref{eq:app-lambdabound} holds trivially; therefore, it is sufficient to consider $\lambda^\star \neq 0$ only. For $u=\lambda^\star/\|\lambda^\star\|$, using $\lambda^\star\in\mathcal K^*$, we have
\[
0
\le
\left\langle \lambda^\star,-g(\tilde x)-r{\lambda^\star}/{\|\lambda^\star\|}\right\rangle
=
-\langle \lambda^\star,g(\tilde x)\rangle-r\|\lambda^\star\|;
\]
thus, we get
\[
r\|\lambda^\star\|
\le
-\langle \lambda^\star,g(\tilde x)\rangle
\le
f(\tilde x)-{q}(v,\lambda),
\]
which implies that $\|\lambda^\star\|
\le ({f(\tilde x)-{q}(v,\lambda)})/{r}$.
According to \cite[Lemma 6.1]{aybat2019distributed}, the largest radius $r^*>0$ such that \sa{the closed ball $\bar{\cB}_{r^*}(-g(\tilde x))\subset\cK$} satisfies $r^* = \min\{
-\langle w,g(\tilde x)\rangle:\ \|w\|=1,\ w\in \mathcal K^*
\}$; therefore, $\|\lambda^\star\|
\le ({f(\tilde x)-{q}(v,\lambda)})/{r^*}$. 

Suppose there exists a primal optimal solution $x^\star$ such that $x^\star \in {\rm int}(\mathcal{X})$. 
Let $y^\star=(v^\star,\lambda^\star)$ be an arbitrary dual optimal solution. Slater's condition holding implies that strong duality holds, which implies that $(x^\star,y^\star)$ is a KKT point; hence, $0=\nabla f(x^\star)+A^\top v^\star+ \sa{\nabla g(x^\star)^\top \lambda^\star}$.
\sa{Therefore, using $C_g$-Lipschitz continuity of $g$ over $\cX$, we get}
\[
\|A^\top v^\star\|
\le
\|\nabla f(x^\star)\|+\|\nabla g(x^\star)\|\,\|\lambda^\star\|
\le
\max_{x\in\mathcal X}\|\nabla f(x)\|+C_g\|\lambda^\star\|.
\] 
Hence, $A$ has full row rank implies that $\sigma_{\rm min}(A)>0$, and we get
\[
\|v^\star\|
\le
\frac{1}{\sigma_{\min}(A)}\|A^\top v^\star\|
\le
\frac{1}{\sigma_{\min}(A)}
\left(
\max_{x\in\mathcal X}\|\nabla f(x)\|+C_g\|\lambda^\star\|
\right).
\]
\sa{Using the bound on $\|\lambda^\star\|$ yields}
\[
\|v^\star\|
\le
\frac{1}{\sigma_{\min}(A)}
\left(
\max_{x\in\mathcal X}\|\nabla f(x)\|
+
C_g\,\frac{f(\tilde x)-{q}(v,\lambda)}{{r^*}}
\right).
\]
Combining the bounds on $\|v^\star\|$ and $\|\lambda^\star\|$ yields the desired bound on $\|(v^\star,\lambda^\star)\|$.
\end{proof}
\begin{remark}
\label{rem:dual-bound}
In Lemma \ref{lem:dual_sol_bounded}, the bounds on both $\norm{\lambda^\star}$ and $\norm{v^\star}$ depend on $r^* = 
\min_w\Big\{
\langle w,-g(\tilde x)\rangle:\ \|w\|=1,\ w\in \mathcal{K}^*
\Big\}$ with $-g(\tilde{x}) \in {\rm int}(\mathcal{K})$. Although computing $r^*$ requires solving a nonconvex optimization problem, its optimal value admits closed-form expressions for several important cases:
\begin{itemize}
    \item if $\mathcal{K}=\mathbb{R}^m_+$, then $r^*=\min_{i\in[m]}\{-g_i(\tilde x)\}$;
    \item if $\mathcal{K}=\mathbb{S}^m_+$, then $r^*=\lambda_{\min}\big(-g(\tilde x)\big)$,
    \jw{where $\|\cdot\|$ 
    is the Frobenius norm for this case};
    \item if $\mathcal{K}=\mathcal{Q}^{m}
    \triangleq \{(t,u)\in\mathbb{R}\times\mathbb{R}^{m-1}:\ t\ge \|u\|\}$,
    \sa{given $g(\tilde x)=[g_i(\tilde x)]_{i\in[m]}\in -{\rm int}(\cK)$, let $\bar g(\tilde x)\triangleq [g_i(\tilde x)]_{i=2}^m\in\reals^{m-1}$, then $r^*=-(g_1(\tilde x)+\|\bar g(\tilde x)\|)/\sqrt{2}$.}
\end{itemize}
\end{remark}

\begin{assumption}
There exists \sa{$\bar B>0$} such that ${\cY^\star}\cap\{y:\ \norm{y}\leq \sa{\bar B}\}\neq\emptyset$.
\label{as:dual-bound}
\end{assumption}
\begin{definition} \label{def:calB}
\sa{Under \cref{as:dual-bound}, given some arbitrary 
$B>\bar B$, let} 
\begin{align} \label{eq: defcalB}
\mathcal{B} \triangleq \{ y \in \mathbb{R}^{p} \times \mathbb{R}^m: \| y\| \le B \}.
\end{align}
\end{definition}
\begin{remark}
    \sa{For the sake of notational simplicity throughout the paper, we defined $\cB$ as in \eqref{eq: defcalB}. Algorithmically this is not preferable as $\norm{y}=(\norm{v}^2+\norm{\lambda}^2)^{1/2}$ couples $v$- and $\lambda$-updates. Instead, one can alternatively set $\mathcal{B} =\{ (v,\lambda) \in \mathbb{R}^{p} \times \mathbb{R}^m: \| v\| \le B_v,\ \norm{\lambda}\leq B_\lambda \}$ for some $B_v,B_\lambda>0$ such that $\cY^\star\subset{\rm int}(\cB)$. With this choice of $\cB$, $y$-update becomes separable in $v$ and $\lambda$ in case $\mathbf{D}_d$ is separable.}  
\end{remark}
\sa{Assumption~\ref{assu: slater} and the choice of $B$ based on Lemma~\ref{lem:dual_sol_bounded} and 
Remark~\ref{rem:dual-bound} provide a sufficient condition for Assumption~\ref{as:dual-bound} --indeed, under \Cref{assu: slater}, one can set $\cB$ such that ${\cY^\star}\subset {\rm int}(\cB)$ holds. Indeed, \Cref{as:dual-bound} is weaker than \Cref{assu: slater} as it only implies that ${\cY^\star}\cap {\rm int}(\cB)\neq\emptyset$.} 

\sa{We now consider a slightly modified version of \eqref{eq:minimax-pro} obtained by further restricting 
the dual domain in \eqref{eq:minimax-pro}} to $\cY \cap \mathcal{B}$,~see \cite{aybat2019distributed,hamedani2021primal}, i.e.,
\begin{align} \label{eq:minmax_rest} \tag{R-min-max}
\min_{x \in \mathbb{R}^n } \max_{y \in 
\sa{\cB}} {\mathcal{L}}(x,y),
\end{align}
where $\cL$ is defined in \eqref{eq:minimax-pro}. Thus, \eqref{eq:minmax_rest} is equivalent to $\min_x\max_y \iota_{\mathcal{X}}(x) + 
\Phi(x,y) - \iota_{\hat{\mathcal{Y}}}(y)$ where $\hat{\mathcal{Y}} \triangleq \cY \cap \mathcal{B}$ \sa{denotes the compact 
dual domain and $\Phi$ is defined in~\eqref{eq:ncc-coupling}.} 
\begin{definition}\label{def:hatzstar}
    \sa{Let the compact set $\hat{\mathcal{Z}} \triangleq \mathcal{X} \times \hat{\mathcal{Y}}$ denote the domain of \eqref{eq:minmax_rest} with $\hat{\mathcal{Y}} \triangleq \cY \cap \mathcal{B}$, and ${ \hat{\mathcal{Z}}^\star}$ denote the set of 
saddle points of \eqref{eq:minmax_rest}.} 
\end{definition}
\sa{Now that the domain $\hat\cZ$ is compact, we can consider another variant of \apdbxy. We call it \apdbyx{} as $x$- and $y$-updates are switched compared to \apdbxy. Indeed, suppose that the update rules for $\alpha_{k+1}, \beta_{k+1}$ in \texttt{line} \ref{algeq:alphabeta}, and for $y^{k+1}, x^{k+1}$ in \texttt{line}~\ref{algeq:s-update}-\ref{algeq:y-update} in \Cref{alg:apdb} are replaced by}
\begin{align}
\sigma_k &\gets \gamma_k\tau_k,\quad \theta_k \gets \sigma_{k-1}/\sigma_k,\quad \alpha_{k+1}\gets c_\alpha/\sigma_k,\quad \beta_{k+1}\gets 0,\nonumber\\
s^k &\gets (1+\theta_k)\nabla_y \Phi(x^k,y^k)
- \theta_k \nabla_y \Phi(x^{k-1},y^{k-1}),\nonumber \\
y^{k+1} &\gets 
\mathop{\arg\min}_{y\in \hat{\cY}}
\left\{
- \langle s^k,y\rangle
+ \frac{1}{\sigma_k} \textbf{D}_d(y,y^k)
\right\}, \label{eq: apdb-yx-y} \\
x^{k+1} &\gets 
\mathop{\arg\min}_{x\in\mathcal X}
\left\{ \rev{\frac{\mu}{2}\norm{x-x^k}^2}+\langle \nabla_x \Phi(x^k, y^{k+1}),x\rangle
+ \frac{1}{\tau_k} \textbf{D}_p(x,x^k)
\right\}; \nonumber 
\end{align}
in addition,  
\sa{the test function $E_k$ defined in \eqref{eq: backtracking} for \apdbxy{} is replaced by}
\begin{align*}
E_k(x,y)\triangleq\;&
\Phi(x,y)-\Phi(x^k,y)-\left\langle \nabla_x \Phi(x^k,y),\,x-x^k\right\rangle-\rev{\frac{\mu}{2}\norm{x-x^k}^2}-\frac{1}{\tau_k}\textbf{D}_{p}(x,x^k)\\
&+\frac{1}{2\alpha_{k+1}}\left\|\nabla_y\Phi(x,y)-\nabla_y\Phi(x^k,y)\right\|^2-\left(\frac{1}{\sigma_k}-\theta_k\alpha_k\right)\textbf{D}_{d}(y,y^k).
\end{align*}
\begin{remark}
\rev{The update rules in~\eqref{eq: apdb-yx-y} follow from implementing \cite[Algorithm 2.3]{hamedani2021primal} on $$\min_x\max_y \tilde f(x)+\tilde\Phi(x,y)-\tilde h(y)$$ for the smooth coupling function $\tilde\Phi(x,y)=\Phi(x,y)-\frac{\mu}{2}\norm{x}^2$ and for closed convex functions $\tilde f(x)=\iota_{\mathcal{X}}(x)+\frac{\mu}{2}\norm{x}^2$ and $\tilde h(y)=\iota_{\hat{\mathcal{Y}}}(y)$, i.e., we add and subtract $\frac{\mu}{2}\norm{x}^2$ since the results for \cite[Algorithm 2.3]{hamedani2021primal} assume that $\tilde f$ is $\mu$-convex.}
\end{remark}
\begin{definition}
\label{def:hatLxxLyx}
Consider $\Phi:\cX\times\cY\to\reals$ defined in \eqref{eq:ncc-coupling}, let $\hat{L}_{xx}$ and $\hat{L}_{yx}$ denote the Lipschitz constants of $\nabla_x \Phi(\cdot,y)$ and $\nabla_y \Phi(\cdot,y)$, respectively, over the set $\mathcal{X}$ uniformly for all $y\in\hat{\cY}$, i.e.,
\begin{align*}
    \norm{\grad_x\Phi(x',y)-\grad_x\Phi(x,y)}\leq \hat{L}_{xx}\norm{x'-x},\quad\forall~x',x\in\cX,\ y\in\hat\cY,\\
    \norm{\grad_y\Phi(x',y)-\grad_y\Phi(x,y)}\leq \hat{L}_{yx}\norm{x'-x},\quad\forall~x',x\in\cX,\ y\in\hat\cY.
\end{align*}
\end{definition}
\begin{remark}
    It follows from Assumption~\ref{as:Lipschitz-grads} and the definition of $\hat\cY$ that $\hat{L}_{xx}$ and $\hat{L}_{yx}$ in Definition~\ref{def:hatLxxLyx} satisfy $\hat{L}_{xx} = L_f+B\cdot L_g$ and $\hat{L}_{yx}=\norm{A}+C_g$. Specifically, for \eqref{eq: qcqp-pro}, the constants $\hat{L}_{xx}$ and $\hat{L}_{yx}$ have the following form:
\begin{align}
\hat{L}_{xx}= \| Q_0\| + B \cdot \left(\sum_{i=1}^m \| Q_i \|^2\right)^{1/2}, \quad \hat{L}_{yx} = \| A \| + C_g.
\end{align}
\end{remark}
Below we provide convergence results of \texttt{APDB-yx} for \eqref{eq:minmax_rest} established in~\cite{hamedani2021primal}.
\begin{lemma}\label{thm:mainII}
\sa{Suppose \cref{as:sets,as:Lipschitz-grads,as:saddle-point,as:dual-bound} hold.} 
Consider \eqref{eq:minmax_rest}. For \emph{\texttt{APDB-yx}} {with any $c_{\rm nm} \in \{0,1\}$}, initialized from an arbitrary point $(x^0,y^0)\in \mathcal{X}\times \hat{\cY}$, let $\delta\in[0,1), c_\alpha>0$, $c_\beta = 0$, and $c_\alpha + \delta \in (0,1]$, and define
\begin{equation}\label{eq:psi12}
\Psi_2 \triangleq \frac{c_\alpha}{2\gamma_0}\cdot\frac{\hat{L}_{xx}}{\hat{L}_{yx}^2}\,\zeta,
\qquad
\zeta \triangleq -1+\sqrt{1+\frac{4(1-\delta)\gamma_0}{c_\alpha }\cdot\frac{ \hat{L}_{yx}^2}{\hat{L}_{xx}^2}}.
\end{equation}
Then, $\tau_k\geq \hat\tau_k\triangleq \hat\tau\sqrt{\gamma_0/\gamma_k}$, where $\hat\tau=\eta\Psi_2$, and $\sigma_k=\gamma_k\tau_k$ for all $k\geq 0$; moreover, for any $K\ge 1$ and $(x,y)\in\cX\times \hat{\cY}$, it holds that
\begin{align}
\sa{\frac{\gamma_K}{\sigma_0}}
\textbf{D}_{p}(x,x^K)+
\frac{1-c_\alpha}{\sigma_0}
\textbf{D}_{d}(y,y^K)+
T_K\bigl({\cL}(\bar x^K,y)-{\cL}(x,\bar y^K)\bigr)\le\Delta(x,y),
\label{eq: 3.3apd}
\end{align}
where \sa{$\Delta(x,y)$, $T_K$, and $(\bar x^K,\bar y^K)$ are defined as in \Cref{lem:apdbxy-1} and \cref{lem:apdb-xy-bounds}.}
\end{lemma}
\begin{proof}
    \sa{These results are established in \cite[Theorem 2.7]{hamedani2021primal} for $c_{\rm nm}=0$. The same proof arguments continue to hold for $c_{\rm nm}=1$ verbatim.}
\end{proof}
\begin{lemma}[Monotone setting] Under the premise of \cref{thm:mainII}, consider \emph{\apdbyx} with $c_{\rm nm}=0$. \sa{The results stated in \cref{lem: apdb-sublinear} for \emph{\apdbxy} continue to hold for \emph{\apdbyx} with $\hat\tau$ redefined as in \cref{thm:mainII}, i.e., $\hat\tau=\eta\Psi_2$ with $\Psi_2$ given in \eqref{eq:psi12}.}
\label{lem: apdb-yx-monotone}
\end{lemma}

\begin{lemma}[Non-monotone setting] Under the premise of \cref{thm:mainII}, consider \emph{\apdbyx} with $c_{\rm nm}=1$. \sa{The results stated in \cref{lem: apdb-sublinear-nm} for \emph{\apdbxy} continue to hold for \emph{\apdbyx} with $\hat\tau$ redefined as in \cref{thm:mainII}, i.e., $\hat\tau=\eta\Psi_2$ with $\Psi_2$ given in \eqref{eq:psi12}.} 
\label{lem: apdb-yx-sublinear-nm}
\end{lemma}
\subsection{Suboptimality and Infeasibility of APDB sequence}
\sa{Recall that $\hat \cZ$ denotes the restricted problem domain, i.e., $\hat\cZ=\cX\times\hat\cY$.} We first characterize the relationship between $\cZ^\star$ and $\hat{\cZ}^\star$, the sets of saddle points of \eqref{eq:minimax-pro} and \eqref{eq:minmax_rest}, respectively. 
{In the rest, we adopt the following notation}: $\Pi_x(z')=x'$ and $\Pi_y(z')=y'$ for all $z'=(x',y')\in\cX\times \cY$.\looseness=-10
\begin{lemma}\label{lem:zstar-hatzstar}
Under \cref{as:sets,as:Lipschitz-grads,as:saddle-point,as:dual-bound}, consider the two minimax optimization problems in \eqref{eq:minmax_rest} and \eqref{eq:minimax-pro} \sa{with the coupling function $\Phi$ as defined in \eqref{eq:ncc-coupling}.}
Then,
\[
    \emptyset\neq \hat{\cZ}\cap\cZ^\star=\hat{\cZ}^\star .
\]
Let $\hat{\cX}^\star\triangleq \Pi_x(\hat\cZ^\star)$ and $\hat{\cY}^\star\triangleq \Pi_y(\hat\cZ^\star)$, then
\[
    \hat{\cX}^\star=\cX^\star,
    \qquad
    \hat{\cY}^\star=\cY^\star\cap\hat{\cY}.
\]
Moreover, \(\cL(z^\star)=f^\star\) for any \(z^\star\in\hat\cZ^\star\), where $f^\star$ denotes the optimal value of \eqref{eq:conic-problem}.
\end{lemma}

\begin{proof}
By \cref{as:dual-bound} and \cref{def:calB}, there exists
\(y^\star\in\cY^\star\cap {\rm int}(\cB)\). \sa{For the rest of the proof, we fix such $y^\star$.} Together with
\cref{as:saddle-point,lem:product-saddle-set}, this implies that, for any
\(x^\star\in\cX^\star\), $(x^\star,y^\star)\in \cZ^\star\cap\hat\cZ$; hence, \(\hat\cZ\cap\cZ^\star\neq\emptyset\).

We first show \(\hat\cZ\cap\cZ^\star\subseteq\hat\cZ^\star\). Indeed, for
any \((x^\star,y^\star)\in\hat\cZ\cap\cZ^\star\), the saddle point inequality for
\eqref{eq:minimax-pro} gives
\[
    \cL(x^\star,y)\le \cL(x^\star,y^\star)\le \cL(x,y^\star),
    \qquad
    \sa{\forall 
    x\in\reals^n,\quad \forall y\in\reals^p\times\reals^m.}
\]
Since 
$y^\star\in\cB$ and the same inequality
holds for all 
$y\in\cB$, we have
\((x^\star,y^\star)\in\hat\cZ^\star\), \sa{i.e., $\hat\cZ\cap\cZ^\star\subseteq\hat\cZ^\star$.} 

Note that \(\emptyset\neq \hat\cZ\cap\cZ^\star\subseteq\hat\cZ^\star\) implies that the restricted problem \eqref{eq:minmax_rest}
has the same saddle value \(f^\star\) as the original problem \eqref{eq:minimax-pro}. Indeed, \sa{since the saddle point value of a convex-concave minimax problem is unique, the fact that $\cL(z^\star)=f^\star$ for any $z^\star\in\hat{\cZ}^\star \cap \cZ^\star$ implies that $\cL(z^\star)=f^\star$ for all $z^\star \in \hat{\cZ}^\star$.}

Next, we show the reverse inclusion. Fix any
\((\hat x,\hat y)\in\hat\cZ^\star\). Since \eqref{eq:minmax_rest} has saddle
value \(f^\star\),
\[
    \max_{y\in
    \sa{\cB}}\cL(\hat x,y)=f^\star,
    \qquad
    \inf_{x}
    \cL(x,\hat y)=f^\star;
\]
hence, \(y^\star\in\sa{\cB}\) implies that $\cL(\hat x,y^\star)\le f^\star$.
On the other hand, since \(y^\star\in\cY^\star\), it holds that
\[
    \cL(x,y^\star)\ge f^\star,\qquad \forall x\in\sa{\reals^n}.
\]
Therefore, \(\cL(\hat x,y^\star)=f^\star\). Writing
\(y^\star=(v^\star,\lambda^\star)\), this means that 
$(v^\star,\lambda^\star)$ maximizes
the affine function
\[
    (v,\lambda)\mapsto
    v^\top(A\hat x-b)+\langle \lambda,g(\hat x)\rangle
\]
over \sa{$(\reals^p\times\cK^*)\cap\cB$, i.e., over \(\hat\cY=\cY\cap\cB\).} We now show that \(\hat x\) is primal feasible. For any 
\(\epsilon>0\), let $y(\epsilon)\triangleq (v(\epsilon), \lambda(\epsilon))$ such that $v(\epsilon)=v^\star+\epsilon (A\hat x-b)$ and $\lambda(\epsilon)=\lambda^\star+\epsilon \Pi_{\cK^*}(g(\hat{x}))\in\cK^*$; hence, $y(\epsilon)\in\cY$ for any $\epsilon>0$. Moreover, since \(y^\star\in{\rm int}(\cB)\), we have \(y(\epsilon)\in\hat\cY\) for all
sufficiently small \(\epsilon>0\). Fix such $\epsilon>0$ so that \(y(\epsilon)\in\hat\cY\).

Next, using the observation that 
\sa{$y^\star=(v^\star,\lambda^\star)$} maximizes $v^\top(A\hat x-b)+\langle \lambda,g(\hat x)\rangle$ over $\hat\cY$,  we get $v(\epsilon)^\top(A\hat x-b)+\langle \lambda(\epsilon),g(\hat x)\rangle \le
v^\star(A\hat x-b)+\langle \lambda^\star,g(\hat x)\rangle$, i.e.,
\begin{align}
    \label{eq:feasibility-ineq}
    \epsilon\|A\hat x-b\|^2+
    \epsilon
    \left\langle \Pi_{\cK^\ast}(g(\hat x)),g(\hat x)\right\rangle
    \le 0.
\end{align}
By Moreau decomposition,
\[
    g(\hat x)=\Pi_{-\cK}(g(\hat x))+\Pi_{\cK^\ast}(g(\hat x)),
    \qquad
    \left\langle
    \Pi_{-\cK}(g(\hat x)),\ \Pi_{\cK^\ast}(g(\hat x))
    \right\rangle=0;
\]
hence,
$\left\langle \Pi_{\cK^\ast}(g(\hat x)),\ g(\hat x)\right\rangle=
    \|\Pi_{\cK^\ast}(g(\hat x))\|^2$.
Therefore, \eqref{eq:feasibility-ineq} implies that $\epsilon\|A\hat x-b\|^2+
    \epsilon\|\Pi_{\cK^\ast}(g(\hat x))\|^2
    \le 0$,
which yields \(A\hat x=b\) and \({\rm dist}\big(g(\hat x),-\cK\big)=0\), i.e., \(g(\hat x)\in-\cK\). Furthermore, \sa{since \((\hat x,\hat y)\in\hat\cZ^\star\), we have $\hat x\in\cX$.} Thus, \(\hat x\) is primal feasible, \sa{i.e., $\hat x\in\cX_{\rm feas}$.} Moreover, since 
\sa{\(0\in\cB\),} we further have
\[
    f^\star=\max_{y\in\sa{\cB}}
    \cL(\hat x,y)
    \ge \cL(\hat x,0)=f(\hat x);
\]
\sa{therefore, $\hat x\in\cX_{\rm feas}$ implies that \(f(\hat x)=f^\star\); and we conclude that}
\(\hat x\in\cX^\star\).

It remains to show that \(\hat y\in\cY^\star\). Since $(\hat x,\hat y)\in\hat\cZ^\star$, we have
$\cL(\hat x,\hat y)\le \cL(x,\hat y)$ for all $x\in\reals^n$.
Thus,
\[
    q(\hat y)\triangleq \inf_{x\in\reals^n}\cL(x,\hat y)
    =
    \cL(\hat x,\hat y)=f^\star,
\]
which implies \(\hat y\in\cY^\star\). Since \(\hat y\in\hat\cY\), we have
\(\hat y\in\cY^\star\cap\hat\cY\). 
By the product structure
\(\cZ^\star=\cX^\star\times\cY^\star\), we obtain
\((\hat x,\hat y)\in\cZ^\star\); moreover, using
\((\hat x,\hat y)\in\sa{\cX\times\hat\cY}=\hat\cZ\), we get $(\hat x,\hat y)\in\hat\cZ\cap\cZ^\star$, i.e., 
\[
    \hat\cZ^\star\subseteq \hat\cZ\cap\cZ^\star.
\]
Combining the two inclusions yields
    $\hat\cZ^\star=\hat\cZ\cap\cZ^\star$.
Finally, using \(\cZ^\star=\cX^\star\times\cY^\star\), we have
\[
    \hat\cZ^\star
    =
    \hat\cZ\cap\cZ^\star
    =
    (\cX\times\hat\cY)\cap(\cX^\star\times\cY^\star)
    =
    \cX^\star\times(\cY^\star\cap\hat\cY).
\]
Thus, \(\hat\cX^\star=\cX^\star\) and
\(\hat\cY^\star=\cY^\star\cap\hat\cY\), which completes the proof. 
\end{proof}

\begin{table}[h!]
\centering
\begin{threeparttable}
\caption{Unified notation for the two APDB variants applied\tnote{1} to \eqref{eq:minmax_rest}.}
\label{tab:xy-yx-notation}
\renewcommand{\arraystretch}{1.15}
\setlength{\tabcolsep}{6pt}
\begin{tabular}{lccccc}
\toprule
Method
& \makecell{Choice\tnote{2}\\ of \(B\)}
& \makecell{Restriction set\\ \(\cB=\{y:~\norm{y}\leq B\}\)}
& \makecell{Dual\\ domain ($\hat \cY$)}
& \makecell{Primal-dual\\ domain ($\hat \cZ$)}
& \makecell{Set of saddle\\ points ($\hat\cZ^\star$)}  \\
\midrule
\texttt{APDB-xy}
& \(B=\infty\)
& $\cB=\reals^p\times\reals^m$
& \(\hat{\cY}=\cY\)
& \(\hat{\cZ}=\cX\times\cY\ (\hat{\cZ}=\cZ)\)
& \(\hat\cZ^\star=\cZ^\star\) \\
\texttt{APDB-yx}
& \(B>\bar B\)
& $\cB\subset\reals^p\times\reals^m$
& \(\hat{\cY}=\cY\cap\cB\)
& \(\hat{\cZ}=\cX\times\hat{\cY}\ (\hat{\cZ}\subset\cZ)\)
& \(\hat{\cZ}^\star \jw{=}\cZ^\star\cap\hat\cZ\) \\
\bottomrule
\end{tabular}
\begin{tablenotes}
\footnotesize
\item[1] Both \apdbxy{} and \apdbyx{} are applied to \eqref{eq:minmax_rest} with $\cL$ defined as in \eqref{eq:minimax-pro}.
\item[2] For \apdbxy{}, we set $B=+\infty$ so that \eqref{eq:minmax_rest} becomes equivalent to \eqref{eq:minimax-pro}, while for \apdbyx, $B>\bar B$ is chosen for some $\bar B$ satisfying \cref{as:dual-bound}.
\end{tablenotes}
\end{threeparttable}
\end{table}

\sa{To unify the analyses of both \apdbxy{} and \apdbyx{} (as well as their restarted versions), we consider the formulation in~\eqref{eq:minmax_rest} with $\cB\subset\reals^p\times\reals^m$ defined as in \eqref{eq: defcalB} for some $B\in\reals_+\cup\{+\infty\}$. Recall that \apdbxy{} is stated for the formulation in \eqref{eq:minimax-pro}, which can be considered as a special case of \eqref{eq:minmax_rest}. Indeed, for \apdbxy{}, setting $B = +\infty$ implies that $\cB = \mathbb{R}^p \times \mathbb{R}^m$; hence, \eqref{eq:minmax_rest} reduces to \eqref{eq:minimax-pro}, i.e., the dual domain of \eqref{eq:minmax_rest} becomes $\hat{\cY} = \cY \cap \cB = \cY$ implying that $\hat{\cZ} = \cZ$ which recovers \eqref{eq:minimax-pro}. Therefore, in the remaining part of this paper, for both \apdbxy{} and \apdbyx{} we consider the formulation given in \eqref{eq:minmax_rest} using the convention stated in \Cref{tab:xy-yx-notation}.}


The following result is an extension of \cite[Corollary 4.2]{hamedani2021primal}.
\begin{lemma}\label{lem:x-solve-ncp}
\sa{For \emph{\apdbxy}, suppose \cref{as:sets,as:Lipschitz-grads,as:saddle-point} hold, and for \emph{\apdbyx} in addition to these, suppose \cref{as:dual-bound} holds as well.} The primal sequence $\{\bar x^K\}_{K\ge 1}$ of either \emph{\apdbxy} or \emph{\apdbyx} approximately solves \eqref{eq:conic-problem} \sa{for any $c_{\rm nm} \in \{0,1\}$}, in the sense that
\begin{align*}
\max\Biggl\{
\dist
\left(
\begin{bmatrix}
A\bar{x}^K - b \\
-g(\bar{x}^K)
\end{bmatrix}
,~\{0\} \times \mathcal{K}
\right),
\ |f(\bar{x}^K)-f^\star|
\Biggr\}
= \mathcal{O}(1/T_K)
\end{align*}
for all $K\ge 1$; \sa{moreover, $T_K=\Omega(K)$ for $\mu=0$ while $T_K=\Omega(K^2)$ for $\mu>0$. Finally, for the case $\mu >0$, it also holds that} $\| \bar{x}^K - x^\star\|^2 = \cO(1/K^2)$.   
\end{lemma}
\begin{proof}
\sa{To simplify the notation, let $\cK_0\triangleq \{0\}\times\cK\subset\reals^p\times\reals^m$. Fix some $\kappa>0$. Consider \eqref{eq:minmax_rest} with $B=+\infty$ for \apdbxy{} and with $B=\bar B+\kappa$ for \apdbyx{}. Under this setting, for \apdbxy, since $\hat\cZ^\star=\cZ^\star\neq\emptyset$, fix some $z^\star=(x^\star,y^\star)\in\cZ^\star$; on the other hand, for \apdbyx{}, due to \cref{as:dual-bound}, let $z^\star=(x^\star,y^\star)$ such that $x^\star\in\cX^\star$ and $y^\star\in\cY^\star$ such that $\norm{y^\star}\leq \bar B$ --hence, \cref{lem:zstar-hatzstar} implies that $z^\star\in \hat\cZ^\star\cap\cZ^\star$.} Recall that under the premise of \cref{lem:apdb-xy-bounds}, \apdbxy{} iterate sequence satisfies \eqref{eq: 4.16apd}; similarly, under the premise of \cref{thm:mainII}, \apdbyx{} iterate sequence satisfies \eqref{eq: 3.3apd}.
Next, for $k\geq 0$, 
we define \sa{$\hat y^k$} such that $\sa{\hat y^k} \triangleq 0$ \sa{if both $A\bar{x}^k-b=0$ and $\Pi_{\mathcal{K}^*}\bigl(g(\bar{x}^k)\bigr)$=0; otherwise, we set}
\[ \hat y^k \triangleq (\|y^\star\|+\kappa)
\begin{bmatrix}
A\bar{x}^k-b\\
\Pi_{\mathcal{K}^*}\bigl(g(\bar{x}^k)\bigr)
\end{bmatrix}
\left\|
\begin{bmatrix}
A\bar{x}^k-b\\
\Pi_{\mathcal{K}^*}\bigl(g(\bar{x}^k)\bigr)
\end{bmatrix}
\right\|^{-1} \in \hat{\cY}
.\] Then, for any $K\geq 1$, \eqref{eq: 4.16apd} and \eqref{eq: 3.3apd} imply that ${\cL}(\bar x^K,\hat y^K)-{\cL}(x^\star,\bar y^K)\leq \Delta(x^\star,\hat y^K)/T_K$ holds for both \apdbxy{} and \apdbyx, \sa{where $\Delta(\cdot,\cdot)$ is defined in \eqref{eq:Delta}.} Observe that \sa{using the definition of $\cL$ given in \eqref{eq:minimax-pro}, we get}
\begin{align*}
{\cL}(\bar{x}^K,\hat y^K) =
f(\bar{x}^K)+
\left\langle
\hat{y}^K,
\begin{bmatrix}
A\bar{x}^K-b\\
g(\bar{x}^K)
\end{bmatrix}
\right\rangle
=
f(\bar{x}^K)
+
(\|y^\star\|+\kappa)\cdot
\dist\left(
\begin{bmatrix}
A\bar{x}^K-b\\
-g(\bar{x}^K)
\end{bmatrix}
,~\cK_0\right),
\end{align*}
where the last equality comes from the facts that for any $w\in \mathbb{R}^m$,
\[
w=\Pi_{-\mathcal{K}}(w)+\Pi_{\mathcal{K}^*}(w), \quad
\left\langle \Pi_{-\mathcal{K}}(w),\Pi_{\mathcal{K}^*}(w)\right\rangle=0, \quad \text{and}\quad  d_{-\cK}(w) = \| \Pi_{\cK^*}(w)\|.
\]
In addition, we have $f^\star = {\cL}(x^\star,y^\star) \ge
{\cL}(x^\star,{\bar{y}^K})$. \sa{Finally, note that $\norm{\hat y^K}\leq\norm{y^\star}+\kappa$ for all $K\geq 1$; therefore, $\sup\{\Delta(x^\star,\hat y^K):\ K\geq 1\}=\cO(1)$.}
As a result, the following inequality holds for both \apdbxy{} and \apdbyx{}:
\begin{align}
 f(\bar{x}^K)-f^\star+
(\|y^\star\|+\kappa )\cdot 
\dist
\left(
\begin{bmatrix}
A\bar{x}^K-b\\
-g(\bar{x}^K)
\end{bmatrix}
,~\cK_0\right)
\le
\frac{1}{T_K}\Delta(x^\star,\hat{y}^K)\sa{=\cO\left(\frac{1}{T_K}\right)}.  
\label{eq:subinfe1}
\end{align}
On the other hand, \sa{since $y^\star\in\reals^p\times\cK^*$, for any $y\in\reals^p\times\reals^m$, we have $\fprod{y^\star,y}\leq \fprod{y^\star,\Pi_{\reals^p\times\cK^*}(y)}\leq \norm{y^\star}\cdot 
\dist(y, -\cK_0)$; hence, we get}
\begin{align}\label{eq:subinfe2}
0& \le {\cL}(\bar{x}^K,y^\star)-{\cL}(x^\star,y^\star) =
f(\bar{x}^K)-f^\star+
\left\langle
y^\star,
\begin{bmatrix}
A\bar{x}^K-b\\
g(\bar{x}^K)
\end{bmatrix}
\right\rangle \nonumber\\
&\le
f(\bar{x}^K)-f^\star+ \|y^\star\|\cdot
\dist
\left(\begin{bmatrix}
A\bar{x}^K-b\\
-g(\bar{x}^K)
\end{bmatrix}
,~\cK_0\right).   
\end{align}
Combining \eqref{eq:subinfe1} and \eqref{eq:subinfe2}
gives the desired bound.
\end{proof}

\section{Restarted APDB: Algorithm and Convergence Analysis}
\sa{Based on the unifying convention discussed in \Cref{tab:xy-yx-notation}, throughout this section we consider \eqref{eq:minmax_rest}; hence, setting $B=+\infty$, we can also investigate the behavior of \rapdbxy{} on \eqref{eq:minimax-pro} as well.}
Duality gap given in \Cref{def:gap} and smoothed duality gap given in \Cref{def: smoothed_duality_gap1} are two widely used metrics 
for minimax problems. 
\begin{definition}[Duality gap]
\label{def:gap}
\sa{Consider the minimax problem \eqref{eq:minmax_rest} for $\cB\subseteq\reals^p\times\reals^m$, i.e., for some $B\in\reals_+\cup\{+\infty\}$.} \sa{Let $\hat\cL:\reals^n\times\reals^{p+m}\to\reals\cup\{\pm\infty\}$ such that $\hat\cL(x,y)\triangleq\iota_{\mathcal{X}}(x) + 
\Phi(x,y) - \iota_{\hat{\mathcal{Y}}}(y)$, where $\hat{\mathcal{Y}} \triangleq \cY \cap \mathcal{B}$ denotes the 
dual domain and $\Phi$ is defined in~\eqref{eq:ncc-coupling}.}

For any \sa{$z=(x,y)\in\reals^n\times\reals^{p+m}$}, the duality gap at $z$ is defined as 
\begin{align}
    \sa{\cG(z)\triangleq\sup_{z'}\{\hat{\cL}(x,y')-\hat{\cL}(x',y):\ z'=(x',y')\in\reals^n\times\reals^{p+m}\}.}
\end{align}
\end{definition}
Since \eqref{eq:minmax_rest} is a convex-concave saddle point problem, for any \sa{$z^\star\in\hat\cZ^\star$}, 
it holds that $\cG(z^\star)=0$, i.e., any 
\sa{saddle point of} \eqref{eq:minmax_rest} has zero duality gap. \sa{Therefore, under \Cref{as:saddle-point}, according to \cref{lem:product-saddle-set}, for any primal-dual optimal solution $z^\star=(x^\star,y^\star)$ to \eqref{eq:conic-problem}, it holds that $\cG(z^\star)=0$.}
Furthermore, the duality gap is always nonnegative on $\hat{\cZ}$, i.e., 
\[
\sa{\cG(z) = \sup_{z'}\hat{\cL}(x,y')-\hat{\cL}(x',y)
\ge \Phi(x,y)-\Phi(x,y) = 0,\quad \forall~z \in \hat{\cZ}.}
\]
However, the duality gap cannot properly measure the algorithmic progress towards primal-dual optimality in case the problem domain \jw{$\hat{\cZ}$} is unbounded --as the supremum over $\hat{\cZ}$ may be infinite; \sa{indeed, this is the case for \eqref{eq:minimax-pro} since the dual domain $\cY$ is unbounded.} 
To overcome this issue, the \emph{smoothed duality gap} was proposed in \cite{fercoq2023quadratic} which \sa{modifies the definition of $\cG(\cdot)$ by adding a quadratic regularization term so that the supremum is computed for a strongly concave function;} hence, it always takes a finite value. Given some reference point $\dot z \in \sa{\cZ}$ and a regularization parameter $\xi>0$, we next define the smoothed duality gap function $\mathcal{G}_{\xi}(\cdot;\dot z)$ over $\sa{\reals^n\times\reals^{p+m}}$ using Bregman distances.
\begin{definition}[Smoothed duality gap]
\label{def: smoothed_duality_gap1}
For any given $\xi> 0$ and $\dot z\in \sa{\cZ}$, the smoothed duality gap at \sa{$z=(x,y)\in\reals^n\times\reals^{p+m}$},
centered at the reference point $\dot z$ is defined as
\[
\mathcal{G}_{\xi}(z;\dot z) \triangleq
\sup_{z'} 
\left\{\, \mathcal{Q}(z,z')-2{\xi}\textbf{D}(z',\dot z)\right\},\quad\mbox{where}\quad \mathcal{Q}(z,z')\triangleq \hat{\cL}(x,y')-\hat{\cL}(x',y),
\]
and $\hat\cL(\cdot,\cdot)$ is given in \Cref{def:gap}. In particular, $\mathcal{G}_{\xi}(z) \triangleq \mathcal{G}_{\xi}(z;z)$ is called the (self-centered) smoothed duality gap.
\end{definition}
For any fixed $\xi>0$, it is easy to observe that $\mathcal{G}_\xi(z) \ge 0$ for all $z \in \hat{\cZ}$. 
\sa{Moreover, for any $z^\star \in \hat{\cZ}^\star$,} 
\begin{align*}
\mathcal{G}_\xi(z^\star) &= \sup_{z'}
\left\{\, \hat{\cL}(x^\star,y')-\hat{\cL}(x',y^\star)-2{\xi}\textbf{D}_{p}(x',x^\star)-2{\xi}\textbf{D}_{d}(y',y^\star)\right\}\\
& \le \sup_{z'}
\left\{\, \hat{\cL}(x^\star,y')-\hat{\cL}(x',y^\star) \right\}=\sup_{y'}\hat{\cL}(x^\star,y')-\inf_{x'}\hat{\cL}(x',y^\star)=0,
\end{align*}
where the 
\sa{final} equality follows from $z^\star\in\sa{\hat\cZ^\star}$. 
This, together with $\mathcal{G}_\xi(z^\star)\ge 0$, yields $\mathcal{G}_\xi(z^\star)= 0$. \sa{Conversely, for any $z^\star\in\hat\cZ$, if $\mathcal{G}_\xi(z^\star)= 0$, then  $z^\star\in\hat\cZ^\star$ --in the interest of space, we omit the proof\footnote{This result follows from the arguments in the proof of \cref{thm: qg}. For the proof in the Euclidean case, i.e., $\varphi_p(\cdot)=\varphi_d(\cdot)=\frac{1}{2}\| \cdot \|^2$, see \cite[Proposition 32]{fercoq2023quadratic}.}.}
\subsection{APDB with fixed-frequency restarts}
Building on the \texttt{APDB-xy} and \texttt{APDB-yx} schemes discussed above, we propose a restarted algorithm, referred to as \texttt{rAPDB}, \sa{which is displayed in \Cref{alg:pdhg-qcqp}. The proposed \texttt{rAPDB} method has two variants: \texttt{rAPDB-xy} and \texttt{rAPDB-yx}, depending on whether \texttt{APDB-xy} or \texttt{APDB-yx} is employed, respectively, as the base algorithm.} The method \rapdb{} incorporates an additional restart mechanism when compared to \texttt{APDB}, \sa{i.e., after every fixed number of \texttt{APDB} iterations, say $\hat{K}\in\integers_+$, the \sa{base} method is restarted using the weighted average iterate obtained at the end of $\hat{K}$-th  \texttt{APDB} iteration as the initial point.} \sa{Restarting} strategies have been shown to improve both practical performance and theoretical convergence \sa{properties} in convex optimization \cite{o2015adaptive,aujol2024parameter,roulet2020sharpness,renegar2022simple,tang2018rest} as well as in convex-concave minimax optimization with bilinear couplings \cite{applegate2023faster,lu2026practical,fercoq2023quadratic}. A lower bound on the restart period $\hat{K}\in\integers_+$ to ensure a fast convergence is given in Theorem \ref{thm: lr_rapdb} for both \texttt{rAPDB-xy} and \texttt{rAPDB-yx}, \sa{where we establish that the sublinear convergence rate of \texttt{APDB} can be improved to linear convergence rate through the restarting mechanism.}

\begin{algorithm}[h!]
\caption{\rapdb{} for \eqref{eq:minmax_rest} }
\label{alg:pdhg-qcqp}
\begin{algorithmic}[1]
\Require ${\rm routine}\in\{\apdbxy,\apdbyx\}$
\Require $\mu\ge0$, $c_\alpha,c_\beta,\delta \ge 0$, 
$\eta\in(0,1)$, $\bar\tau$, $\gamma_0>0$, \jw{$c_{\rm nm}\in\{0,1\}$}, 
$\hat{K}\in\integers_+$
\State Initialize $(x^{0,0},y^{0,0})\in \mathcal{X} \times \hat{\cY}$ and $\sa{t}\leftarrow 0$
\While{stopping criteria are not met}
\State \sa{$(\bar{x}^{t,\hat{K}},\bar{y}^{t,\hat{K}})\gets {\rm routine}(x^{t,0},y^{t,0},\hat K)$}
\State 
\sa{$(x^{t+1,0},y^{t+1,0}) \gets (\bar{x}^{t,\hat{K}},\bar{y}^{t,\hat{K}})$} \label{algrapdb-line4}
\State \sa{$t\leftarrow t+1$}
\EndWhile
\end{algorithmic}
\end{algorithm}

\subsection{APDB with adaptive restarts}
In practice, tuning an appropriate restart frequency $\hat{K}\in\integers_+$ is challenging \sa{since the theoretical lower bound established in Theorem \ref{thm: lr_rapdb} depends on the quantities \jw{${\rho}$} and $\hat{\rho}$ that are usually unknown. To mitigate this practical problem,} 
we propose the following adaptive restart scheme: \sa{given some 
$\xi > 0$ and contraction parameter $q \in (0,1)$, for each outer iteration $t\in\integers_+$ of \rapdb, the base subroutine, i.e., either \apdbxy{} or \apdbyx{}, is} restarted, i.e., $z^{t+1,0}=\bar{z}^{t,K}$, whenever the condition,
\begin{align}\label{eq: ada_restart}
\mathcal{G}_\xi(\bar z^{t,K}) \le q\, \mathcal{G}_{\xi}(z^{t,0}) 
\end{align}
\sa{holds for some $K\in\integers_+$}. In the rest of the paper, we refer to \texttt{rAPDB-xy} and \texttt{rAPDB-yx} with adaptive restarts as \texttt{rAPDB-xy-ada} and \texttt{rAPDB-yx-ada}, respectively.
\begin{remark}
Computing 
$\mathcal{G}_\xi$ at any given point $(x,y) \in \hat{\cZ}$ boils down to solving the following  problems:
\begin{align}
    \min_{\hat{x} \in \mathcal{X}} &\; f(\hat{x}) + v^\top (A \hat{x} - b) + \langle \lambda, g(\hat{x})\rangle + 2{\xi}\textbf{D}_{p}(\hat{x},x), \label{eq:ada-compute-x}\\
    \max_{\hat{y} \in \hat{\cY}} & \; f(x) + \hat{v}^\top (A x - b) + \langle {\hat{\lambda}}, g(x) \rangle - 2{\xi}\textbf{D}_{d}(\hat{y},y),\label{eq:ada-compute-y}
\end{align}
\sa{which can be solved efficiently for several important special cases.} In particular, 
for the Bregman distances $\textbf{D}_{p}(\hat x,x)=\frac{1}{2}\norm{\hat x-x}^2$ and $\textbf{D}_{d}(\hat y,y)=\frac{1}{2}\norm{\hat y-y}^2$, 
\eqref{eq:ada-compute-y} reduces to a simple projection problem onto $\hat{\cY}$. These projections can be computed efficiently for many important examples of $\cK$, such as the nonnegative orthant, the second-order cone, and the exponential cone, e.g., see~\cite[Section 6.3]{parikh2014proximal}, \cite{friberg2023projection}, and \cite{bauschke2018projecting}. On the other hand, by exploiting the quadratic structure of \eqref{eq: qcqp-pro}, if the feasible set $\mathcal{X}$ is described by linear constraints, such as box constraints or simplex, then \eqref{eq:ada-compute-x} reduces to solving a strongly convex quadratic program. If instead $\mathcal{X}$ is described by a Euclidean ball constraint, then \eqref{eq:ada-compute-x} reduces to solving a strongly convex trust-region subproblem.
\end{remark}
\subsection{Metric subregularity and quadratic growth}
\sa{In order to 
establish linear convergence of the iterate sequence generated by \rapdb{} and \rapdbada{} 
for \jw{solving \eqref{eq:minmax_rest}}, we require that the KKT map satisfies the metric subregularity condition around each $z^\star\in\hat\cZ^\star$.} 
\begin{assumption}[Metric subregularity \cite{dontchev2009implicit,facchinei2003finite,rockafellar2009variational}]
\label{assu: metric-subreg}
\sa{Let $\hat\cL:\reals^n\times\reals^{p+m}\to\reals\cup\{\pm\infty\}$ such that $\hat\cL(x,y)\triangleq\iota_{\mathcal{X}}(x) + 
\Phi(x,y) - \iota_{\hat{\mathcal{Y}}}(y)$, where $\hat{\mathcal{Y}} \triangleq \cY \cap \mathcal{B}$ denotes the 
dual domain and $\Phi$ is defined in~\eqref{eq:ncc-coupling}.}
The set-valued map $\mathcal{F}: \jw{\hat{\cZ}} \rightrightarrows \mathbb{R}^{n+p+m}$, 
\[
\mathcal{F}(z)
\triangleq
\begin{bmatrix}
\partial_x \sa{\hat\cL(z)} \\
\partial_y (-\sa{\hat\cL(z)})
\end{bmatrix},
\]
associated {with
\eqref{eq:minmax_rest}} is metrically subregular \sa{locally} at every 
\sa{saddle point $z^\star \in \hat\cZ^\star$, i.e.,} 
for every $z^\star \in \sa{\hat\cZ^\star}$, 
there exists an open neighborhood
\sa{$\cB_{\varepsilon_{z^\star}}(z^\star)$} centered at $z^\star$ with the radius $\varepsilon_{z^\star}>0$ and a constant
$\rho_{z^\star} > 0$ (depending on $z^\star$) such that
\begin{align}\label{eq: metric-subreg}
{\rm dist}(z, \hat{\cZ}^\star)
\le
\rho_{z^\star}\cdot
{\rm dist}(0, \mathcal{F}(z)),
\quad
\forall z \in \sa{\hat\cZ} \cap \cB_{\varepsilon_{z^\star}}(z^\star).\end{align}
\end{assumption}

There are various \sa{results in the literature} establishing sufficient conditions under which Assumption \ref{assu: metric-subreg} holds, \sa{e.g.,} see \cite{zheng2007metric,BurkeEngle2020}, \cite[Section 4]{robinson1980strongly}, and \cite[Section 6]{facchinei2003finite}. \sa{We present one such result below for a smooth constrained optimization problem that is very general and extends \cite[Proposition 6.2.7]{facchinei2003finite}. In its full generality, it requires neither the convexity of the feasible set nor the objective function; moreover, it does not make some strong assumptions that are adopted frequently to make the analysis simpler, i.e., it does not require LICQ, does not require multiplier uniqueness, and does not require strict complementarity.}
\begin{theorem}
\label{thm:jongshi}
Let \(\Kcone\subset\R^m\) be a closed convex polyhedral cone and let \(\cX\subset\R^n\) be a nonempty polyhedral set. Consider the following nonlinear optimization problem:
\begin{equation}
\label{eq:js-general_opt}
\min\{f(x):\ x\in \feas\}\quad \mbox{with}\quad \feas\triangleq\{x\in \cX:\ h(x)=0,\ g(x)\in-\Kcone\},
\end{equation}
where \(f:\R^n\to\R\), \(g:\R^n\to\R^m\) and \(h:\R^n\to\R^p\) are twice continuously differentiable near a KKT point \(\bar x\). 
Define the smooth part of the Lagrangian map by
\begin{equation}
\label{eq:L-def}
\Phi(x,v,\lambda)\triangleq f(x)+\langle v,~h(x)\rangle+\langle \lambda,~g(x)\rangle.
\end{equation}
Let $\cM(\bar x)$ denote the multiplier set at \(\bar x\), i.e.,
\begin{equation*}
\cM(\bar x)\triangleq
\left\{(v,\lambda)\in\reals^p\times\cK^*:
-\nabla_x\Phi(\bar x,v,\lambda)\in \cN_{\cX}(\bar x),\quad \langle \lambda,g(\bar x)\rangle=0
\right\}.
\end{equation*}
For $x\in \cX$, $v\in\reals^p$, and $\lambda\in\cK^*$, define the residual
\begin{equation*}
r(x,v,\lambda)\triangleq
\dist\bigl(0,\nabla_x\Phi(x,v,\lambda)+\cN_{\cX}(x)\bigr)+
\|h(x)\|+
\dist\bigl(0,-g(x)+\cN_{\cK^*}(\lambda)\bigr).
\end{equation*}

Suppose conic MFCQ \cite[Eq.(2.190) and Eq.(2.196)]{bonnans2000perturbation} 
holds relative to \(\cX\) at \(\bar x\), i.e.,
\begin{enumerate}
\item[(i)] the equality system is nondegenerate relative to \(\cX\):
\begin{equation}
\label{eq:relative-equality-nondegeneracy}
\nabla h(\bar x)^\top v+s=0,\quad s\in \cN_{\cX}(\bar x)
\quad\Longrightarrow\quad
v=0,\ s=0;
\end{equation}
\item[(ii)] there exists \(d\in \cT_{\cX}(\bar x)\) such that
\begin{equation}
\label{eq:relative-mfcq-direction}
\nabla h(\bar x)d=0,
\qquad
g(\bar x)+\nabla g(\bar x)d\in -\operatorname{int}(\Kcone).
\end{equation}
\end{enumerate}
The conditions in \eqref{eq:relative-equality-nondegeneracy}--\eqref{eq:relative-mfcq-direction} imply that 
\(\cM(\bar x)\) is bounded. Finally, define the critical cone\footnote{$\cC(\bar x;\feas,\nabla f)={\cT}_{\mathcal{X}_{\rm feas}}(\bar x) \cap \left(\nabla {f}(\bar x)\right) ^\perp$, where ${\cT}_{\mathcal{X}_{\rm feas}}(\bar x)$ denotes the tangent cone to $\feas$ at $\bar x$.}
\begin{equation}
\label{eq:critical-cone-def}
\cC(\bar x;\feas,\nabla f)\triangleq
\left\{
 d:
 d\in \cT_{\cX}(\bar x),\quad
 \nabla h(\bar x)d=0,
 \quad
 \nabla g(\bar x)d\in \cT_{-\cK}(g(\bar x)),
 \quad
 \langle \nabla f(\bar x),d\rangle=0
\right\}.
\end{equation}
Assume that for every \((v,\lambda)\in \cM(\bar x)\), the following 
system has only the trivial solution \(d=0\):
\begin{equation}
\label{eq:homogeneous-system-no-q}
\cC(\bar x;\feas,\nabla f)\ni
 d\perp \nabla_{xx}^2\Phi(\bar x,v,\lambda)d
\in \cC(\bar x;\feas,\nabla f)^*.
\end{equation}
Then, \jw{$\bar x$} is an isolated KKT point, and there exist constants \(c>0\) and \(\varepsilon>0\) such that 
\begin{equation}
\label{eq:error-bound-conclusion}
x\in \cX,
\quad
\|x-\bar x\|+
\dist((v,\lambda),\cM(\bar x))\le \varepsilon \implies \|x-\bar x\|+
\dist((v,\lambda),\cM(\bar x))
\le
c\,r(x,v,\lambda).
\end{equation}
\end{theorem}
\begin{proof}
    \sa{The proof is provided in {\Cref{sec:jongshi-proof}}. It extends the proof of \cite[Proposition 6.2.7]{facchinei2003finite} to handle polyhedral-conic constraints intersected with a polyhedral domain $\cX$, whereas \cite[Proposition 6.2.7]{facchinei2003finite} considers $\cK=\reals^m_+$ and $\cX=\reals^n$.}
\end{proof}

\sa{Consider \eqref{eq: qcqp-pro} as a special case of the problem class considered in \eqref{eq:js-general_opt}. Let $\mathcal{X}$ be a nonempty polyhedral set and $x^\star\in \cX_{\rm feas}=\Big\{x\in\cX:\ Ax=b, \; g_i(x)\le 0,\; i\in[m]\Big\}$ be an optimal solution to \eqref{eq: qcqp-pro}. Suppose that MFCQ holds\footnote{MFCQ at $x^\star$ relative to $\cX$ requires that (i) if $A^\top v + s = 0$ holds for some $v\in\reals^p$ and $s \in \cN_\cX(x^\star)$, then $v=0$ and $s = 0$, and that (ii) there exists $d\in \cT_{\cX}(x^\star)$ such that $Ad=0$ and $\big(Q_i x^\star+q_i\big)^\top d < 0$ for $i\in I(x^\star)$; see \cite[Eq.(2.190) and Eq.(2.196)]{bonnans2000perturbation}.} relative to $\cX$ at $x^\star$; hence, the tangent cone is given by
\begin{align*}
\cT_{\cX_{\rm feas}}(x^\star)=
\left\{
d \in \cT_{\cX}(x^\star):\ 
Ad=0,\quad
\big(Q_i x^\star+q_i\big)^\top d\le 0
\quad \forall i\in I(x^\star)
\right\},
\end{align*}
where $I(x^\star)=\{i\in[m]:\ g_i(x^\star)=0\}$ denotes the active set at $x^\star$. Therefore, the critical cone at $x^\star$, i.e., $\mathcal{C}(x^\star)\triangleq \mathcal{C}(x^\star;\mathcal{X}_{\rm feas},\nabla {f})$, can be written as follows:
\[
\mathcal{C}(x^\star)=
\left\{d \in \cT_{\cX}(x^\star):\ 
Ad=0,\quad \big(Q_0 x^\star+q_0\big)^\top d = 0,\quad 
\big(Q_i x^\star+q_i\big)^\top d\le 0
\quad \forall i\in I(x^\star)
\right\}.
\]
Moreover, given $(v,\lambda)\in\cM(x^\star)\subset\reals^p\times\reals^m_+$, let $Q(\lambda)\triangleq Q_0+{\sum_{i\in[m]}}\lambda_i Q_i$ --since $\lambda\in\reals^m_+$, we have $Q(\lambda)\in \mathbb{S}^n_+$. In the light of these observations, applying \Cref{thm:jongshi} to the convex problem \eqref{eq: qcqp-pro}, we obtain the following result.}

\sa{\begin{corollary}
\label{prop:jongshi-coro}
    Let $x^\star$ be an optimal solution to \eqref{eq: qcqp-pro}.
 Suppose that MFCQ {holds relative to $\cX$ at ${x}^\star$} and that for every $(v,\lambda)\in \mathcal{M}({x}^\star)$, the linear complementarity system,
	\begin{equation}
    \label{eq:qcqp-cond}
		d\in\mathcal{C}(x^\star),\quad Q(\lambda)d\in \mathcal{C}(x^\star)^*,\quad d^\top Q(\lambda)d=0,
	\end{equation}
	has a unique solution $d=0$. Then, there exist positive constants $c>0$ and $\varepsilon>0$ such that for all $z=(x,v,\lambda) \in \mathcal{X} \times \mathbb{R}^p \times \mathbb{R}^m_+$ satisfying $\|x-{x}^\star\|+\operatorname{dist}\bigl((v,\lambda),\mathcal{M}({x}^\star)\bigr)\le \varepsilon$,
	it holds that 
	\begin{align*}
	{\rm dist}(z, \mathcal{Z}^\star) \le c \cdot {\rm dist}(0, \mathcal{F}(z)),
    \end{align*}
    where $\cF(\cdot)$ denotes the set-valued map defined in \Cref{assu: metric-subreg} associated with \eqref{eq:minmax_rest} for $B=\infty$, in which case $\hat\cL(\cdot)=\cL(\cdot)$, corresponding to \eqref{eq:minimax-pro} formulation of \eqref{eq: qcqp-pro}.
    \end{corollary}
    Next, we discuss a sufficient condition to ensure \eqref{eq:qcqp-cond} has a trivial solution $d=0$.
    \begin{remark}
    Let $x^\star$ be an optimal solution to \eqref{eq: qcqp-pro}. If $d^\top Q(\lambda)d>0$ for all $d\in \cC(x^\star)\setminus\{0\}$ and $(v,\lambda)\in\cM(x^\star)$, then \eqref{eq:qcqp-cond} has a trivial solution $d=0$. Note that even the sufficient condition discussed in this remark requires neither strict complementarity nor LICQ. 
    \end{remark}}
    
\begin{remark}
     Under the Slater's condition in \Cref{assu: slater}, from \Cref{lem:dual_sol_bounded}, there exists $\bar{B}>0$ such that $\| y^\star \| \le \bar{B}$ for all $y^\star \in \cY^\star$. Let $\cB = \{ y: \| y\| \le B\}$ for some $B > \bar{B}$; hence, $\cY^\star \subset {\rm int}(\cB)$ and $\cY^\star \cap {\rm int}(\cB) = \cY^\star$. As a result, under the premise of \Cref{thm:jongshi}, the metric subregularity in \cref{assu: metric-subreg} still holds even if the set $\cB$ is defined using a finite $B>0$ as given above. Therefore, under \Cref{assu: slater}, \emph{\texttt{APDB-yx}} applied to \eqref{eq:minmax_rest} can be shown to enjoy faster convergence guarantees whenever the conditions in \Cref{thm:jongshi} hold.
\end{remark}
 
Under \cref{assu: metric-subreg}, the metric subregularity holds on any compact \sa{subset of {$\hat{\cZ}$}.}
\begin{lemma}\label{lem: msr-compact}
\sa{Suppose that \cref{as:sets,as:Lipschitz-grads} hold.} Under Assumption \ref{assu: metric-subreg}, for each compact set $\bar{\mathcal{Z}} \subseteq \hat{\cZ}$ with $\bar{\mathcal{Z}} \cap \sa{\hat\cZ^\star} \neq \emptyset$, there exists a positive constant \sa{$\rho(\bar\cZ)$}, which may depend on $\bar{\mathcal{Z}}$, such that 
\begin{align}\label{eq: metric-subreg-local}
{\rm dist}(z, \sa{\hat\cZ^\star}) \le \sa{\rho(\bar\cZ)} \cdot{\rm dist}(0, \mathcal{F}(z)), \quad \forall z \in \bar{\mathcal{Z}}.
\end{align}
\end{lemma}
\begin{proof}
Recall that
$\mathcal{F}:\hat{\cZ}\rightrightarrows \mathbb{R}^{n+p+m}$ is given by $\mathcal{F}(z)= G(z)+\mathcal{N}_{\hat{\cZ}}(z)$ for $z=(x,v,\lambda)\in\hat{\cZ}$,
where
\begin{align}\label{eq: proofmr_G}
{G}(z)\triangleq
\begin{bmatrix}
\nabla f(x) + A^\top v + \sum_{i=1}^m \lambda_i \nabla g_i(x)\\
b-Ax\\ -g(x)\end{bmatrix},
\qquad
\mathcal{N}_{\hat{\cZ}}(z)\triangleq
\begin{bmatrix}
\mathcal{N}_{\mathcal{X}}(x)\\
\mathcal{N}_{\hat{\cY}}(y)
\end{bmatrix}.   
\end{align}
\paragraph{Closed graph of $\mathcal{F}$.}
Since $\hat{\cZ}$ is a closed convex set, $\mathcal{N}_{\hat{\cZ}}(\cdot)$ has a closed graph \cite[Exercise 12.8(b)]{rockafellar2009variational}, i.e.,
given $\{\jw{z^k}\}\subseteq \hat{\cZ}$ and $\{\jw{\omega^k}\}$ such that $\omega^k\in \mathcal{N}_{\hat{\cZ}}(z^k)$ for $k\geq 0$, if $z^k\to \bar z$ and $\omega^k\to \bar \omega$ as $k\to\infty$, then $\bar\omega\in \mathcal{N}_{\hat{\cZ}}(\bar z)$.
Note that ${G}(\cdot)$ is continuous on $\hat{\cZ}$. Consider $\{z^k\}\subseteq \hat{\cZ}$ and $\{v^k\}$ such that $v^k\in \mathcal{F}(z^k)$ for $k\geq 0$, in addition to $z^k\to \bar z$ and $v^k\to \bar v$ as $k\to\infty$.
Since $v^k\in \mathcal{F}(z^k)$, there exists $\omega^k\in N_{\hat{\cZ}}(z^k)$ such that
$v^k={G}(z^k)+\omega^k$ for all $k\geq 0$; hence,
\[
\lim_{k\to\infty}\omega^k
=\lim_{k\to\infty} 
v^k-{G}(z^k)
=\bar v-{G}(\bar z).
\]
The map $\mathcal{N}_{\hat{\cZ}}(\cdot)$ having a closed-graph implies that $\bar v-{G}(\bar z)\in \mathcal{N}_{\hat{\cZ}}(\bar z)$; thus, $\bar v\in {G}(\bar z)+\mathcal{N}_{\hat{\cZ}}(\bar z)=\mathcal{F}(\bar z)$. This shows that $\mathcal{F}(\cdot)$ also has a closed graph.

\paragraph{Lower semicontinuity of $\mathrm{dist}(0,\mathcal{F}(\cdot))$.} See \cite[Theorem 5.7(a) and Theorem 5.11(a)]{rockafellar2009variational}. We include the derivation for completeness. Define $d: \hat{\cZ}\rightarrow \mathbb{R}_+$ with $d(z)\triangleq \mathrm{dist}(0,\mathcal{F}(z))=\inf_{v\in \mathcal{F}(z)}\|v\|$. Let $\{z^k\}\subseteq \hat\cZ$ such that for some $\bar z$, $z^k\to \bar z$ as $k\to\infty$, and define $\bar d\triangleq \liminf_{k\to\infty} d(z^k)$.
Pick a subsequence $\{z^{k_j}\}$ such that $d(z^{k_j})\to \bar d$.
For each $j$, choose $v^{k_j}\in \mathcal{F}(z^{k_j})$ satisfying $\|v^{k_j}\|\le d(z^{k_j})+{1}/{k_j}$. Since $\{v^{k_j}\}$ is bounded, (passing to a further subsequence if needed) we may assume
$v^{k_j}\to \bar v$. \sa{Thus, we get $\norm{\bar v}\leq \liminf_{j\to\infty} d(z^{k_j})+{1}/{k_j}=\bar d$.}
Since $\mathcal{F}$ has a closed graph and $z^{k_j}\to \bar z$ with $v^{k_j}\in \mathcal{F}(z^{k_j})$,
we obtain $\bar v\in \mathcal{F}(\bar z)$.
Consequently,
\[
d(\bar z)=\inf_{v\in \mathcal{F}(\bar z)}\|v\|\le \|\bar v\|
\le 
\bar d,
\]
which shows that $d(\cdot)$ is lower semicontinuous on $\hat{\cZ}$.

Now consider a compact set $\bar{\mathcal{Z}}\subseteq \hat{\cZ}$ such that
$\bar{\mathcal{Z}}\cap \sa{\hat\cZ^\star}\neq \emptyset$.
Since $\sa{\hat\cZ^\star}$ is closed, $\bar{\mathcal{Z}}\cap \sa{\hat\cZ^\star}$ is compact as well.
Consider the collection of open sets defined in \cref{assu: metric-subreg}, i.e., $\{\mathcal{B}_{\varepsilon_{z^\star}}(z^\star)\}_{z^\star\in \bar{\mathcal{Z}}\cap \sa{\hat\cZ^\star}}$,
which forms an open cover of $\bar{\mathcal{Z}}\cap \sa{\hat\cZ^\star}$. Therefore, the compactness of $\bar{\mathcal{Z}}\cap \sa{\hat\cZ^\star}$ implies that there exists a finite subcover $\{\mathcal{B}_{\varepsilon_\ell}(z_\ell^\star)\}_{\ell=1}^M$ for some integer $M$ such that $\varepsilon_\ell\triangleq \varepsilon_{z_\ell^\star}$ for $\ell\in[M]$, \sa{i.e., $\{z_\ell^\star\}_{\ell\in[M]}\subset\bar\cZ\cap\hat\cZ^\star$ and $\bigcup_{\ell\in[M]}\mathcal{B}_{\varepsilon_{\ell}}(z_\ell^\star)\supset\bar\cZ\cap\hat\cZ^\star$.}
Let
\[
\widetilde{\mathcal{B}}\triangleq \bigcup_{\ell=1}^M \mathcal{B}_{\varepsilon_{\ell}}\left(z_\ell^\star \right),
\qquad
\rho_1\triangleq \max_{\ell\in [M]}\rho_{z_\ell^\star}.
\]
Then, it follows from \cref{assu: metric-subreg} that
\begin{align}
\mathrm{dist}(z,\sa{\hat\cZ^\star})\le \rho_1\cdot \mathrm{dist}\bigl(0,\mathcal{F}(z)\bigr),
\qquad \forall\, z\in \bar{\mathcal{Z}}\cap \widetilde{\mathcal{B}}.
\label{eq: proof-lemma1-1}
\end{align}

Since $\bar{\mathcal{Z}}$ is compact and $\widetilde{\mathcal{B}}$ is open, $\bar{\mathcal{Z}}\setminus \widetilde{\mathcal{B}}$ is compact and $\sa{\hat\cZ^\star}\cap(\bar{\mathcal{Z}}\setminus \widetilde{\mathcal{B}})= \sa{\hat\cZ^\star}\cap \bar{\mathcal{Z}} \cap \widetilde{\mathcal{B}}^{c}= \emptyset$.
Note that $d(z)=0$ \sa{for some $z\in\hat\cZ$} if and only if $z\in \sa{\hat\cZ^\star}$; hence, $d(z)>0$ for all $z\in \bar{\mathcal{Z}}\setminus \widetilde{\mathcal{B}}$. Now, let
\begin{align}\label{eq: define-d0}
d_0\triangleq \inf\{d(z):\ z\in \bar{\mathcal{Z}}\setminus \widetilde{\mathcal{B}}\}.
\end{align}
\sa{Next, we show that $d_0>0$. For the sake of contradiction, suppose that $d_0=0$.} Since $d(\cdot)$ is lower semicontinuous and
$\bar{\mathcal{Z}}\setminus \widetilde{\mathcal{B}}$ is compact, \sa{$d_0=0$ implies that} there exists a point $\tilde z\in \bar{\mathcal{Z}}\setminus \widetilde{\mathcal{B}}$
such that $d(\tilde z)=0$; however, \sa{$\tilde z\in\hat\cZ$ and $d(\tilde z)=0$ together imply that} $\tilde z\in \sa{\hat\cZ^\star}$, which is a contradiction.

\sa{Having established that $d_0>0$, let $\rho_2\triangleq \mathrm{diam}(\bar{\mathcal{Z}})/d_0$, and we can conclude that}
\begin{align}
\mathrm{dist}(z,\sa{\hat\cZ^\star})\le \mathrm{diam}(\bar{\mathcal{Z}})
\le \frac{\mathrm{diam}(\bar{\mathcal{Z}})}{d_0}\, d(z)
=\rho_2\,\mathrm{dist}\bigl(0,\mathcal{F}(z)\bigr),\quad\forall~z\in \bar{\mathcal{Z}}\setminus \widetilde{\mathcal{B}},
\label{eq: proof-lemma1-2}
\end{align}
where the first inequality is due to $\bar{\mathcal{Z}}\cap \sa{\hat\cZ^\star}\neq \emptyset$, and the second inequality 
\sa{follows} from \eqref{eq: define-d0}. Finally, 
\sa{setting} $\sa{\rho(\bar\cZ)} \triangleq \max\{\rho_1,\rho_2\}$, we conclude from \eqref{eq: proof-lemma1-1} and \eqref{eq: proof-lemma1-2} that \sa{the desired inequality in \eqref{eq: metric-subreg-local} holds.}
\end{proof}

In the following, we show that under the metric subregularity, the self-centered smoothed duality gap possesses the quadratic growth property. Such a regularity condition plays a significant role in establishing strong convergence results for various iterative methods, see, e.g., \cite{rebjock2025fast,anitescu2000degenerate} for smooth optimization and \cite{drusvyatskiy2018error,liao2024error,davis2025local} for nonsmooth optimization.
\begin{remark}
    The quadratic growth result stated in \cref{thm: qg} applies to the self-centered smoothed duality gap $\mathcal{G}_{\xi}(\cdot)$. 
    On the other hand, the existing work~\cite{fercoq2023quadratic,lu2026practical,huang2025restarted} for tackling LPs and QPs, which are special cases of the problem class \eqref{eq:conic-problem} we consider in this paper, 
    \sa{rely on the quadratic growth of $\mathcal{G}_{\xi}(\cdot;z^\star)$ for $z^\star\in\cZ^\star$ to show their} linear convergence results, e.g., see \cite[Assumption 22]{fercoq2023quadratic}, \cite[Theorem 2]{lu2026practical}, and \cite[Theorem 1]{huang2025restarted}.
\end{remark}
\begin{lemma}[Quadratic growth of smoothed duality gap] \label{thm: qg}
Under \Cref{as:sets,as:Lipschitz-grads,assu: metric-subreg}, for each compact set $\bar{\mathcal{Z}} \subseteq \hat{\cZ}$ with $\bar{\mathcal{Z}} \cap \sa{\hat\cZ^\star} \neq \emptyset$, there exists a positive constant $\sa{\bar{\rho}(\bar\cZ)}$, which may depend on $\bar{\mathcal{Z}}$, such that $\mathcal{G}_\xi(\bar{z}) = \mathcal{G}_{\xi}(\bar{z};\bar{z})$ satisfies
\begin{align}\label{eq: QG-local}
{\rm dist}^2(\bar{z}, \sa{\hat\cZ^\star}) \le \sa{\bar{\rho}(\bar\cZ)} \cdot \mathcal{G}_\xi(\bar{z}), \quad \forall \bar{z} \in \bar{\mathcal{Z}}.
\end{align}
\end{lemma}
\begin{proof}
\jw{Fix} $\xi >0$ and $\bar{z} \in \bar{\mathcal{Z}}$, let $\phi_{\bar{z}}: \hat{\cZ} \rightarrow \mathbb{R}$ be defined as
\begin{align*}
\phi_{\bar{z}}(z) = \sa{\hat\cL}(\bar{x},y)-\sa{\hat\cL}(x,\bar{y}) - 2\xi \textbf{D}(z,\bar z),	
\end{align*}
where \sa{$\hat\cL(x,y)\triangleq\iota_{\mathcal{X}}(x) + 
\Phi(x,y) - \iota_{\hat{\mathcal{Y}}}(y)$ as defined in \cref{assu: metric-subreg} and the Bregman distance $\mathbf{D}(\cdot,\cdot)$ satisfies 
\Cref{def:bregman}; hence,} $\phi_{\bar{z}}$ is $2\xi$-strongly concave. We define $T: 
\sa{\bar\cZ} \rightarrow \hat{\cZ}$ such that 
\begin{align}
\label{eq:Tmap}
    \sa{T(\bar{z}) = \mathop{\arg \max}_{z} \phi_{\bar{z}}(z)
    = \mathop{\arg \max}_{z \in {\hat{\cZ}}} \{\Phi(\bar x,y)-\Phi(x,\bar y)- 2\xi \textbf{D}(z,\bar z)\},\quad\forall~\bar z\in\bar\cZ.}
\end{align}
For any \sa{$z=(x,y)\in\hat\cZ$ and $z'=(x',y')\in \hat\cZ$}, the three-point identity holds: 
\begin{align*}
\textbf{D}_p(x,\bar x)
&=\textbf{D}_p(x,x')+
\textbf{D}_p(x',\bar x)+
\langle
\nabla\varphi_p(x')-\nabla\varphi_p(\bar x),\,x-x'
\rangle, \\
\textbf{D}_d(y,\bar y)
&=\textbf{D}_d(y,y')+
\textbf{D}_d(y',\bar y)+
\langle
\nabla\varphi_d(y')-\nabla\varphi_d(\bar y),\,y-y'
\rangle.
\end{align*}
\jw{Summing the two identities yields} that
\begin{align}\label{eq:qg-1Dpd}
-2\xi \textbf{D}(z,\bar z)
&=-2\xi \textbf{D}(z',\bar z)-2\xi
\left\langle
\begin{bmatrix}
\nabla\varphi_p(x')-\nabla\varphi_p(\bar x)\\
\nabla\varphi_d(y')-\nabla\varphi_d(\bar y)
\end{bmatrix},~
z-z'\right\rangle
-2\xi \textbf{D}(z,z').
\end{align}
\sa{Since $\bar z\in\bar\cZ \jw{\subseteq}\hat\cZ$,
for any $z\in\hat\cZ$, we have $\hat\cL(\bar x,y)-\hat\cL(x,\bar y)=\Phi(\bar x,y)-\Phi(x,\bar y)$; therefore,} 
\begin{align}\label{eq:qg-1L}
\sa{\hat\cL}(\bar x,y)-\sa{\hat\cL}(x,\bar y)\le
\sa{\hat\cL}(\bar x,y')-\sa{\hat\cL}(x',\bar y)+
\left\langle \begin{bmatrix}
-\nabla_x \Phi(x', \sa{\bar y})\\ 
\phantom{-}\nabla_y \Phi(
\sa{\bar x},y')
\end{bmatrix},\,z-z'\right\rangle,\quad\sa{\forall~z,z'\in\hat\cZ.}
\end{align}
Combining \eqref{eq:qg-1Dpd} and \eqref{eq:qg-1L}, we obtain
\[
\phi_{\bar z}(z)\le
\phi_{\bar z}(z')
-2\xi \textbf{D}(z,z')+
\left\langle \begin{bmatrix}
-\nabla_x \Phi(x',\sa{\bar y})\\
\phantom{-}\nabla_y \Phi(\sa{\bar x},y')
\end{bmatrix}-2\xi \begin{bmatrix}
\nabla\varphi_p(x')-\nabla\varphi_p(\bar x)\\
\nabla\varphi_d(y')-\nabla\varphi_d(\bar y)
\end{bmatrix},\;z-z'\right\rangle,\quad\forall~z,z'\in \hat\cZ.
\]
\sa{For $z'=T(\bar z)$, the optimality condition for \eqref{eq:Tmap} implies that 
the inner product in the above inequality is non-positive for all $z\in\hat\cZ$; therefore,}
we obtain
\[
\phi_{\bar z}(z)
\le
\phi_{\bar z}(T(\bar z))
- 2\xi \textbf{D}(z,T(\bar z)),
\quad \forall z\in \hat\cZ.
\]  
Observe that
\begin{align}\label{eq: qg-1}
\mathcal{G}_{\xi}(\bar{z}) = \sa{\max_{z} 
\phi_{\bar{z}}(z)=\phi_{\bar z}(T(\bar z)),}
\end{align}
and $\phi_{\bar{z}}(\bar z) = 0$ \sa{by definition of $\phi_{\bar z}(\cdot)$.} 
As a result, we have
\begin{align}\label{eq: qg-quad}
\mathcal{G}_\xi(\bar{z}) = \phi_{\bar{z}}(T(\bar{z}))\ge \phi_{\bar{z}}(\bar z) + 2\xi \textbf{D}(\bar{z}, T(\bar{z})) = 2\xi \textbf{D}(\bar{z}, T(\bar{z})).
\end{align}
In the rest of the proof, let $\sa{\bar z^*} \triangleq T(\bar{z})$ \sa{with $\bar z^*=(\bar x^*,\bar y^*)$}. Then, by optimality conditions of \eqref{eq: qg-1},
\begin{subequations}
\begin{align}
2\xi\big( \nabla \varphi_p(\bar{x})- \nabla \varphi_p(\sa{\bar x^*})\big)
&\in \partial_x \sa{\hat\cL}\big(\sa{\bar x^*},\bar{y} \big) =  \nabla f(\sa{\bar x^*}) + A^\top \bar{v}+ \sa{\nabla g(\sa{\bar x^*})^\top\bar{\lambda}} +  \mathcal{N}_{\mathcal{X}}(\sa{\bar x^*}), \label{eq: 2a}\\
2\xi\left(\nabla \varphi_d(\bar{y}) - \nabla \varphi_d(\sa{\bar y^*})\right) &\in -
\begin{bmatrix}
    A\bar{x}-b\\
    g(\bar{x})
\end{bmatrix} + \mathcal{N}_{\hat\cY}(\sa{\bar y^*}). \label{eq: 2b}
\end{align}
\end{subequations}
Let \sa{$\Delta x \triangleq \bar x^* - \bar{x}$ and $\Delta y \triangleq \bar y^* - \bar{y}$.}
It follows from \eqref{eq: 2a} that there exists a vector $
\sa{d_x}\in \mathcal{N}_{\mathcal{X}}(\bar x^*)$ such that $\nabla f(\bar x^*) +
\begin{bmatrix}
    A^\top & \nabla g(\bar x^*)^\top
\end{bmatrix} \bar{y} + \sa{d_x} = 2\xi \big( \nabla \varphi_p(\bar{x})- \nabla \varphi_p(\bar x^*)\big) $. Consequently, we have
\begin{align*}
\nabla f(\bar x^*) + 
\begin{bmatrix}
    A^\top & \nabla g(\bar x^*)^\top
\end{bmatrix}
\bar y^* + d_x =  2\xi \big( \nabla \varphi_p(\bar{x})- \nabla \varphi_p(\bar x^*)\big) + \begin{bmatrix}
    A^\top & \nabla g(\bar x^*)^\top
\end{bmatrix} \Delta y.
\end{align*}
\sa{Since $\nabla f(\bar x^*) + 
\begin{bmatrix}
    A^\top & \nabla^\top g(\bar x^*)
\end{bmatrix}
\bar y^* + d_x\in \partial_x \hat\cL(\bar x^*,\bar y^*)$ and $\grad \varphi_p$ is $L_\varphi$-Lipschitz,} we have
\begin{align}
{\rm dist}\big(0, \partial_x \hat\cL(\bar x^*,\bar y^*)\big) 
& \le 
\sa{2\xi L_\varphi\cdot \| \Delta x\|} + \sqrt{\| A\|^2 + B_g^2} \cdot \| \Delta y\| \sa{\triangleq R_x(\bar z^*)}, \label{eq: 3b}
\end{align}
where the constant $B_g$ is defined in \eqref{eq:def_B_g}. Similarly, it follows from \eqref{eq: 2b} that there exists a vector \sa{$d_y \in \mathcal{N}_{\hat\cY}(\bar y^*)$} such that
$2\xi\big( \nabla \varphi_{d}(\bar{y})- \nabla \varphi_{d}(\bar y^*)\big) = -
\begin{bmatrix}
    A\bar{x}-b\\
    g(\bar{x})
\end{bmatrix} + \sa{d_y}$. 
Consequently, we have
\begin{align*}
\sa{\partial_y\big(-\hat\cL(\bar x^*,\bar y^*)\big) \ni} -
\begin{bmatrix}
    A\bar x^*-b\\
    g(\bar x^*)
\end{bmatrix} 
+ d_y = 2\xi\big( \nabla \varphi_{d}(\bar{y})- \nabla \varphi_{d}(\bar y^*)\big) + 
\begin{bmatrix}
    A(\bar{x} - \bar x^*)\\
    g(\bar{x})-g(\bar x^*)
\end{bmatrix};
\end{align*}
and \sa{using the fact that $\grad \varphi_d$ is $L_\varphi$-Lipschitz, we get}
\begin{align}\label{eq: 3a}
{\rm dist}\left(0, \sa{\partial_y\big(-\hat\cL(\bar x^*,\bar y^*)\big)} \right) 
\le 
\sa{2\xi L_\varphi\cdot \| \Delta y\|} + \sqrt{\| A \|^2+C_g^2}\cdot \|\Delta x \|\sa{\triangleq R_y(\bar z^*)},
\end{align}
where the constant $C_g$ is defined in \eqref{eq:def_L_g}.
By combining \eqref{eq: 3b} and \eqref{eq: 3a}, we obtain ${\rm dist}(0, \mathcal{F}(\bar z^*)) \le \sqrt{2} \max \sa{\{R_x(\bar z^*),~R_y(\bar z^*)\}}$,
and using Cauchy-Schwarz we immediately get
\begin{align}\label{eq: dist-deltaz}
{\rm dist}(0, \mathcal{F}(\bar z^*)) \le C_1 \sa{\norm{\bar z^*-\bar z},}
\end{align}
where $C_1 \triangleq  \sqrt{2}~\sa{\max \{ (4\xi^2L_\varphi^2 + \|A\|^2 + C_g^2 )^{1/2}, ( 4\xi^2L_\varphi^2 + \| A\|^2 + B_g^2 )^{1/2} \}=\sqrt{2}(4\xi^2L_\varphi^2 + \| A\|^2 + B_g^2 )^{1/2}}$.

Note that $T: \sa{\bar \cZ} \rightarrow \hat\cZ$ is a continuous function from Berge’s maximum theorem \cite[Corollary 9.20]{sundaram1996first}.
This, together with the compactness of $\bar{\mathcal{Z}}$, implies that $\{T(\bar z): \bar{z} \in \bar{\mathcal{Z}}\}$ is compact.
\sa{Invoking \Cref{lem: msr-compact} for the compact set} $\sa{\cS\triangleq}\bar{\mathcal{Z}} \cup \{T(\bar z): \bar{z} \in \bar{\mathcal{Z}}\}$ and using \eqref{eq: dist-deltaz}, we obtain 
\begin{align*}
\operatorname{dist}\big(\bar z^*,\sa{\hat\cZ^\star}\big)
&\le \sa{\rho(\cS)}\cdot{\rm dist}\big(0, \mathcal{F}(\bar z^*)\big)
\le \sa{\rho(\cS)\cdot C_1 
\| \bar z^* - \bar z\|;}
\end{align*}
\sa{therefore, using the \jw{triangle} inequality, we get}
\begin{align*}
{\rm dist}\big(\bar z,\sa{\hat\cZ^\star}\big)
&\le \|\bar z^*-\bar z\| + {\rm dist}\big(\bar z^*,\sa{\hat\cZ^\star}\big)
\le \sa{(1+\sa{\rho(\cS)} C_1 
)}\,\|\bar z^*-\bar z\|.
\end{align*}
Moreover, from \eqref{eq: qg-quad}, we have $\mathcal{G}_\xi(\bar z) \ge \xi\|\bar z^*-\bar z\|^2$;
thus,
\begin{align*}
\mathcal{G}_\xi(\bar z)
\ge \frac{\xi}{\sa{(1+\rho(\cS) C_1 
)^2}}\,{\rm dist}^2\big(\bar z,\sa{\hat\cZ^\star}\big).
\end{align*}
By letting $\bar{\rho}(\bar\cZ) = (1+ \rho(\cS) C_1 
)^2/\xi$, we finish the proof.
\end{proof}
With the above quadratic growth property at hand, we are ready to show linear convergence rate of \texttt{rAPDB}, displayed in \Cref{alg:pdhg-qcqp}.
\subsection{Linear convergence of \texttt{rAPDB}}
\begin{theorem}\label{thm: lr_rapdb}
Consider \emph{\texttt{rAPDB}} in \Cref{alg:pdhg-qcqp}. For \emph{\rapdbxy}, suppose \cref{as:sets,as:Lipschitz-grads,as:saddle-point,assu: metric-subreg} hold, and for \emph{\rapdbyx} in addition to these, suppose \cref{as:dual-bound} holds as well. Let $\hat\tau=\eta\Psi_1$ for \emph{\rapdbxy} and $\hat\tau=\eta\Psi_2$ for \emph{\rapdbyx}.

For the case $\mu=0$, the restart parameter $\hat K\in\integers_+$ for \emph{\rapdb} is chosen such that
\begin{align}
\label{eq: hatK-restart}
\hat{K} \ge
\begin{cases}
\displaystyle
\max\{4e^2\rho c_1(1+c_2^2),\, 2c_1/\xi\},
& \text{for } \emph{\texttt{rAPDB-xy}}, \\[2mm]
\displaystyle
\max\{4e^2\hat{\rho} {c}_1(1+\hat{c}_2^2),\, 2{c}_1/\xi\},
& \text{for } \emph{\texttt{rAPDB-yx}},
\end{cases}
\end{align}
where $c_1 \triangleq \frac{L_\varphi}{2}\frac{\bar\tau}{\hat\tau}\frac{1}{\min\{\tau_0,\sigma_0\}}$, $c_2 \triangleq \sa{\sqrt{\frac{L_\varphi \cdot \max\{\gamma_0, 1\}}{ \min \{ {\gamma_0 (1-(c_\alpha + c_\beta))}, 1 \}}}}$, 
and $\hat{c}_2 \triangleq \sa{\sqrt{\frac{L_\varphi \cdot \max\{\gamma_0, 1\}}{ \min \{\gamma_0, 1-c_\alpha\}}}}$; moreover, $\rho$ and $\hat{\rho}$ are some positive constants pertaining to the quadratic growth property in \cref{thm: qg} for some compact sets related to \emph{\texttt{rAPDB-xy}} and \emph{\texttt{rAPDB-yx}}, respectively. More precisely, $\rho\triangleq \bar\rho(\cZ_0)$ is associated with \emph{\texttt{rAPDB-xy}} for the compact set
\begin{align*}
\mathcal{Z}_0 \triangleq \left\{z \in {\cZ}: \| z - z^{0,0}\| \le R \right\},
\end{align*}
where $R\triangleq \tfrac{c_2+1}{1-1/e}\;{\rm dist}\bigl(z^{0,0},{\cZ}^\star\bigr)$; and $\hat\rho\triangleq \bar\rho(\hat\cZ)$ is associated with \emph{\texttt{rAPDB-yx}} for the compact set $\hat{\cZ}$.

For the case $\mu >0$, \jw{we set $\varphi_p(\cdot)=\frac{1}{2}\|\cdot\|^2$} and the parameter $\hat K\in\integers_+$ is chosen such that
\begin{align}
\label{eq: hatK-restart-k2}
\hat{K} \ge
\begin{cases}
\displaystyle
\max\{2e\sqrt{\rho c_3(1+c_2^2)},\, \sqrt{2c_3/\xi},\,2\},
& \text{for } \emph{\texttt{rAPDB-xy}}, \\[2mm]
\displaystyle
\max\{2e\sqrt{\hat{\rho}{c}_3(1+\hat{c}_2^2)},\, \sqrt{2c_3/\xi},\,2\},
& \text{for } \emph{\texttt{rAPDB-yx}},
\end{cases}
\end{align}
where \sa{$c_3 \triangleq \frac{6L_\varphi}{\mu}\frac{1}{(\hat{\tau})^2}\frac{1}{\min\{1,\gamma_0\}}$.} 
Then, the following results hold:
\begin{enumerate}
\item[(a)] For \emph{\texttt{rAPDB-xy}}, $\{z^{t,0}\}_{t\ge 0}$ is bounded and satisfies 
\begin{align}
\label{eq:result-iterate-bound}
    \bigl\|z^{t,0}-z^{0,0}\bigr\| \le R,
    \qquad\forall~t\geq 0.
\end{align}
\item[(b)] \sa{For \emph{\texttt{rAPDB-xy}} and \emph{\texttt{rAPDB-yx}},} the distance between 
$\{z^{t,0}\}_{t\geq 0}$ and 
$\hat{\cZ}^\star$ diminish exponentially, i.e.,
\begin{align}
\label{eq:result-rate}
{\rm dist}\bigl(z^{t,0},\hat{\cZ}^\star\bigr)
\le
e^{-t}\,{\rm dist}\bigl(z^{0,0},\hat{\cZ}^\star\bigr),\qquad \forall t \ge 0.
\end{align}
\end{enumerate}
Moreover, $\lim_{t \rightarrow +\infty} z^{t,0} = \tilde{z}^\star$ for some saddle point $\tilde{z}^\star \in \hat{\mathcal{Z}}^\star$. 
\end{theorem}

\begin{proof}
\sa{Consider \rapdbxy, i.e., ${\rm routine}=\apdbxy{}$ in \Cref{alg:pdhg-qcqp}, to solve \eqref{eq:minmax_rest} with $B=\infty$. Recall from \Cref{tab:xy-yx-notation} that for this setting $\hat\cY=\cY$, $\hat\cZ=\cZ$ and $\hat\cZ^\star=\cZ^\star$. Thus, given an arbitrary saddle point $({x}^\star, {y}^\star) \in {\cZ}^\star$, since ${\cL}(\bar{x}^{t,k},{y}^\star) - {\cL}({x}^\star,\bar{y}^{t,k}) \ge 0$ for all $t\ge 0$ and $k \ge 1$, \eqref{eq: 4.16apd} implies that}
\[
\frac{\gamma_0 (1-(c_\alpha + c_\beta))}{\sigma_0}
\textbf{D}_{p}({x}^\star,x^{t,k}) + \frac{1}{\sigma_0}
\textbf{D}_{d}({y}^\star,y^{t,k}) \le\Delta({x}^\star,{y}^\star) \le \max\left\{\frac{1}{2\tau_0},\frac{1}{2\sigma_0}\right\}\cdot L_\varphi \|z^{t,0}- {z}^\star\|^2
\]
\sa{holds for all $t\ge 0$ and $k \ge 1$}, where we used \sa{$\gamma_{k-1}\ge \gamma_0$} for any $k\ge 1$; \sa{hence, $\sigma_0=\gamma_0\tau_0$ implies that}
\begin{align} \label{eq:restart_proof_psudo_contraction}
\|z^{t,k}-{z}^\star\|\le c_2\,\|z^{t,0}- {z}^\star\|,\qquad\forall~t\geq 0,\quad \forall~k\geq 1,  
\end{align}
holds for both $\mu=0$ and $\mu>0$. \sa{For the case $\mu = 0$,  \Cref{lem: apdb-sublinear,lem: apdb-sublinear-nm} show $T_K \ge K \hat\tau/\bar{\tau}$ for any $K\geq 1$, where $\hat\tau=\eta \Psi_1$; therefore, invoking \eqref{eq: 4.16apd} one more time, we get}
\begin{align} \label{eq:restart_proof_sublinear}
{\cL}(\bar x^{t,K},y)-{\cL}(x,\bar y^{t,K} )\le  \frac{c_1}{K}\,\|z-z^{t,0}\|^2 \quad \forall~t \ge 0,\quad \forall~K\ge 1,\quad \forall~z = (x,y)\in \sa{\cZ}. 
\end{align}
\sa{Next, we consider the smoothed duality gap $\mathcal{G}_\xi(\cdot)$ in \Cref{def: smoothed_duality_gap1} and we show an upper bound on $\mathcal{G}_\xi(\bar z^{t,K})$. Indeed, \eqref{eq:restart_proof_sublinear} implies that}
\begin{align}\label{eq: quadGxi}
{\cL}(\bar x^{t,K},y)-{\cL}(x,\bar y^{t,K})-2{\xi}\textbf{D}_{p}(x,\bar x^{t,K})-2{\xi}\textbf{D}_{d}(y,\bar y^{t,K})
\le \frac{c_1}{K}\|z-z^{t,0}\|^2-\xi\|z-\bar z^{t,K}\|^2,
\end{align}
\sa{which follows from \eqref{eq:Bregman-bounds}. For any $K\ge \frac{2c_1}{\xi}$, since $\frac{c_1}{K}<\xi$}, the right-hand side  of \eqref{eq: quadGxi} is concave in $z$ \sa{and its supremum is} $\frac{\xi c_1/K}{\xi-c_1/K}\,\|\bar z^{t,K}-z^{t,0}\|^2$. \sa{Therefore, for any $K\ge \frac{2c_1}{\xi}$}, \sa{due to} $\xi-\frac{c_1}{K}\ge \frac{\xi}{2}$, we get
\begin{align}\label{eq: G_xi_upperbound}
\max_{z}\Bigl\{\frac{c_1}{K}\|z-z^{t,0}\|^2-\xi\|z-\bar z^{t,K}\|^2\Bigr\}
\le \frac{2c_1}{K}\,\|\bar z^{t,K}-z^{t,0}\|^2.
\end{align}
Hence, \sa{\eqref{eq: quadGxi} and \Cref{def: smoothed_duality_gap1} imply for any $t\geq 0$ and $K\ge \frac{2c_1}{\xi}$ that}
\begin{align}\label{eq:gxi_linearrate_xy}
\mathcal{G}_\xi(\bar z^{t,K})\le \frac{2c_1}{K}\,\|\bar z^{t,K}-z^{t,0}\|^2
\le \frac{4c_1}{K}\Bigl(\|z^{t,0}-z^\star\|^2+\|\bar z^{t,K}-z^\star\|^2\Bigr)
\le \frac{4c_1(1+c_2^2)}{K}\,{\rm dist}^2\bigl(z^{t,0},{\cZ}^\star\bigr),
\end{align}
where $z^\star\in \mathop{\arg\min}_{z\in{\cZ}^\star}\|z^{t,0}-z\|^2$, \sa{and the last inequality follows from \eqref{eq:restart_proof_psudo_contraction}, which implies that $\|\bar z^{t,K}-{z}^\star\|\le c_2\,\|z^{t,0}- {z}^\star\|$.}

\sa{Consider \rapdbyx, i.e., ${\rm routine}=\apdbyx{}$ in \Cref{alg:pdhg-qcqp}, to solve \eqref{eq:minmax_rest} using $B>\bar B$ such that $\bar B$ satisfies \cref{as:dual-bound}. Given an arbitrary saddle point $({x}^\star, {y}^\star) \in \hat{\cZ}^\star$, similar to the \rapdbxy{} discussion above, it follows from \eqref{eq: 3.3apd} that}
\begin{align}
\|z^{t,k}-\hat{z}^\star\|\le \hat{c}_2\,\|z^{t,0}-\hat{z}^\star\|,\qquad\forall~t\geq 0,\quad \forall~k\geq 1,
\label{eq:restart_proof_psudo_contraction_yx}
\end{align}
holds for both $\mu=0$ and $\mu>0$. \sa{For the case $\mu = 0$, using $z^\star\in \mathop{\arg\min}_{z\in\hat{\cZ}^\star}\|z^{t,0}-z\|^2$ as in \eqref{eq:gxi_linearrate_xy}, for any $t\geq 0$ and $K \ge \frac{2{c}_1}{\xi}$, it holds that}
\begin{align}
\mathcal{G}_\xi(\bar z^{t,K})\le \frac{2{c}_1}{K}\,\|\bar z^{t,K}-z^{t,0}\|^2
\le \frac{4{c}_1}{K}\Bigl(\|z^{t,0}-z^\star\|^2+\|\bar z^{t,K}-z^\star\|^2\Bigr)
\le \frac{4{c}_1(1+\hat{c}_2^2)}{K}\,{\rm dist}^2\bigl(z^{t,0},\hat{\cZ}^\star\bigr).\label{eq:gxi_linearrate_yx}    
\end{align}
\paragraph{Proof for \rapdbxy{} with $\mu=0$.}
\sa{Recall that $\hat\cZ^\star=\cZ^\star$ for this scenario.} Now, \sa{for \rapdbxy{} with $\mu=0$, we prove the results in \eqref{eq:result-iterate-bound} and \eqref{eq:result-rate} by induction.} When $t=0$, 
\sa{both \eqref{eq:result-iterate-bound} and \eqref{eq:result-rate} trivially hold.
Next, we consider the inductive step, i.e., given an arbitrary $D\in\integers_+$, suppose that \eqref{eq:result-iterate-bound} and \eqref{eq:result-rate} hold for all $t\le D$. Using $z_t^\star \in \mathop{\arg \min}_{z \in {\cZ}^\star}\| z^{t,0} - z\|^2$ for $t=0,\ldots,D$, we argue that}
\begin{align}
\MoveEqLeft \|z^{D+1,0}-z^{0,0}\|
\le \sum_{t=0}^{D}\|z^{t+1,0}-z^{t,0}\|
\le \sum_{t=0}^{D}\bigl(\|z^{t+1,0}-z_t^\star\|+\|z_t^\star-z^{t,0}\|\bigr) \nonumber \\
&\le (c_2+1)\sum_{t=0}^{D}{\rm dist}(z^{t,0},{\cZ}^\star)
\le (c_2+1)\sum_{t=0}^{D}e^{- t}{\rm dist}(z^{0,0},{\cZ}^\star) \le \frac{c_2+1}{1-1/e}{\rm dist}(z^{0,0},{\cZ}^\star), 
\label{eq: cauchy-seq}
\end{align}
where the third inequality uses \eqref{eq:restart_proof_psudo_contraction} and the fact that $z^{t+1,0}$ can be written as a convex combination of $\{z^{t,k}\}_{k=1}^{\hat{K}}$, and the fourth inequality follows from the induction hypothesis of \eqref{eq:result-rate} for $t \le D$. \sa{This completes the induction argument for \eqref{eq:result-iterate-bound} establishing that $\{z^{t,0}\}_{t\geq 0}\subset\cZ_0$.} 
Furthermore,
\begin{align*}
{\rm dist}^2\bigl(z^{D+1,0},{\cZ}^\star\bigr)
&\overset{(a)}{\le} \sa{\rho} \cdot \mathcal{G}_\xi\!\bigl(z^{D+1,0}\bigr)
= \sa{\rho} \cdot \mathcal{G}_\xi\bigl(\bar z^{D,\hat{K}}\bigr)
\overset{(b)}{\le}\ \sa{\rho}\cdot \frac{4c_1(1+c_2^2)}{\hat{K}}\,\sa{{\rm dist}^2\bigl(z^{D,0},{\cZ}^\star\bigr)}\\
&\le e^{-2}{\rm dist}^2\bigl(z^{D,0},{\cZ}^\star\bigr)
\le e^{-2(D+1)}{\rm dist}^2\bigl(z^{0,0},{\cZ}^\star\bigr),
\end{align*}
where $(a)$ \sa{follows from the quadratic growth property in \Cref{thm: qg} corresponding to the compact set $\cZ_0\supset \{z^{t,0}\}_{t\geq 0}$, i.e., $\rho=\bar\rho(\cZ_0)$,} 
and $(b)$ is due to the sublinear rate of $\mathcal{G}_\xi(\bar z^{D,K})$ in \eqref{eq:gxi_linearrate_xy} for $K > 0$. \sa{This completes the induction argument for \eqref{eq:result-rate} establishing the linear convergence rate.} \sa{Finally, \eqref{eq: cauchy-seq} implies that $\sum_{t=0}^{\infty}\norm{z^{t+1,0} -z^{t,0}}<\infty$, which further implies} $\{z^{t,0}\}_{t \ge 0}$ is a Cauchy sequence. Thus, $\lim_{t \rightarrow +\infty}{\rm dist}(z^{t,0},{\cZ}^\star) = 0$ implies that there exists $\tilde{z}^\star \in {\cZ}^\star$ for which $\lim_{t \rightarrow +\infty} z^{t,0} = \tilde{z}^\star$. 
\paragraph{Proof for \rapdbyx{} with $\mu=0$.} \sa{For this setting, by construction, $\hat\cZ$ is compact and we have $\{z^{t,0}\}_{t\geq 0}\subset\hat\cZ$.} For any \sa{$D \in\integers_+$}, it follows that
\begin{align*}
{\rm dist}^2\bigl(z^{D+1,0},\hat{\mathcal{Z}}^\star\bigr)
&\overset{(a)}{\le} \hat{\rho} \cdot {\mathcal{G}}_\xi \bigl(z^{D+1,0}\bigr)
= \hat{\rho} \cdot {\mathcal{G}}_\xi\bigl(\bar z^{D,\hat{K}}\bigr)
\overset{(b)}{\le}\ \hat{\rho}\cdot \frac{4c_1(1+\hat{c}_2^2)}{\hat{K}}\,{\rm dist}^2\bigl(z^{D,0},\hat{\mathcal{Z}}^\star\bigr)\\
&\le e^{-2}{\rm dist}^2\bigl(z^{D,0},\hat{\mathcal{Z}}^\star\bigr)
\le e^{-2(D+1)}{\rm dist}^2\bigl(z^{0,0},\hat{\mathcal{Z}}^\star\bigr),
\end{align*}    
where $(a)$ follows from the quadratic growth property in Lemma \ref{thm: qg} over the compact set $\hat{\cZ}$, i.e., $\hat\rho=\bar\rho(\hat\cZ)$, and $(b)$ is due to \eqref{eq:gxi_linearrate_yx}. Choosing $z_t^\star \in \mathop{\arg \min}_{z \in \hat{\mathcal{Z}}^\star}\| z^{t,0} - z\|^2 $ for $t \ge 0$, we have, 
\begin{align}
& \quad \sum_{t=0}^{\infty}\|z^{t+1,0}-z^{t,0}\|
\le \sum_{t=0}^{\infty}\bigl(\|z^{t+1,0}-z_t^\star\|+\|z_t^\star-z^{t,0}\|\bigr) \nonumber \\
&\le (\hat{c}_2+1)\sum_{t=0}^{\infty}{\rm dist}(z^{t,0},\hat{\cZ}^\star)
\le (\hat{c}_2+1)\sum_{t=0}^{\infty}e^{- t}{\rm dist}(z^{0,0},\hat{\cZ}^\star)<\infty; \label{eq: cauchy-seq2}
\end{align}
hence, $\{z^{t,0}\}_{t \ge 0}$ is a Cauchy sequence and $\lim_{t \rightarrow +\infty} z^{t,0} = \tilde{z}^\star$ for some $\tilde{z}^\star \in \hat{\mathcal{Z}}^\star$. 

\paragraph{Proof for \rapdb{} with $\mu >0$.} The case of $\mu>0$ for \rapdb{} can be proved by utilizing the bound $T_K \ge \frac{\Gamma^2}{12\mu \sigma_0}K^2$ with $\Gamma = \mu \hat{\tau}\sqrt{\gamma_0}$ for $K\ge 2$, \sa{which holds for \apdbxy{} --see~\Cref{lem: apdb-sublinear,lem: apdb-sublinear-nm} for $c_{\rm nm}=0$ and $c_{\rm nm}=1$, respectively, and also for \apdbyx{} --see~\Cref{lem: apdb-yx-monotone,lem: apdb-yx-sublinear-nm} for $c_{\rm nm}=0$ and $c_{\rm nm}=1$, respectively.} For \rapdbxy{}, 
\eqref{eq: 4.16apd} implies that for any $t \ge 0$ and $K\ge 2$, we have
${\cL}(\bar x^{t,K},y)-{\cL}(x,\bar y^{t,K}) \le
\frac{c_3}{K^2}\|z-z^{t,0}\|^2$ for all $z=(x,y)\in{\cZ}$; hence, 
\begin{align} \label{eq:lr_sc_gxi_ub}
\mathcal G_\xi(\bar z^{t,K})\le
\max_{z}
\left\{
\frac{c_3}{K^2}\|z-z^{t,0}\|^2-
\xi\|z-\bar z^{t,K}\|^2
\right\}.    
\end{align}
\sa{For any $t\geq 0$ and $K\ge \max\{\sqrt{{2c_3}/{\xi}},~2\}$, we have
\(c_3/K^2\le \xi/2<\xi\); thus, the r.h.s. of \eqref{eq:lr_sc_gxi_ub} is a strongly concave problem. Explicitly computing the supremum yields for any $t\geq 0$ that}
\[
\mathcal G_\xi(\bar z^{t,K}) \le \max_{z}
\left\{
\frac{c_3}{K^2}\|z-z^{t,0}\|^2-
\xi\|z-\bar z^{t,K}\|^2
\right\}
\leq
\frac{2c_3}{K^2}\|\bar z^{t,K}-z^{t,0}\|^2,\qquad \sa{\forall~K\ge \max\left\{\sqrt{{2c_3}/{\xi}},~2\right\}.}
\]
For any $t\geq 0$, let $z_t^\star
\in
\operatorname*{argmin}_{z\in{\cZ}^\star}
\|z^{t,0}-z\|$. \sa{Therefore, using \eqref{eq:restart_proof_psudo_contraction} for large enough $K$, we get}
\[
\mathcal G_\xi(\bar z^{t,K})
\le \frac{4c_3}{K^2}(\|\bar z^{t,K}-z_t^\star\|^2
+\|z^{t,0}-z_t^\star\|^2) \le \frac{4c_3}{K^2}(1+c_2^2)\|z^{t,0}-z_t^\star\|^2 =
\frac{4c_3(1+c_2^2)}{K^2}
\operatorname{dist}^2(z^{t,0},{\cZ}^\star).
\]
\sa{Note that the argument we used for showing \eqref{eq: cauchy-seq} still holds; therefore, we have $\{z^{t,0}\}_{t\geq 0}\subset\cZ_0$ for this scenario as well.} Thus, using the quadratic growth property in \Cref{thm: qg} over the compact set $\cZ_0$,
\[
\operatorname{dist}^2(z^{t+1,0},{\cZ}^\star)
=
\operatorname{dist}^2(\bar z^{t,{K}},{\cZ}^\star)
\le
\rho\,\mathcal G_\xi(\bar z^{t,{K}}) \le \frac{4\rho c_3(1+c_2^2)}{K^2}
\operatorname{dist}^2(z^{t,0},{\cZ}^\star)\leq e^{-2}\operatorname{dist}^2(z^{t,0},{\cZ}^\star),
\]
for $K\ge
\max\{
2e\sqrt{\rho c_3(1+c_2^2)},
\sqrt{{2c_3}/{\xi}}, 2 \}$, where $\rho=\bar\rho(\cZ_0)$.
Thus, $\operatorname{dist}(z^{t+1,0},{\cZ}^\star)\le e^{-1}
\operatorname{dist}(z^{t,0},{\cZ}^\star)$.
The associated results of \rapdbyx{} can be derived in a similar manner.

\end{proof}

\subsection{Linear convergence of \texttt{rAPDB-ada}}
\label{sec:rAPDB-ada}
\paragraph{\texttt{rAPDB-yx-ada}.} Suppose \Cref{as:sets,as:Lipschitz-grads,as:saddle-point,assu: metric-subreg,as:dual-bound} hold. For the case $\mu = 0$, given some $\xi >0$, by \eqref{eq:gxi_linearrate_yx} and the quadratic growth property in Lemma \ref{thm: qg} over the compact set $\hat{\cZ}$, it holds that
\begin{align*}
{\mathcal{G}}_\xi(\bar z^{t,K})\le \frac{4 {c}_1(1+\hat{c}_2^2)}{K}\,{\rm dist}^2\bigl(z^{t,0},\hat{\mathcal{Z}}^\star\bigr) \le \frac{4{c}_1(1+\hat{c}_2^2)}{K} \hat{\rho} \cdot {\mathcal{G}}_\xi(z^{t,0}) \quad \forall t \ge 1,\quad \forall~K \ge {2{c}_1}/{\xi}.
\end{align*}
Consequently, for every $t\geq 0$, the requirement ${\mathcal{G}}_\xi(\bar z^{t,K}) 
\le q \cdot {\mathcal{G}}_\xi(z^{t,0})$ can always be satisfied for any $K \ge \bar K$ where the integer $\bar{K} \triangleq \lceil \max \{ 4 {c}_1(1+\hat{c}_2^2) \hat{\rho}/q, 2c_1/\xi \}\rceil$. Thus, we obtain ${\mathcal{G}}_\xi({z}^{t,0}) \le q^t  {\mathcal{G}}_\xi({z}^{0,0})$ for all $t\ge 0$.
This, together with the quadratic growth property, gives
\begin{align}\label{eq:yx-ada-contraction}
{\rm dist}^2({z}^{t,0}, \hat{\mathcal{Z}}^\star) \le \hat{\rho } \cdot {\mathcal{G}}_\xi({z}^{t,0}) \le q^t \hat{\rho}\cdot {\mathcal{G}}_\xi({z}^{0,0})\qquad \forall~t\geq 0.
\end{align}
Choosing $z_t^\star \in \mathop{\arg \min}_{z \in \hat{\mathcal{Z}}^\star}\| z^{t,0} - z\|^2 $ for $t \ge 0$, by \eqref{eq:restart_proof_psudo_contraction_yx} and \eqref{eq:yx-ada-contraction}, we have, 
\begin{align}
\MoveEqLeft \sum_{t=0}^{\infty}\|z^{t+1,0}-z^{t,0}\|
\le \sum_{t=0}^{\infty}\bigl(\|z^{t+1,0}-z_t^\star\|+\|z_t^\star-z^{t,0}\|\bigr) \nonumber \\
&\le (\hat{c}_2+1)\sum_{t=0}^{\infty}{\rm dist}(z^{t,0},\hat{\cZ}^\star)
\le (\hat{c}_2+1) \sum_{t=0}^{\infty}{q}^{t/2} \sqrt{\hat{\rho}\cdot {\mathcal{G}}_\xi({z}^{0,0})} < \infty. \label{eq: cauchy-seq2-ada}
\end{align}
Hence, $\{z^{t,0}\}_{t \ge 0}$ is a Cauchy sequence and $\lim_{t \rightarrow +\infty} z^{t,0} = \tilde{z}^\star\in \hat{\mathcal{Z}}^\star$ with a linear rate. Results for the case $\mu>0$ can be derived similarly by replacing the $\mathcal{O}(1/K)$ upper
bound in~\eqref{eq:gxi_linearrate_yx} with the corresponding
$\mathcal{O}(1/K^2)$ bound.

\paragraph{\texttt{rAPDB-xy-ada}.} Suppose that \Cref{as:sets,as:Lipschitz-grads,as:saddle-point,assu: metric-subreg} hold. If the restart points \(\{z^{t,0}\}_{t\ge 0}\) generated by \texttt{rAPDB-xy-ada} are bounded, then the linear convergence result for \texttt{rAPDB-xy-ada} can be derived in the same way as that for \texttt{rAPDB-yx-ada}. Establishing the same guarantee without this boundedness assumption is left as future work.

\section{Numerical Experiments}
In this section, we conduct two sets of experiments to evaluate various methods including \texttt{rAPDB}, \texttt{rAPDB-ada}, \texttt{APDB}~\cite{hamedani2021primal}, EGM~\cite[Section 3.3]{diaz2026active}, and MOSEK (called through CVXPY \cite{diamond2016cvxpy}). The first set consists of randomly generated convex QCQPs with \(\mu=0\). The second set is a kernel matrix learning problem, which leads to a strongly convex-concave minimax formulation with \(\mu=2\). All experiments were run on a cluster with 1 NVIDIA A100 GPU, 
128 GB RAM, and 8 CPU cores. Our code is available at \href{https://github.com/JinxinWang-resilience/rAPDB_2026}{https://github.com/JinxinWang-resilience/rAPDB\_2026}.

\subsection{
Synthetic QCQPs}
We consider \eqref{eq: qcqp-pro} without linear equality constraints
setting $\mathcal{X}=[-10,10]^n$, 
and choosing $\{Q_i\}_{i=0}^m \subset \mathbb{S}_+^n$ such that
$Q_i = \Lambda_i^\top S_i \Lambda_i,\ i=0,1,\ldots,m$,
where $\Lambda_i\in\mathbb{R}^{n\times n}$ is a random orthonormal matrix and $S_i\in\mathbb{R}_+^{n\times n}$ is diagonal with entries sampled uniformly from $[0,100]$ \sa{(the smallest element is set to $0$ to ensure $\mu=0$).} The orthonormal matrix $\Lambda_i$ is obtained by computing an orthonormal basis of a random Gaussian matrix $\tilde{\Lambda}_i\in\mathbb{R}^{n\times n}$. The vectors $\{q_i\}_{i=0}^m \subset \mathbb{R}^n$ are sampled from the standard Gaussian distribution $\mathcal{N}(0,I_n)$, while the constants $\{-r_i\}_{i=1}^m$ are drawn independently from the uniform distribution on $[0,1]$.

\begin{figure}[htb]
\centering
\includegraphics[width=0.95\textwidth]
{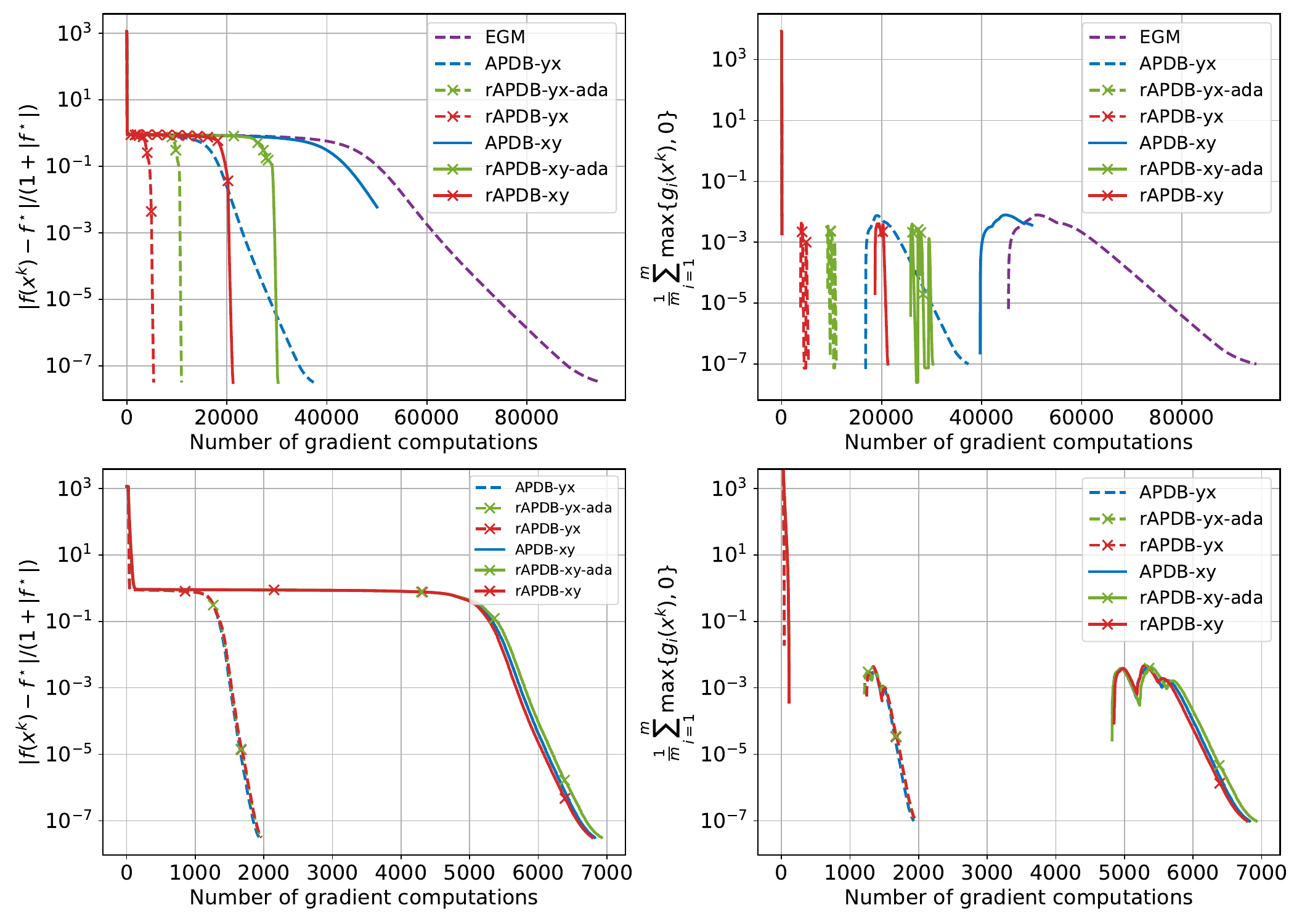}
\caption{\small Convergence performance of \rapdb, \rapdbada, \texttt{APDB}, and \texttt{EGM} on randomly generated QCQP instances with $m=10$ and $n=1000$. The plots display the relative suboptimality and constraint violation against the total number of gradient evaluations. The top row corresponds to monotone step-size search, i.e., $c_{\rm nm}=0$, while the bottom row corresponds to non-monotone step-size search, i.e., $c_{\rm nm}=1$.}
\label{fig:randqcqp}
\vspace*{-3mm}
\end{figure}


\paragraph{Implementation details}
We set \(m=10\), \(n=1000\), generate four random i.i.d. QCQP instances, and run all algorithms until the termination criterion,
{\small
$\max\left\{\frac{|f(x^k)-f^\star|}{1 + |f^\star|},
\ \frac{1}{m}\sum_{i=1}^m \max \left\{g_i(x^k),0 \right\}\right \}\le \varepsilon$},
is satisfied for \(\varepsilon=10^{-7}\), which accounts for both relative suboptimality and infeasibility error. Here, \(f^\star\) is obtained by MOSEK. The maximum number of iterations of tested algorithms is set to \(50{,}000\). For \texttt{EGM}, we tune the step sizes and report the best-performing results. The backtracking parameter is fixed at \(\eta=0.7\) for all \texttt{APDB}-type methods. For the yx-variants of \texttt{APDB}-type methods, we do not impose the set \(\mathcal{B}\) in \eqref{eq: apdb-yx-y} and choose \(c_\alpha = 0.4\) and \(\delta = 0.5\). For the xy-variants, we choose \(c_\alpha = 0.25\), \(c_\beta = 0.3\), and \(\delta = 0.4\). We test both monotone ($c_{\rm nm}=0$) and non-monotone ($c_{\rm nm}=1$) step-size search strategies. \sa{In the rest of this section, we abbreviate the setting with $c_{\rm nm}=0$ as (mono), and the setting with $c_{\rm nm}=1$ as (non-mono).}

The restart frequencies for \texttt{rAPDB-yx} (mono), \texttt{rAPDB-xy} (mono), \texttt{rAPDB-yx} (non-mono), and \texttt{rAPDB-xy} (non-mono) are set to \(800\), \(2000\), \(400\), and \(1000\), respectively. For the adaptive restart variants of \texttt{APDB}, we fix \(\xi = 0.04\). For \texttt{rAPDB-yx-ada}, we first compute \(\mathcal{G}_{\xi}\) at iteration \(k=50\) without restarting. We then evaluate the adaptive restart condition at every \(200\) iterations 
when $c_{\rm nm}=0$, and at every \(500\) iterations when $c_{\rm nm}=1$,
to check whether \(\mathcal{G}_{\xi}\) has decreased by at least half compared with its value at the last restart point, i.e., $q=\frac{1}{2}$. If no restart has ever occurred, the comparison is made with the value of \(\mathcal{G}_{\xi}\) at 
\(k=50\). Similarly, for \texttt{rAPDB-xy-ada} in both monotone ($c_{\rm nm}=0$) and non-monotone ($c_{\rm nm}=1$) cases, we first compute \(\mathcal{G}_{\xi}\) at iteration \(k=100\) without restarting, and then evaluate the condition once every \(500\) iterations to check whether \(\mathcal{G}_{\xi}\) has decreased by half, i.e., $q=\frac{1}{2}$. 
If no restart has occurred, the comparison is made with the value of \(\mathcal{G}_{\xi}\) at 
\(k=100\). 

\begin{table}[t]
\centering
\begin{threeparttable}
\caption{Iteration counts and wall-clock time until termination on randomly generated QCQP instances with $m=10$ and $n=1000$.}
\label{tab:timecomparison_random_qcqp}
\setlength{\tabcolsep}{8pt}
\renewcommand{\arraystretch}{1.15}
\begin{tabular}{lcccc}
\toprule
Method
& Iter
& \makecell{CPU\\Time (s)}
& \makecell{GPU\\Time (s)}
& \makecell{GPU/CPU\\Speedup}
\\
\midrule
MOSEK
& -- & 32 & -- & -- \\
\midrule
EGM
& 44392 & 352 & 110 & 3$\times$ \\

\midrule

(mono) \texttt{rAPDB-xy}
& 20686 & 282 & 60 & 5$\times$ \\

(mono) \texttt{rAPDB-yx}
& 4609 & 70 & 19 & 4$\times$ \\

(mono) \texttt{rAPDB-xy-ada}
& 28405 & 379 & 91 & 4$\times$ \\

(mono) \texttt{rAPDB-yx-ada}
& 10319 & 152 & 42 & 4$\times$ \\

(mono) \texttt{APDB-xy}
& \underline{50000}\tnote{a} & 647 & 140 & 5$\times$ \\

(mono) \texttt{APDB-yx}
& 34947 & 520 & 131 & 4$\times$ \\

\midrule

(non-mono) \texttt{rAPDB-xy}
& 3008 & 87 & 17 & 5$\times$ \\

(non-mono) \texttt{rAPDB-yx}
& \cellcolor{blue!5}873
& 23
& \cellcolor{blue!5}6
& 4$\times$ \\

(non-mono) \texttt{rAPDB-xy-ada}
& 3040 & 79 & 18 & 4$\times$ \\

(non-mono) \texttt{rAPDB-yx-ada}
& \cellcolor{blue!5}877
& 25 & 7 & 4$\times$ \\

(non-mono) \texttt{APDB-xy}
& 3026 & 82 & 17 & 5$\times$ \\

(non-mono) \texttt{APDB-yx}
& \cellcolor{blue!5}871
& 24
& \cellcolor{blue!5}6
& 4$\times$ \\

\bottomrule
\end{tabular}
\begin{tablenotes}
\footnotesize
\item[a]The method reaches the maximum iteration limit with the relative error $6\times 10^{-3}$.
\end{tablenotes}
\end{threeparttable}
\end{table}

We adopted the OSQP solver \cite{stellato2018osqp} to solve the strongly convex
QP subproblems arising in our methods employing adaptive restarts. We include these adaptive variants mainly to evaluate whether the theoretically motivated adaptive restart rule can reduce iteration counts; however, their wall-clock performance is affected by the cost of \(\mathcal{G}_{\xi}\) evaluations and the strongly convex QP subproblem solves. In the implementation of restart-type algorithms, namely \texttt{rAPDB-xy}, \texttt{rAPDB-yx}, \texttt{rAPDB-xy-ada}, and \texttt{rAPDB-yx-ada}, we use the \emph{last iterate} instead of the weighted average iterate (cf. \texttt{line}~\ref{algrapdb-line4} in Algorithm~\ref{alg:pdhg-qcqp}), see, e.g., \cite[Section 3.4]{lu2025cupdlp}. This choice leads to better empirical performance.

\paragraph{Experiment results} Figure~\ref{fig:randqcqp} reports the relative suboptimality and infeasibility 
of the tested methods against the total number of gradient evaluations
\((\nabla_x \Phi\) and \(\nabla_y \Phi)\), where each curve is averaged over the four random QCQP instances.
\Cref{tab:timecomparison_random_qcqp} summarizes the 
iteration counts, CPU wall-clock time, GPU wall-clock time, and GPU-over-CPU speedups.
Several observations can be made. First, the non-monotone step-size search substantially improves the performance of APDB-type methods: compared with their monotone counterparts, it leads to much smaller iteration counts and wall-clock times. In the monotone setting, even GPU-accelerated first-order methods (except for \rapdbyx) are slower than MOSEK, whereas in the non-monotone setting the APDB-type methods become competitive with MOSEK, as their GPU runtime is better than the runtime of MOSEK. Second, among APDB-type methods, the yx-variants consistently outperform the corresponding xy-variants, suggesting that updating the dual variable before the primal variable is advantageous for this class of randomly generated QCQPs. Third, restarts significantly accelerate APDB in the monotone setting, e.g.,  \texttt{rAPDB-yx} reduces the iteration count from \(34947\) to \(4609\). In the non-monotone setting, however, the baseline APDB method is already very effective, and the benefit of restarts becomes less pronounced. In addition, since adaptive restarts require solving auxiliary subproblems, their wall-clock-time gains can be less significant than those of fixed-frequency restarts. Finally, GPU implementation yields a consistent acceleration, with speedups mostly between \(3\times\) and \(5\times\).

\begin{table}[!htbp]
\centering
\captionsetup{font=footnotesize,skip=3pt}
\caption{Iteration counts and wall-clock time for 
KML with
$n=4000$ ($\varepsilon=10^{-7}$) and $n=10000$ ($\varepsilon=10^{-6}$).}
\label{tab:timecomparison_kml}

\resizebox{\textwidth}{!}{%
\begin{threeparttable}
\setlength{\tabcolsep}{6pt}
\renewcommand{\arraystretch}{1.15}

\begin{tabular}{lcccccccc}
\toprule
& \multicolumn{4}{c}{$n=4000$}
& \multicolumn{4}{c}{$n=10000$} \\
\cmidrule(lr){2-5}
\cmidrule(lr){6-9}

Method
& Iter
& \makecell{CPU\\Time (s)}
& \makecell{GPU\\Time (s)}
& \makecell{GPU/CPU\\Speedup}
& Iter
& \makecell{CPU\\Time (s)}
& \makecell{GPU\\Time (s)}
& \makecell{GPU/CPU\\Speedup}
\\
\midrule

MOSEK
& -- & 173 & -- & --
& -- & 1685 & -- & --
\\

\midrule

EGM
& \underline{9999}\tnote{a} & 658 & 451 & 1.5$\times$
& \underline{14999}\tnote{b} & 6207 & 3529 & 1.8$\times$
\\

\midrule

(mono) \texttt{rAPDB-yx}
& 4341 & 234 & 103 & 2.3$\times$
& 5028 & 2501 & 645 & 3.9$\times$
\\

(mono) \texttt{rAPDB-yx-ada}
& 4341 & 331 & 200 & 1.7$\times$
& 5028 & 5771 & 3263 & 1.8$\times$
\\

(mono) \texttt{APDB-yx}
& 9236 & 498 & 219 & 2.3$\times$
& 10957 & 5332 & 1420 & 3.8$\times$
\\

\midrule

(non-mono) \texttt{rAPDB-yx}
& \cellcolor{blue!5}234
& 24
& \cellcolor{blue!5}10
& 2.3$\times$
& \cellcolor{blue!5}291
& 279
& \cellcolor{blue!5}70
& 4.0$\times$
\\

(non-mono) \texttt{rAPDB-yx-ada}
& \cellcolor{blue!5}234
& 62
& 48
& 1.3$\times$
& \cellcolor{blue!5}291
& 1365
& 940
& 1.5$\times$
\\

(non-mono) \texttt{APDB-yx}
& 340
& 36
& 16
& 2.3$\times$
& 451
& 444
& 112
& 4.0$\times$
\\

\bottomrule
\end{tabular}

\begin{tablenotes}
\footnotesize
\item[a] The method reaches the maximum iteration limit with relative error
$5.2\times10^{-4}$ for $n=4000$.
\item[b] The method reaches the maximum iteration limit with relative error
$6.2\times10^{-3}$ for $n=10000$.
\end{tablenotes}

\end{threeparttable}
}
\end{table}
\subsection{Kernel matrix learning~(KML): $\ell_2$-norm soft margin SVM}
Given a training set consisting of feature vectors
$\{a_i\}_{i=1}^n \subset \mathbb{R}^d$
and the corresponding labels 
$\{b_i\}_{i=1}^n \subset \{-1,+1\}$,
consider $m \in \mathbb{Z}_{+}$ different embeddings of the data and let
$K_i \in \mathbb{S}_+^n$
be the corresponding kernel matrix for $i\in [m]$.
The objective of
KML is to learn a 
matrix $K$ belonging to a class of kernel matrices
$\mathcal{K} \subset \mathbb{S}_+^n$
such that it minimizes the training error of an $\ell_2$-norm soft-margin SVM over
$K \in \mathcal{K}$ --- see \cite{lanckriet2004learning} for more details.
In this setting, it is assumed that the class $\mathcal{K}$ is described as a convex set generated by
$\{K_i\}_{i=1}^m$, 
$\mathcal{K}
\triangleq
\left\{
\sum_{i=1}^m y_i K_i
:\;
y_i \ge 0,\ i\in [m]
\right\}
\subset \mathbb{S}_+^n$.
Then, learning over the class $\mathcal{K}$ is formulated as
$\min_{y \in \mathbb{R}_+^m}
\max_{x \in \mathbb{R}_+^n}
\{
2x^\top 1_n
-
\sum_{i=1}^m y_i\, x^\top H(K_i)x
-
\lambda \|x\|_2^2:~\langle r,y\rangle = c, \langle b,x\rangle = 0\}$,
where $c>0$ and $\lambda \ge 0$ are model parameters,
$y = [y_i]_{i\in [m]}$,
$r = [\operatorname{trace}(K_i)]_{i\in [m]}$,
$b = [b_j]_{j\in [n]}$,
and $H(K_i) \triangleq \operatorname{diag}(b)\, K_i\, \operatorname{diag}(b)$ for $i\in[m]$. \sa{We consider three kernel functions ($m=3$): polynomial kernel function $\phi_1(a, a') = (1 + a^\top a')^2$,
Gaussian kernel function
$\phi_2(a,a') = \exp \bigl(-0.5(a- a')^\top (a- a')/0.1\bigr)$,
and the linear kernel function
$\phi_3(a,a') = a^\top a'$
to compute $K_1,K_2,K_3 \in \mathbb{S}_+^n$, respectively, i.e., for $i\in\{1,2,3\}$, $[K_i]_{k_1 k_2}= \phi_i(a_{k_1}, a_{k_2})$ for $k_1,k_2\in [n]$, where $n$ denotes the number of data points.} We set $\lambda = 1$ and $c = \sum_{i=1}^3 r_i$, where $r_i = \operatorname{trace}(K_i)$ for $i=1,2,3$. The kernel matrices are normalized as in~\cite{lanckriet2004learning}; thus,
$\operatorname{diag}(K_i) = 1_n$ and $r_i = n$ for each $i=1,2,3$.
Note that 
KML is equivalent to
\vspace{-2mm}
{\small
\begin{equation} \label{eq:kml_pro1} \tag{min-max-KML}
\min_{x\in \mathbb{R}_+^n}
\max_{y \in \mathbb{R}_+^m}
\{\lambda \|x\|_2^2 - 2x^\top 1_n + \sum_{i=1}^m \frac{c}{r_i}\, y_i\, x^\top H(K_i)x:\ \langle b,x\rangle = 0,~\langle 1_m,y\rangle = 1\},
\vspace*{-2mm}
\end{equation}}%
which is a special case of minimax optimization problems with the strong convexity parameter $\mu=2$ --since we set $\lambda=1$. In \cite{lanckriet2004learning}, \eqref{eq:kml_pro1} is shown to be equivalently represented as a QCQP.

\begin{figure}[!htbp]
\centering
\begin{subfigure}[t]{0.85\textwidth}
    \centering
    \includegraphics[width=\linewidth]{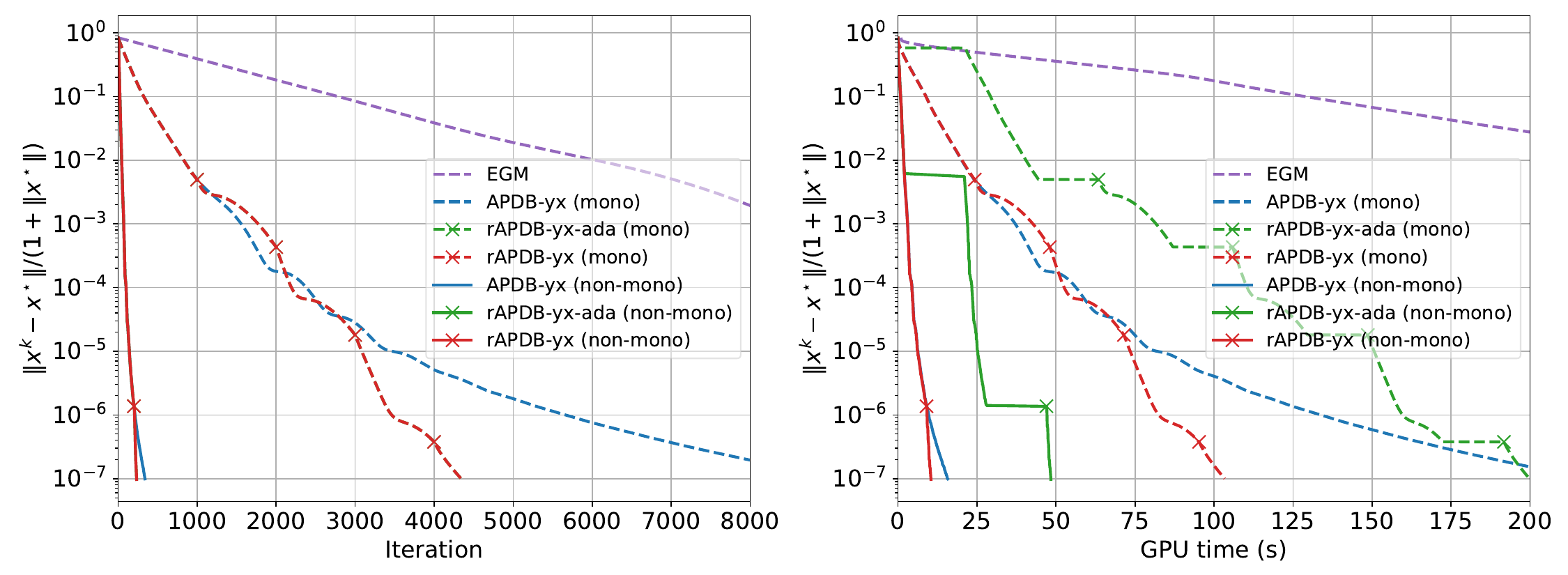}
    \caption{$n=4{,}000$.}
    \label{fig:kml_n4000}
\end{subfigure}
\hfill
\begin{subfigure}[t]{0.85\textwidth}
    \centering
    \includegraphics[width=\linewidth]{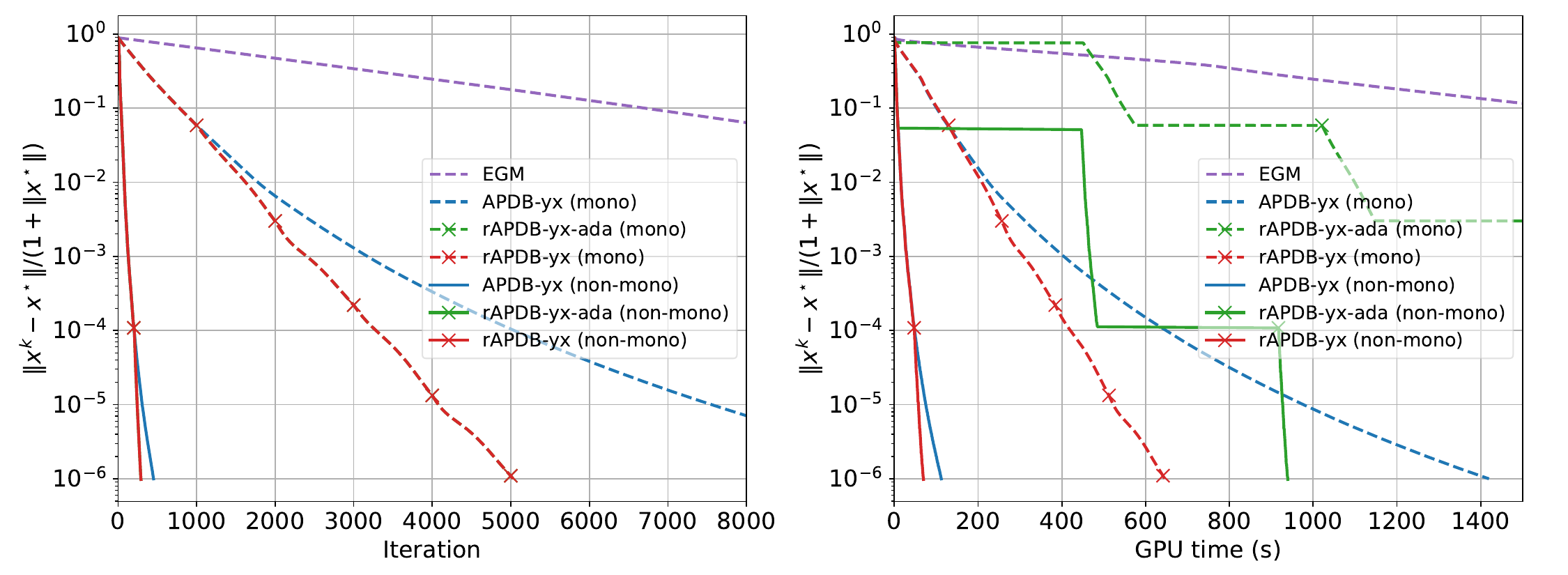}
    \caption{$n=10{,}000$.}
    \label{fig:kml_n10000}
\end{subfigure}
\vspace*{-3mm}
\caption{\small
Convergence performance of \rapdb, \rapdbada, \texttt{APDB}, and
\texttt{EGM} on the kernel matrix learning problem instances, measured
by the relative error against iteration count and GPU time.}
\label{fig:kml}
\vspace*{-5mm}
\end{figure}

\paragraph{Implementation details} We consider the data set \texttt{SIDO0} \cite{guyon2008sido} \sa{and we conducted experiments with two different problem sizes, i.e., 
we sampled $n=4,000$ and $n=10,000$ data points from the same dataset.} 
We conduct both experiments on the normalized data such that each feature column is mean-centered and divided by its standard deviation. 

We use MOSEK to compute the optimal solution $x^\star$, which is 
then used to calculate the relative error ${\| x^k - x^\star\| }/{(1 + \| x^\star\|)}$ to monitor the performance of EGM, \texttt{APDB}, \texttt{rAPDB}, and \texttt{rAPDB-ada}. Here, we only employ the yx-variants since they perform better than the xy-variants. We set the maximal iteration to be 9999 (resp. 14999) and terminate algorithms if the relative error ${\| x^k - x^\star\| }/{(1 + \| x^\star\|)}$ is less than $10^{-7}$ (resp. $10^{-6}$) for $n=4000$ (resp. $n=10000$). For EGM, we tune the step sizes and report the best-performing results. The backtracking parameter is fixed at \(\eta=0.7\) for all \texttt{APDB}-type methods. For variants of \texttt{APDB-yx}-type methods, we do not impose the set \(\mathcal{B}\) in \eqref{eq: apdb-yx-y} and choose \(c_\alpha = 0.4\) and \(\delta = 0.5\). We test both monotone and non-monotone step-size update strategies. The restart frequencies for \texttt{rAPDB-yx} (mono) and \texttt{rAPDB-yx} (non-mono) are set to \(1000\), and \(200\), respectively. For the adaptive restart variants of \texttt{APDB}, we fix \(\xi = 0.04\). For \texttt{rAPDB-yx-ada}, we first compute \(\mathcal{G}_{\xi}\) at iteration \(k=50\) without restarting. We then check every \(200\) iterations in the non-monotone case, and every \(1000\) iterations in the monotone case, whether \(\mathcal{G}_{\xi}\) has decreased by at least half compared with its value at the last restart point. If no restart has occurred, the comparison is made with the value of \(\mathcal{G}_{\xi}\) at iteration \(k=50\). We adopted the OSQP solver \cite{stellato2018osqp} to solve the strongly convex QP subproblems arising in adaptive restarts. In the implementation of restart-type algorithms, namely \texttt{rAPDB-yx} and \texttt{rAPDB-yx-ada}, we use the \emph{last iterate} to restart the algorithms.

\paragraph{Experiment results} \cref{fig:kml_n4000,fig:kml_n10000} display the relative error 
against the iteration count and GPU time for the kernel matrix learning problem  instances with \(n=4000\) and \(n=10000\), respectively.
\Cref{tab:timecomparison_kml} summarizes the corresponding iteration counts, CPU wall-clock time, GPU wall-clock time, and GPU-over-CPU speedups.
The results are consistent across the two problem sizes. First, EGM is substantially slower than the APDB-type methods and reaches the maximum iteration limit in both cases, indicating that EGM is less effective for these KML instances. Second, the non-monotone step-size search dramatically improves the performance of APDB-type methods, e.g., for \(n=4000\), \texttt{rAPDB-yx} decreases from \(4341\) iterations in the monotone setting to \(234\) iterations in the non-monotone setting; for \(n=10000\), it decreases from \(5028\) to \(291\) iterations. Third, fixed-frequency restarts provide a clear acceleration over APDB in both the monotone and non-monotone settings. In particular, in the non-monotone setting, \texttt{rAPDB-yx} reduces the iteration count from \(340\) for \texttt{APDB-yx} to \(234\) when \(n=4000\), and from \(451\) for \texttt{APDB-yx} to \(291\) when \(n=10000\). 
Adaptive restarts achieve the same iteration counts but incur additional overhead from computing the smoothed duality gap and solving the auxiliary quadratic subproblems. Consequently, \texttt{rAPDB-yx} with non-monotone step-size search gives the best overall wall-clock performance. Finally, GPU acceleration is effective for the APDB-type methods, especially for \apdbyx~and \rapdbyx, with speedups around \(2\times\)--\(4\times\).

\section{Concluding Remarks}

We developed restarted accelerated primal-dual algorithms with adaptive
stepsize search for convex nonlinear conic programs, with convex QCQPs as a
special case. The proposed framework equips APDB with fixed-frequency and
adaptive restarts, and supports both monotone and non-monotone stepsize search
strategies. Since the methods rely only on first-order information,
matrix-vector products, and simple projection-type operations, they are well
suited for large-scale and GPU-accelerated computation.

Under metric subregularity of the KKT map, we established a quadratic growth
property for a self-centered smoothed duality gap and used it to prove global
linear convergence of the restarted methods. We also provided sufficient
conditions for metric subregularity for nonlinear programs over polyhedral
cones, without requiring strict complementarity or multiplier uniqueness. Numerical experiments on random QCQPs and kernel matrix learning instances
demonstrate the effectiveness of the proposed methods. Future work includes designing cheaper
adaptive restart criteria and incorporating other practical acceleration techniques.

\bibliographystyle{plain}
\bibliography{references}

\section{Appendix}
\subsection{Proof of \Cref{thm:jongshi}}
\label{sec:jongshi-proof}
Since \(\Kcone\subset\reals^m\) is polyhedral, there exists \(A\in\R^{\bar m\times m}\) for some $\bar m\in \integers_+$ such that
\begin{equation}
\label{eq:cone-representation}
\Kcone=\{y\in\R^m:Ay\ge0\};
\end{equation}
therefore, for any $x\in\reals^n$,
\begin{equation}
\label{eq:conic-to-scalar}
g(x)\in-\Kcone
\quad\Longleftrightarrow\quad
Ag(x)\le0.
\end{equation}
Define $G(x)\triangleq Ag(x)$.
Thus, we obtain another representation of the feasible set:
\begin{equation*}
\feas=\{x\in \cX:\ h(x)=0,\ G(x)\le0\}.
\end{equation*}
The dual multiplier for the componentwise inequality system \(G(x)\le0\) will be denoted by \(\rho\ge0\), and the corresponding conic multiplier is
$\lambda=A^\top \rho$.
Indeed, the dual cone can be explicitly represented as
\begin{equation}
\label{eq:dual-cone-representation}
\Kcone^*=\{A^\top \rho:\rho\ge0\}.
\end{equation}

We first verify the boundedness of \(\cM(\bar x)\). Suppose, to the contrary, that \(\cM(\bar x)\) is unbounded. Then there exists \(\{(v^r,\lambda^r)\}_{r\in \integers_+}\subset \cM(\bar x)\) such that as $r\to\infty$, 
\begin{equation*}
\tau_r\triangleq\| (v^r,\lambda^r)\|\to\infty.
\end{equation*}
For each \(r\in\integers_+\), define
\begin{equation}
\label{eq:sr-def-boundedness}
s^r\triangleq-\nabla_x\Phi(\bar x,v^r,\lambda^r)\in \cN_{\cX}(\bar x).
\end{equation}
\sa{Note that $\{(v^r,\lambda^r)/\tau_r\}$ is a bounded sequence, which implies that $\{s^r/\tau_r\}$ is also a bounded sequence; indeed, \eqref{eq:sr-def-boundedness} and \eqref{eq:L-def} imply that
$s^r/\tau_r
= -\nabla f(\bar x)/\tau_r
-\nabla h(\bar x)^\top v^r/\tau_r
-\nabla g(\bar x)^\top \lambda^r/\tau_r$.
Hence,} after dividing \eqref{eq:sr-def-boundedness} by \(\tau_r\) and passing to a subsequence, we may assume
\begin{equation*}
\frac{v^r}{\tau_r}\to \bar v,
\qquad
\frac{\lambda^r}{\tau_r}\to \bar\lambda,
\qquad
\frac{s^r}{\tau_r}\to \bar s,\quad\mbox{as}\quad r\to\infty;
\end{equation*}
moreover, we also have
\begin{equation}
\label{eq:normalized-nonzero}
\|(\bar v,\bar\lambda)\|=1.
\end{equation}
Since \(\cN_{\cX}(\bar x)\) and \(\Kcone^*\) are closed cones, and since \(\langle \lambda^r,g(\bar x)\rangle=0\), we get
\begin{equation*}
\bar s\in \cN_{\cX}(\bar x),
\qquad
\bar\lambda\in\Kcone^*,
\qquad
\langle \bar\lambda,g(\bar x)\rangle=0.
\end{equation*}
Dividing the identity
\begin{equation*}
\nabla f(\bar x)+\nabla h(\bar x)^\top  v^r+\nabla g(\bar x)^\top \lambda^r+s^r=0
\end{equation*}
by \(\tau_r\), and passing to the limit, gives
\begin{equation}
\label{eq:normalized-stationarity}
\nabla h(\bar x)^\top \bar v+
\nabla g(\bar x)^\top \bar\lambda+\bar s=0.
\end{equation}
Let \(d\in \cT_{\cX}(\bar x)\) be the direction from \eqref{eq:relative-mfcq-direction}; therefore, using \(\nabla h(\bar x)d=0\), \eqref{eq:normalized-stationarity} implies that
\begin{equation}
\label{eq:normalized-pairing}
0
=
\langle \bar\lambda,\nabla g(\bar x)d\rangle+
\langle \bar s,d\rangle.
\end{equation}
Since \(\bar s\in \cN_{\cX}(\bar x)\) and \(d\in \cT_{\cX}(\bar x)\), we have
\begin{equation}
\label{eq:sbar-d-nonpositive}
\langle \bar s,d\rangle\le0.
\end{equation}
Furthermore, according to \eqref{eq:relative-mfcq-direction}, we have
$y\triangleq g(\bar x)+\nabla g(\bar x)d\in -\operatorname{int}(\Kcone)$.
Thus, from \(\langle \bar\lambda,g(\bar x)\rangle=0\),
\begin{equation}
\label{eq:lambda-gradg-pairing}
\langle \bar\lambda,\nabla g(\bar x)d\rangle
=
\langle \bar\lambda,y\rangle.
\end{equation}
If \(\bar\lambda\ne0\), then \(\bar\lambda\in\Kcone^*\setminus\{0\}\) and \(y\in-\operatorname{int}(\Kcone)\), so
\begin{equation}
\label{eq:strict-negative-lambda}
\langle \bar\lambda,y\rangle<0.
\end{equation}
Indeed, every nonzero vector in \(\Kcone^*\) is strictly positive on \(\operatorname{int}(\Kcone)\); hence it is strictly negative on \(-\operatorname{int}(\Kcone)\). Combining \eqref{eq:normalized-pairing}, \eqref{eq:sbar-d-nonpositive}, \eqref{eq:lambda-gradg-pairing}, and \eqref{eq:strict-negative-lambda} yields a contradiction unless \(\bar\lambda=0\). Hence, \(\bar\lambda=0\). Then, \eqref{eq:normalized-stationarity} reduces to
\begin{equation*}
\nabla h(\bar x)^\top \bar v+\bar s=0,
\qquad
\bar s\in \cN_{\cX}(\bar x).
\end{equation*}
By the 
nondegeneracy condition in \eqref{eq:relative-equality-nondegeneracy}, \(\bar v=0\) and \(\bar s=0\). This contradicts \eqref{eq:normalized-nonzero}; therefore, \(\cM(\bar x)\) is bounded. Next, we prove the desired result by contradiction using the same proof idea in \cite[Proposition 6.2.7]{facchinei2003finite}.

\textbf{Contradiction sequence.}
Suppose the desired error bound is false. Then, there exist sequences $\{x^k\}\subset\cX$, $\{v^k\}\subset\reals^p$ and $\{\lambda^k\}\subset\cK^*$ such that $x^k\to\bar x$ as $k\to\infty$, and that
\begin{equation*}
t_k\to0,
\qquad
\frac{r(x^k,v^k,\lambda^k)}{t_k}\to 0,
\end{equation*}
where 
\begin{equation}
\label{eq:tk-def}
t_k\triangleq\|x^k-\bar x\|+\dist((v^k,\lambda^k),\cM(\bar x)),\quad\forall~k\in\integers_+.
\end{equation}
Thus, from the definition of the residual function $r(\cdot)$, we get
\begin{equation*}
\frac{\dist\bigl(0,\nabla_x\Phi(x^k,v^k,\lambda^k)+\cN_{\cX}(x^k)\bigr)}{t_k}\to0,\qquad
\frac{\|h(x^k)\|}{t_k}\to0,\qquad 
\frac{\dist\bigl(0,-g(x^k)+\cN_{\cK^*}(\lambda^k)\bigr)}{t_k}\to0.
\end{equation*}
The boundedness of \(\cM(\bar x)\) and $\dist((v^k,\lambda^k),\cM(\bar x))\to0$ imply that the sequence \((v^k,\lambda^k)\) is bounded. 

\medskip
\noindent
\textbf{The choice of \(s^k\).}
For each \(k\in\integers_+\), choose \(s^k\in \cN_{\cX}(x^k)\) such that
\begin{equation*}
s^k=\operatorname*{argmin}_{s\in \cN_{\cX}(x^k)}
\|\nabla_x\Phi(x^k,v^k,\lambda^k)+s\|.
\end{equation*}
The unique minimizer exists because \(\cN_{\cX}(x^k)\) is a closed convex cone. Define
\begin{equation*}
u^k\triangleq\nabla_x\Phi(x^k,v^k,\lambda^k)+s^k.
\end{equation*}
Then our choice of $s^k$ implies that
\begin{equation*}
\|u^k\|=
\dist\bigl(0,\nabla_x\Phi(x^k,v^k,\lambda^k)+\cN_{\cX}(x^k)\bigr).
\end{equation*}
Moreover, define $w^k\triangleq h(x^k)$.
Since the residual $r(x^k,v^k,\lambda^k)=o(t_k)$, we get
\begin{equation}
\label{eq:first-residual-asymptotic}
\|u^k\|+\|w^k\|+
\dist\bigl(0,-g(x^k)+\cN_{\cK^*}(\lambda^k)\bigr)
=o(t_k).
\end{equation}

\medskip
\noindent
\textbf{The conic residual vs. the componentwise inequality residual.}
We next justify that the conic residual controls the complementarity residual for a componentwise inequality system. 
Let $\delta:\cX\times\reals^m\to\reals\cup\{+\infty\}$ such that $\delta(x,\lambda)=+\infty$ if \(\cN_{\cK^*}(\lambda)=\emptyset\); otherwise,
\begin{equation*}
\delta(x,\lambda)\triangleq\dist\bigl(0,-g(x)+\cN_{\cK^*}(\lambda)\bigr).
\end{equation*}
If \(\delta(x,\lambda)<\infty\), then \(\cN_{\cK^*}(\lambda)\ne\emptyset\), and hence \(\lambda\in\Kcone^*\). Therefore, by \eqref{eq:dual-cone-representation}, there exists \(\rho\ge0\) such that
\begin{equation}
\label{eq:rho-rep}
A^\top \rho=\lambda.
\end{equation}
Let \(z\in \cN_{\cK^*}(\lambda)\) be such that
\begin{equation}
\label{eq:z-minimizer}
\delta(x,\lambda)=\|-g(x)+z\|.
\end{equation}
For a closed convex cone \(\Kcone^*\),
\begin{equation}
\label{eq:normal-dual-cone}
\cN_{\cK^*}(\lambda)=(\Kcone^*)^\circ\cap\lambda^\perp,
\end{equation}
where we adopt the convention $0^\perp=\mathbb{R}^m$; hence, $\cN_{\cK^*}(0)=(\Kcone^*)^\circ$. Since \((\Kcone^*)^\circ=-\Kcone\), \eqref{eq:normal-dual-cone} implies
$z\in-\Kcone$ and $\langle z,\lambda\rangle=0$. By \eqref{eq:cone-representation}, \(z\in-\Kcone\) is equivalent to $Az\le0$. Moreover, using \eqref{eq:rho-rep}, we get
\begin{equation}
\label{eq:z-lambda-complementarity}
0=\langle z,\lambda\rangle
=
\langle z,A^\top \rho\rangle
=
\langle Az,\rho\rangle
=
\sum_{i\in [\bar m]}\rho_i(Az)_i.
\end{equation}
Since \(\rho_i\ge0\) and \((Az)_i\le0\) for all $i\in [\bar m]$, each term in \eqref{eq:z-lambda-complementarity} must be zero: $\rho_i(Az)_i=0$ for all $i\in [\bar m]$.
Therefore,
\begin{equation*}
\rho\ge0,
\qquad
-Az\ge0,
\qquad
\rho_i(-Az)_i=0\quad\forall i\in[\bar m];
\end{equation*}
hence, for \(z\in \cN_{\cK^*}(\lambda)\) satisfying \eqref{eq:z-minimizer}, we have
\begin{equation}
\label{eq:min-rho-Az-zero}
\min\{\rho,-Az\}=0.
\end{equation}
Now compare \(-Ag(x)\) and \(-Az\). Since 
\(\min\{\rho,\cdot\}\) is $1$-Lipschitz for every fixed \(\rho\in\reals^{\bar m}\), using \eqref{eq:min-rho-Az-zero},
\begin{equation*}
\|\min\{\rho,-Ag(x)\}\|
\le
\|A(z-g(x))\|
\le
\|A\|\,\|z-g(x)\|.
\end{equation*}
By \eqref{eq:z-minimizer}, $\|z-g(x)\|
=\delta(x,\lambda)$.
Thus,
\begin{equation*}
\|\min\{\rho,-Ag(x)\}\|
\le
\|A\|\,
\dist\bigl(0,-g(x)+\cN_{\cK^*}(\lambda)\bigr).
\end{equation*}

Next, we argue that for $\{\lambda^k\}$ bounded, we can choose \(\{\rho^k\}\subset\reals^{\bar m}_+\) bounded. Let \(B\triangleq A^\top \). The columns of \(B\) generate \(\Kcone^*\), i.e., $\Kcone^*=\operatorname{cone}\{b_1,\ldots,b_{\bar m}\}$ is the conic hull of $\{b_i\}_{i\in[\bar m]}\subset\reals^m$, where \(b_i\) denotes the \(i\)-th column of \(B\) for $i\in[\bar m]$. Conic Carath\'eodory's theorem implies that for any \(\lambda\in\Kcone^*\), there exists $I\subset [\bar m]$ such that \(\{b_i\}_{i\in I}\) is linearly independent and that for some $\rho_i\ge0$ for $i\in I$, $\lambda$ can be represented as $\lambda=\sum_{i\in I}\rho_i b_i$.
Let \(B_I\) denote the matrix with columns \(b_i\) for \(i\in I\), then
$\lambda=B_I\rho_I$ for some $\rho_I\geq 0$; therefore,
$\rho_I=B_I^\dagger\lambda$, where \(B_I^\dagger\) is the Moore--Penrose inverse, i.e., $B_I^\dagger=(B_I^\top B_I)^{-1}B_I^\top$. Hence,
\begin{equation*}
\|\rho_I\|\le \|B_I^\dagger\|\,\|\lambda\|.
\end{equation*}
There are only finitely many such index sets \(I\subseteq\{1,\ldots,\bar m\}\) such that \(B_I\) has linearly independent columns; consequently, we get
\begin{equation*}
\gamma\triangleq\max_{I\subset[\bar m]}\{\|B_I^\dagger\|:~\det(B_I^\top B_I)\neq 0\}<\infty.
\end{equation*}
Thus, for every \(\lambda\in\Kcone^*\), there exists \(\rho\ge0\) such that
$A^\top \rho=\lambda$ and $\|\rho\|\le \gamma\|\lambda\|$ --here, the important point is that $\gamma$ is independent of $\lambda$. 
Now given \((x^k,\lambda^k)\), we can choose \(\rho^k\ge0\) such that
\begin{equation}
\label{eq:rho-bound}
A^\top \rho^k=\lambda^k,
\qquad
\|\rho^k\|\le\gamma\|\lambda^k\|,
\end{equation}
and
\begin{equation}
\label{eq:vk-control-before-def}
\|\min\{\rho^k,-Ag(x^k)\}\|
\le
\|A\|\,
\dist\bigl(0,-g(x^k)+\cN_{\cK^*}(\lambda^k)\bigr),\qquad\forall~k\in\integers_+.
\end{equation}
Since \(\{\lambda^k\}\) is bounded, \eqref{eq:rho-bound} implies that \(\{\rho^k\}\) is bounded as well. Define
\begin{equation*}
\eta^k\triangleq\min\{\rho^k,-G(x^k)\},\qquad\forall~k\in\integers_+.
\end{equation*}
Then, using \eqref{eq:first-residual-asymptotic} and \eqref{eq:vk-control-before-def}, we get $\frac{\eta^k}{t_k}\to0$. Moreover, let
\begin{equation}
\label{eq:hat-u-k}
\widehat u^k\triangleq\nabla f(x^k)+\nabla h(x^k)^\top v^k+\nabla G(x^k)^\top \rho^k+s^k,\qquad\forall~k\in\integers_+.
\end{equation}
Since $\nabla G(x)^\top \rho=\nabla g(x)^\top A^\top \rho$, and \(A^\top \rho^k=\lambda^k\) for $k\in\integers_+$, we have
\begin{equation*}
\widehat u^k=\nabla_x\Phi(x^k,v^k,\lambda^k)+s^k=u^k,\qquad\forall~k\in\integers_+.
\end{equation*}
Hence, by \eqref{eq:first-residual-asymptotic},
\begin{equation}
\label{eq:uhat-w-v-small}
\frac{\widehat u^k}{t_k}=\frac{u^k}{t_k}\to0,
\qquad
\frac{w^k}{t_k}\to0,
\qquad
\frac{\eta^k}{t_k}\to0,\qquad\mbox{as}\quad k\to\infty.
\end{equation}

\medskip
\noindent
\textbf{Non-asymptotic complementarity system.}
For $k\in\integers_+$, define
\begin{equation}
\label{eq:tilde-rho-sigma}
\tilde\rho^k\triangleq\rho^k-\eta^k,\qquad \sigma^k\triangleq\widehat u^k-\nabla G(x^k)^\top \eta^k,
\end{equation}
where $\widehat{u}^k$ is defined in \eqref{eq:hat-u-k}. The sequence $\{\tilde\rho^k\}$ satisfies
\begin{equation}
\label{eq:tilderho-complementarity}
0\le \tilde\rho^k\perp G(x^k)+\eta^k\le0,\qquad\forall~k\in\integers_+.
\end{equation}
Indeed, note that $\tilde\rho^k=\max\{0,\rho^k+G(x^k)\}\geq 0$ and $G(x^k)+\eta^k=\min\{0,\rho^k+G(x^k)\}\leq 0$; therefore, we can conclude that $\tilde\rho^k\perp G(x^k)+\eta^k$. Furthermore, 
\eqref{eq:hat-u-k} and \eqref{eq:tilde-rho-sigma} together imply that
\begin{equation}
\label{eq:corrected-stationarity}
\nabla f(x^k)-\sigma^k
+
\nabla h(x^k)^\top v^k
+
\nabla G(x^k)^\top \tilde\rho^k
+s^k=0,\qquad\forall~k\in\integers_+.
\end{equation}
Since \(\{\nabla G(x^k)\}_{k\in\integers_+}\) is bounded, \(\eta^k=o(t_k)\), and \(\widehat u^k=o(t_k)\), we also get
\begin{equation}
\label{eq:sigma-small}
\sigma^k=o(t_k).
\end{equation}
Moreover, as shown earlier, the multiplier sequence \(\{(v^k,\lambda^k)\}\) is bounded, \(\{\rho^k\}\) is bounded, and \(\eta^k\to0\); therefore, it follows that the sequence \(\{(v^k,\tilde\rho^k)\}\) is bounded as well. Moreover, \eqref{eq:corrected-stationarity} implies that \(\{s^k\}\) is bounded, where we used the boundedness of \(\{(v^k,\tilde\rho^k)\}\), the continuity of \(\nabla f\), \(\nabla h\), \(\nabla G\), and \(\sigma^k\to0\). Thus, we can conclude that there exists a \textit{bounded} set \(S\) such that
\begin{equation}
\label{eq:bounded-S}
(v^k,\tilde\rho^k,s^k)\in S
\qquad \forall k\in\integers_+.
\end{equation}

\medskip
\noindent
\textbf{Inactive constraints of the componentwise inequality system.}
Let $I_G(\bar x)\subseteq [\bar m]$ denote the active set for the componentwise inequality system, i.e.,
\begin{equation*}
I_G(\bar x)\triangleq\{i\in[\bar m]:G_i(\bar x)=0\}.
\end{equation*}
Fix arbitrary \(i\notin I_G(\bar x)\). By definition, \(G_i(\bar x)<0\), which implies that for all sufficiently large \(k\in\integers_+\), $-G_i(x^k)\ge \eta_i>0$ for some \(\eta_i>0\). Moreover, since
\begin{equation*}
\eta_i^k=\min\{\rho_i^k,-G_i(x^k)\}\to0
\end{equation*}
and \(\rho_i^k\ge0\), it follows that \(\rho_i^k\to0\). Therefore, for all sufficiently large \(k\), we have $\rho_i^k<-G_i(x^k)$, which yields
$\eta_i^k=\rho_i^k$ and $\tilde\rho_i^k=\rho_i^k-\eta_i^k=0$. Thus, after discarding finitely many elements from the sequence, we can safely assume that
\begin{equation}
\label{eq:tilderho-inactive-zero}
\tilde\rho_i^k=0
\qquad
\forall i\notin I_G(\bar x),\quad\forall k\in\integers_+.
\end{equation}

\medskip
\noindent
\textbf{Controlling the distance to the multiplier set.}
\sa{The proof of \cite[Proposition 6.2.7]{facchinei2003finite} invokes \cite[Corollary 3.2.5]{facchinei2003finite} to have a control on the distance of $(v^k,\tilde\rho^k)$ to the multiplier set $\cM(\bar x)$; however, due to conic constraints and $\cX\neq\reals^n$ in our setting, we cannot directly use the same argument. That is why we first invoke \cite[Corollary 3.2.5]{facchinei2003finite} for a lifted polyhedral set and then project back. This is the extra step needed to be able to provide guarantees in terms of the non-lifted multiplier set \(\cM(\bar x)\).}

\smallskip
\noindent
\textbf{(a.) The lifted set.}
Define
\begin{equation*}
\cM_{A}^{\rm lift}(\bar x)\triangleq
\left\{(v,\rho,s):
\begin{array}{l}
 s\in \cN_{\cX}(\bar x),\\[0.2em]
 \nabla f(\bar x)+\nabla h(\bar x)^\top  v+\nabla G(\bar x)^\top \rho+s=0,\\[0.2em]
 \rho\ge0,
 \quad
 \rho_i=0\ \text{if }i\notin I_G(\bar x)
\end{array}
\right\}.
\end{equation*}
Since \(\cX\) is polyhedral, \(\cN_{\cX}(\bar x)\) is a polyhedral cone, so \(\cM_{A}^{\rm lift}(\bar x)\) is a polyhedron. We also define the non-lifted multiplier set corresponding to the componentwise inequality system, $G(x)\leq 0$, as follows:
\begin{equation*}
\cM_A(\bar x)\triangleq
\left\{(v,\rho):
\begin{array}{l}
 \rho\ge0,
 \quad \rho_i=0\ \text{if }i\notin I_G(\bar x),\\[0.2em]
 \nabla f(\bar x)+\nabla h(\bar x)^\top v+\nabla G(\bar x)^\top \rho\in -\cN_{\cX}(\bar x)
\end{array}
\right\}.
\end{equation*}
Next, we consider
\begin{equation}
\label{eq:projection-map}
T':(v,\rho)\mapsto(v,A^\top \rho)
\end{equation}
that maps \(\cM_A(\bar x)\) onto \(\cM(\bar x)\). Indeed, if \((v,\rho)\in \cM_A(\bar x)\), then \(\lambda=A^\top \rho\in\Kcone^*\),
\begin{equation*}
\nabla_x\Phi(\bar x,v,\lambda)=\nabla f(\bar x)+\nabla h(\bar x)^\top v+\nabla G(\bar x)^\top \rho\in -\cN_{\cX}(\bar x),
\end{equation*}
and
\begin{equation*}
\langle \lambda,g(\bar x)\rangle
=
\langle A^\top \rho,g(\bar x)\rangle
=
\langle \rho,G(\bar x)\rangle=0.
\end{equation*}
Thus \((v,\lambda)\in \cM(\bar x)\). Conversely, if \((v,\lambda)\in \cM(\bar x)\), then \(\lambda\in\Kcone^*\), so \(\lambda=A^\top \rho\) for some \(\rho\ge0\). Since \(G(\bar x)=Ag(\bar x)\le0\), \(\rho\ge0\) and $0=\langle \lambda,g(\bar x)\rangle
=
\langle \rho,G(\bar x)\rangle$,
each component satisfies $\rho_iG_i(\bar x)=0$ for $i\in[\bar m]$; hence, \(\rho_i=0\) for all \(i\notin I_G(\bar x)\). Therefore \((v,\rho)\in \cM_A(\bar x)\). We can conclude that
\begin{align}
\label{eq:Tp-map}
    T'(\cM_A(\bar x))=\cM(\bar x).
\end{align}

\smallskip
\noindent
To invoke \cite[Corollary 3.2.5]{facchinei2003finite} on the lifted set $\cM_{A}^{\rm lift}(\bar x)$, we need a simple consequence of the polyhedrality of \(\cX\). Since \(\cX\) is polyhedral and \(x^k\to\bar x\), there exists $K\in\integers_+$ such that for all $k\geq K$, $I_{\cX}(x^k) \subset I_{\cX}(\bar x)$, where $I_{\cX}(x)$ denotes the index set of active constraints, defining the polyhedron $\cX$, at $x\in \cX$; hence, 
one can conclude that
\begin{equation}
\label{eq:normal-inclusion-before-corollary}
\cN_{\cX}(x^k)\subseteq \cN_{\cX}(\bar x)
\qquad k\geq K.
\end{equation}
Therefore, \(s^k\in \cN_{\cX}(\bar x)\) for all sufficiently large \(k\in\integers_+\), i.e., $k\geq K$.

\smallskip
\noindent
\textbf{(b.) The polyhedral perturbation estimate.}
After discarding finitely many terms, \eqref{eq:bounded-S}, \eqref{eq:normal-inclusion-before-corollary}, and \eqref{eq:tilderho-inactive-zero} show that the sequence \(\{(v^k,\tilde\rho^k,s^k)\}\) lies in a fixed bounded set $S$ and satisfies the following two conditions for all $k\in\integers_+$: \(s^k\in \cN_{\cX}(\bar x)\) and \(\tilde\rho_i^k=0\) for every inactive index \(i\notin I_G(\bar x)\). Therefore, we may apply \cite[Corollary 3.2.5(b)]{facchinei2003finite} to the lifted polyhedral set \(\cM_{A}^{\rm lift}(\bar x)\). Using \eqref{eq:corrected-stationarity}, there exists \(L_0>0\), independent of \(k\), such that
\begin{equation*}
\begin{aligned}
&\dist\bigl((v^k,\tilde\rho^k,s^k),\cM_{A}^{\rm lift}(\bar x)\bigr)
\\
&\quad\le
L_0\Big[
\|\nabla f(\bar x)-\nabla f(x^k)+\sigma^k\|
+
\sum_{j\in[p]}\|\nabla h_j(x^k)-\nabla h_j(\bar x)\|
+
\sum_{i\in I_G(\bar x)}\|\nabla G_i(x^k)-\nabla G_i(\bar x)\|
\Big].
\end{aligned}
\end{equation*}
Since \(\nabla f\), \(\nabla h\), and \(\nabla G\) are 
continuously differentiable near \(\bar x\), and since \(\sigma^k=o(t_k)\), the right-hand side 
above is $\cO(\|x^k-\bar x\|)+o(t_k)$. Thus, there exists \(C_1>0\) such that
\begin{equation}
\label{eq:lifted-distance-bound}
\dist\bigl((v^k,\tilde\rho^k,s^k),\cM_{A}^{\rm lift}(\bar x)\bigr)
\le
C_1\|x^k-\bar x\|+o(t_k).
\end{equation}

\smallskip
\noindent
\textbf{(c.) Project back to \(\cM_A(\bar x)\).}
Define the affine map
\begin{equation}
\label{eq:PsiA-def}
\Psi_A(v,\rho)\triangleq
-\nabla f(\bar x)-\nabla h(\bar x)^\top v-\nabla G(\bar x)^\top \rho.
\end{equation}
Then,
\begin{equation*}
(v,\rho)\in \cM_A(\bar x)
\quad\Longleftrightarrow\quad
(v,\rho,\Psi_A(v,\rho))\in \cM_{A}^{\rm lift}(\bar x).
\end{equation*}
That is, the lifted set is exactly the graph of \(\Psi_A\) over \(\cM_A(\bar x)\). Define
\begin{equation*}
T_A(v,\rho)\triangleq(v,\rho,\Psi_A(v,\rho)).
\end{equation*}
For every \((v,\rho)\),
\begin{equation}
\label{eq:dist-project-back}
\dist((v,\rho),\cM_A(\bar x))
\le
\dist(T_A(v,\rho),\cM_{A}^{\rm lift}(\bar x)),
\end{equation}
which is due to 
$\|T_A(v,\rho)-T_A(\bar v,\bar\rho)\|
\ge
\|(v,\rho)-(\bar v,\bar\rho)\|$ for any \((\bar v,\bar\rho)\in \cM_A(\bar x)\).

Next, we bound the distance between \(s^k\) and \(\Psi_A(v^k,\tilde\rho^k)\). Indeed, from the perturbed stationarity equation in \eqref{eq:corrected-stationarity}, we get
$s^k=-\nabla f(x^k)-\nabla h(x^k)^\top v^k-\nabla G(x^k)^\top \tilde\rho^k+\sigma^k$ for $k\in\integers_+$. Therefore,
\begin{equation*}
\begin{aligned}
s^k-\Psi_A(v^k,\tilde\rho^k)
&=\nabla f(\bar x)-\nabla f(x^k)
+\bigl(\nabla h(\bar x)-\nabla h(x^k)\bigr)^\top v^k 
+\bigl(\nabla G(\bar x)-\nabla G(x^k)\bigr)^\top \tilde\rho^k
+\sigma^k,
\end{aligned}
\end{equation*}
for $k\in\integers_+$. Since \(\{v^k\}\) and \(\{\tilde\rho^k\}\) are bounded, and \(\sigma^k=o(t_k)\), there exists \(C_2>0\) such that
\begin{equation}
\label{eq:s-Psi-bound}
\|s^k-\Psi_A(v^k,\tilde\rho^k)\|
\le
C_2\|x^k-\bar x\|+o(t_k).
\end{equation}
By the triangle inequality,
\begin{equation*}
\begin{aligned}
\dist(T_A(v^k,\tilde\rho^k),\cM_{A}^{\rm lift}(\bar x)) \le
\|\Psi_A(v^k,\tilde\rho^k)-s^k\|
+
\dist((v^k,\tilde\rho^k,s^k),\cM_{A}^{\rm lift}(\bar x)),\qquad\forall~k\in\integers_+.
\end{aligned}
\end{equation*}
Using \eqref{eq:lifted-distance-bound} and \eqref{eq:s-Psi-bound}, for \(C_3=C_1+C_2\),
\begin{equation*}
\dist(T_A(v^k,\tilde\rho^k),\cM_{A}^{\rm lift}(\bar x))
\le
C_3\|x^k-\bar x\|+o(t_k).
\end{equation*}
Finally, by \eqref{eq:dist-project-back},
\begin{equation}
\label{eq:MA-distance-bound}
\dist((v^k,\tilde\rho^k),\cM_A(\bar x))
\le
C_3\|x^k-\bar x\|+o(t_k).
\end{equation}

\medskip
\noindent
\textbf{(d.) Tighter control on $\dist((v^k,\tilde\rho^k),\cM_A(\bar x))$.}
The projection map $T'$, defined in \eqref{eq:projection-map}, satisfies \eqref{eq:Tp-map} and it is a linear map; therefore, we have
$\dist((v^k,A^\top \tilde\rho^k),\cM(\bar x))
\le
\|T'\|\dist((v^k,\tilde\rho^k),\cM_A(\bar x))$ for all $k\in\integers_+$.
Hence, by \eqref{eq:MA-distance-bound},
\begin{equation*}
\dist((v^k,A^\top \tilde\rho^k),\cM(\bar x))
\le
C_3 \|T'\|\|x^k-\bar x\|+o(t_k).
\end{equation*}
Moreover, since \(\lambda^k=A^\top \rho^k\), we have
$\|\lambda^k-A^\top \tilde\rho^k\|
=
\|A^\top (\rho^k-\tilde\rho^k)\|
\le
\|A^\top \|\,
\|\eta^k\|
=o(t_k)$.
Thus, for some constant $C>0$, we have
\begin{equation}
\label{eq:M-distance-bound-lambda}
\dist((v^k,\lambda^k),\cM(\bar x))
\le
C\|x^k-\bar x\|+o(t_k).
\end{equation}
Consequently, combining \eqref{eq:tk-def} and \eqref{eq:M-distance-bound-lambda}, we can conclude that
\begin{equation}
\label{eq:tk-order-ak}
t_k=\cO(\|x^k-\bar x\|),
\end{equation}
which is also shown for $\{t_k\}$ in the proof of \cite[Proposition 6.2.7]{facchinei2003finite}.
Therefore, from \eqref{eq:uhat-w-v-small} and \eqref{eq:sigma-small},
\begin{equation}
\label{eq:improved-residual-estimates}
\frac{u^k}{\|x^k-\bar x\|}\to0,
\qquad
\frac{w^k}{\|x^k-\bar x\|}\to0,
\qquad
\frac{\eta^k}{\|x^k-\bar x\|}\to0,
\qquad
\frac{\sigma^k}{\|x^k-\bar x\|}\to0.
\end{equation}
Finally, combining \eqref{eq:tk-order-ak} with \eqref{eq:MA-distance-bound}, we also have
\begin{equation}
\label{eq:MA-distance-O-ak}
\dist((v^k,\tilde\rho^k),\cM_A(\bar x))=\cO(\|x^k-\bar x\|).
\end{equation}

\medskip
\noindent
\textbf{Tangent direction construction.}
Let
\begin{equation*}
a_k\triangleq\|x^k-\bar x\|,
\qquad
 d^k\triangleq\frac{x^k-\bar x}{a_k},\qquad\forall~k\in\integers_+.
\end{equation*}
Passing to a subsequence, we can assume that there exists $d^\infty$ such that as $k\to\infty$,
\begin{equation}
\label{eq:dk-converges}
d^k\to d^\infty,
\qquad
\|d^\infty\|=1.
\end{equation}
Since \(\{x^k\}\subset \cX\), $x^k\to\bar x$ and $a_k\to 0$ as $k\to\infty$, we have
\begin{equation}
\label{eq:dinfty-tangent-X}
d^\infty\in \cT_{\cX}(\bar x).
\end{equation}
For every \(i\in I_G(\bar x)\), we have \(G_i(\bar x)=0\), and \eqref{eq:tilderho-complementarity} implies that \(G_i(x^k)+\eta_i^k\le0\) for $k\in\integers_+$; hence
\begin{equation}
\label{eq:G-linearization-ineq}
0\ge \eta_i^k+G_i(x^k)
=
 \eta_i^k+G_i(\bar x)+\nabla G_i(\bar x)^\top (x^k-\bar x)+o(a_k).
\end{equation}
Since \(G_i(\bar x)=0\), dividing \eqref{eq:G-linearization-ineq} by \(a_k\) and using \eqref{eq:improved-residual-estimates}, we obtain
\begin{equation}
\label{eq:linearized-G-feasible}
\nabla G_i(\bar x)^\top d^\infty\le0
\qquad
\forall i\in I_G(\bar x).
\end{equation}
Since \(G(\cdot)=Ag(\cdot)\), the condition in \eqref{eq:linearized-G-feasible} is equivalent to
\begin{equation}
\label{eq:linearized-conic-feasible}
\nabla g(\bar x)d^\infty\in \cT_{-\cK}(g(\bar x)).
\end{equation}
Moreover, since \(h(x^k)=w^k=o(a_k)\), we similarly get
\begin{equation}
\label{eq:linearized-h-feasible}
\nabla h(\bar x)d^\infty=0.
\end{equation}
Thus, \(d^\infty\in\cT_{\feas}(\bar x)\), i.e., it satisfies the linearized feasibility conditions in \eqref{eq:dinfty-tangent-X}, \eqref{eq:linearized-conic-feasible}, and \eqref{eq:linearized-h-feasible}.

\medskip
\noindent
\textbf{\(d^\infty\) is a critical direction.}
Since \(\{(v^k,\tilde\rho^k)\}\) is a bounded sequence, possibly after passing to a subsequence, we have
\begin{equation}
\label{eq:mu-rho-limit}
(v^k,\tilde\rho^k)\to(v^\infty,\rho^\infty)\qquad\mbox{as}\quad k\to\infty.
\end{equation}
By \eqref{eq:MA-distance-O-ak} and the closedness of \(\cM_A(\bar x)\), we have $(v^\infty,\rho^\infty)\in \cM_A(\bar x)$.
Let $\lambda^\infty\triangleq A^\top \rho^\infty$.
Then, $(v^\infty,\lambda^\infty)\in \cM(\bar x)$ and the stationarity inclusion at \(\bar x\) can be written as $\nabla_x\Phi(\bar x,v^\infty,\lambda^\infty)
\in -\cN_{\cX}(\bar x)$. Equivalently,
\begin{equation*}
s(\bar x,v^\infty,\rho^\infty)
\triangleq\nabla f(\bar x)+\nabla h(\bar x)^\top v^\infty+
\nabla G(\bar x)^\top \rho^\infty
\in -\cN_{\cX}(\bar x).
\end{equation*}
Define
\begin{equation*}
s^\infty\triangleq-s(\bar x,v^\infty,\rho^\infty)\in \cN_{\cX}(\bar x).
\end{equation*}
For $i\in[\bar m]$ such that \(\rho_i^\infty>0\), we have \(\tilde\rho_i^k>0\) for all sufficiently large \(k\), and the complementarity condition in \eqref{eq:tilderho-complementarity} implies that the following condition holds for all sufficiently large \(k\):
\begin{equation}
\label{eq:G-active-equality}
G_i(x^k)+\eta_i^k=0.
\end{equation}
Note that $\rho^\infty_i>0$ implies that $G_i(\bar x)=0$; hence, dividing \eqref{eq:G-active-equality} by \(a_k\) and using \eqref{eq:improved-residual-estimates} gives
\begin{equation}
\label{eq:positive-rho-zero-derivative}
\nabla G_i(\bar x)^\top d^\infty=0.
\end{equation}
Next, we show that
\begin{equation}
\label{eq:s-infty-orthogonality}
\langle s^\infty,d^\infty\rangle=0.
\end{equation}
Since \(s^k\in \cN_{\cX}(x^k)\) and \(\bar x\in \cX\), the normal-cone inequality gives $\langle s^k,\bar x-x^k\rangle\le0$,
or equivalently,
\begin{equation}
\label{eq:normal-ineq-sk-positive}
\left\langle s^k,~(x^k-\bar x)/a_k\right\rangle\ge0.
\end{equation}
In \eqref{eq:bounded-S}, we show that \(\{s^k\}\) is bounded; therefore, after passing to a subsequence, $s^k\to\bar s$ as $k\to\infty$ for some $\bar s$. Moreover, from \eqref{eq:s-Psi-bound} and \eqref{eq:tk-order-ak}, we get that as $k\to\infty$,
\begin{equation}
\label{eq:s-Psi-to-zero}
\|s^k-\Psi_A(v^k,\tilde\rho^k)\|\to0.
\end{equation}
Since \((v^k,\tilde\rho^k)\to(v^\infty,\rho^\infty)\) and $s^k\to\bar s$ hold, \eqref{eq:PsiA-def} and \eqref{eq:s-Psi-to-zero} imply that
\begin{equation*}
\bar s=\Psi_A(v^\infty,\rho^\infty)=s^\infty.
\end{equation*}
Thus, passing to the limit in \eqref{eq:normal-ineq-sk-positive} gives
$\langle s^\infty,d^\infty\rangle\ge0$. On the other hand, \(s^\infty\in \cN_{\cX}(\bar x)\) and \(d^\infty\in \cT_{\cX}(\bar x)\) also imply that
$\langle s^\infty,d^\infty\rangle\le0$. Therefore, \eqref{eq:s-infty-orthogonality} follows. By the definitions of $s(\bar x,v^\infty,\rho^\infty)$ and $s^\infty$, we have $s(\bar x,v^\infty,\rho^\infty)+s^\infty=0$; hence, taking the inner product of
this identity with \(d^\infty\) gives
\begin{equation*}
\langle \nabla f(\bar x),d^\infty\rangle=0,
\end{equation*}
where we used \eqref{eq:linearized-h-feasible}, \eqref{eq:positive-rho-zero-derivative}, and \eqref{eq:s-infty-orthogonality}. Together with \eqref{eq:linearized-G-feasible}, \eqref{eq:linearized-h-feasible}, and \eqref{eq:dinfty-tangent-X}, this yields
\begin{equation}
\label{eq:dinfty-critical}
d^\infty\in \cC(\bar x;\feas,\nabla f).
\end{equation}

\medskip
\noindent
\textbf{Sequence-level complementarity.}
For every \(i\in [\bar m]\) and for all sufficiently large \(k\), it holds that
\begin{equation}
\label{eq:rho-sequence-complementarity}
\tilde\rho_i^k\,\nabla G_i(\bar x)^\top d^\infty=0.
\end{equation}
Indeed, fix some arbitrary \(i\in [\bar m]\). If \(\tilde\rho_i^k>0\) for only finitely many \(k\), then eventually \(\tilde\rho_i^k=0\), and the conclusion is immediate. Suppose instead that \(\tilde\rho_i^k>0\) for infinitely many \(k\). Passing to that subsequence, still indexed by \(k\), we have
\begin{equation}
\label{eq:G-v-zero-subseq}
G_i(x^k)+\eta_i^k=0.
\end{equation}
Since \(\eta_i^k\to0\) and \(x^k\to\bar x\), it follows that \(G_i(\bar x)=0\). Dividing \eqref{eq:G-v-zero-subseq} by \(a_k\), and using \eqref{eq:improved-residual-estimates}, gives
$\nabla G_i(\bar x)^\top d^\infty=0$.
Thus, \eqref{eq:rho-sequence-complementarity} holds for this scenario as well.

We also need the corresponding identity for the normal-cone term:
\begin{equation}
\label{eq:s-sequence-complementarity}
\langle s^k,d^\infty\rangle=0
\end{equation}
for all sufficiently large \(k\). To see this, write the polyhedral set \(\cX\) as
\begin{equation*}
\cX=\{x:Fx\le \bar f,\ Ex=\bar e\}.
\end{equation*}
Then, every \(s^k\in \cN_{\cX}(x^k)\) admits the following representation:
\begin{equation*}
s^k=F^\top \alpha^k+E^\top \beta^k,
\qquad
\alpha^k\ge0,
\qquad
\alpha_j^k(f_j^\top x^k-\bar f_j)=0
\quad\forall j,
\end{equation*}
where $f_j^\top$ denotes the $j$-th row of $F$. As \(d^\infty\in \cT_{\cX}(\bar x)\), we have
\begin{equation*}
Ed^\infty=0,
\qquad
f_j^\top d^\infty\le0
\quad\forall j\in I_{\cX}(\bar x),
\end{equation*}
where $I_{\cX}(\bar x)\triangleq\{j:f_j^\top \bar x=\bar f_j\}$.

Fix an arbitrary row index $j$. If \(j\notin I_{\cX}(\bar x)\), then \(f_j^\top \bar x<\bar f_j\), which implies that \(f_j^\top x^k<\bar f_j\) for all large \(k\); therefore \(\alpha_j^k=0\) eventually. On the other hand, if \(\alpha_j^k>0\) infinitely often, then for those \(k\), we have $f_j^\top x^k-\bar f_j=0$. Since \(\bar x\in \cX\), this implies that \(j\in I_{\cX}(\bar x)\); therefore, $f_j^\top d^\infty=\lim_{k\to\infty} (f_j^\top x^k-f_j^\top \bar x)/a_k=0$. Thus, for each fixed inequality index \(j\), \(\alpha_j^k f_j^\top d^\infty=0\) eventually. Since there are only finitely many inequality indices, there exists $K'\in\integers_+$ such that for $k\geq K'$, this conclusion holds simultaneously for all \(j\), i.e., $\alpha_j^k f_j^\top d^\infty=0$ holds for all $j$ eventually. Since \(Ed^\infty=0\), we obtain
\begin{equation*}
\langle s^k,d^\infty\rangle
=
\sum_j\alpha_j^k f_j^\top d^\infty
+
\langle \beta^k,Ed^\infty\rangle
=0,\qquad\forall~k\geq K'.
\end{equation*}

\medskip
\noindent
\textbf{Contradiction.} We first verify that $d^\infty$ solves the second-order homogeneous system. Define
\begin{equation*}
\Phi_A(x,v,\rho)\triangleq f(x)+\langle v, h(x)\rangle+\langle\rho, G(x)\rangle.
\end{equation*}
Then the perturbed stationarity equation in \eqref{eq:corrected-stationarity} can be written as
\begin{equation}
\label{eq:LA-perturbed-stationarity}
\nabla_{x}\Phi_A(x^k,v^k,\tilde\rho^k)+s^k-\sigma^k=0,\qquad\forall~k\in\integers_+.
\end{equation}
Since \(\{(v^k,\tilde\rho^k)\}\) is bounded, \(f,h\) and \(G\) are \(C^2\), the Taylor remainder in the expansion of \(\nabla_{x}\Phi_A(\cdot,v^k,\tilde\rho^k)\) at \(\bar x\) is \(o(a_k)\) uniformly along the sequence. Hence, for all $k\in\integers_+$,
\begin{equation}
\label{eq:LA-taylor}
\nabla_{x}\Phi_A(x^k,v^k,\tilde\rho^k)
=
\nabla_{x}\Phi_A(\bar x,v^k,\tilde\rho^k)
+
\nabla_{xx}^2\Phi_A(\bar x,v^k,\tilde\rho^k)(x^k-\bar x)
+
e^k,
\end{equation}
where
\begin{equation}
\label{eq:taylor-remainder-small}
\frac{\|e^k\|}{a_k}\to0.
\end{equation}
Combining \eqref{eq:LA-perturbed-stationarity} and \eqref{eq:LA-taylor}, we obtain
\begin{equation}
\label{eq:stationarity-dualcone-rhs}
\nabla_{xx}^2\Phi_A(\bar x,v^k,\tilde\rho^k)(x^k-\bar x)
+
e^k-
\sigma^k
=
-\bigl(\nabla_{x}\Phi_A(\bar x,v^k,\tilde\rho^k)+s^k\bigr),\qquad\forall~k\in\integers_+.
\end{equation}
Let $\mathcal{C}\triangleq\cC(\bar x;\feas,\nabla f)$. We claim that, for all sufficiently large \(k\),
\begin{equation}
\label{eq:minus-LA-plus-s-in-dual}
-\bigl(\nabla_{x}\Phi_A(\bar x,v^k,\tilde\rho^k)+s^k\bigr)
\in \mathcal{C}^*.
\end{equation}
Indeed, fix an arbitrary \(d\in \mathcal{C}\). Then, the condition in \eqref{eq:critical-cone-def} implies that
\begin{equation*}
\langle \nabla f(\bar x),d\rangle=0,
\qquad
\nabla h(\bar x)d=0,
\qquad
d\in \cT_{\cX}(\bar x),
\qquad
\nabla G_i(\bar x)^\top d\le0
\qquad
\forall i\in I_G(\bar x).
\end{equation*}
By \eqref{eq:tilderho-inactive-zero}, \(\tilde\rho_i^k=0\) for \(i\notin I_G(\bar x)\), and \(\tilde\rho_i^k\ge0\) for all \(i\in [\bar m]\). Moreover, by \eqref{eq:normal-inclusion-before-corollary}, \(s^k\in \cN_{\cX}(\bar x)\) for all sufficiently large \(k\). Finally, since \(d\in \cT_{\cX}(\bar x)\), we also have $\langle s^k,d\rangle\le0$ eventually. Therefore,
\begin{equation*}
\begin{aligned}
\langle \nabla_{x}\Phi_A(\bar x,v^k,\tilde\rho^k)+s^k,d\rangle
&=
\langle \nabla f(\bar x),d\rangle
+
\langle v^k,\nabla h(\bar x)d\rangle
+
\sum_{i=1}^{\bar m}\tilde\rho_i^k\nabla G_i(\bar x)^\top d
+
\langle s^k,d\rangle
\\
&=
\sum_{i\in I_G(\bar x)}\tilde\rho_i^k\nabla G_i(\bar x)^\top d
+
\langle s^k,d\rangle
\le0.
\end{aligned}
\end{equation*}
Thus, $\left\langle -\bigl(\nabla_{x}\Phi_A(\bar x,v^k,\tilde\rho^k)+s^k\bigr),d\right\rangle
\ge0$ holds for all $d\in \mathcal{C}$, which proves \eqref{eq:minus-LA-plus-s-in-dual}. Furthermore, from \eqref{eq:stationarity-dualcone-rhs} and \eqref{eq:minus-LA-plus-s-in-dual}, we get
\begin{equation}
\label{eq:prelimit-dual-inclusion}
\nabla_{xx}^2\Phi_A(\bar x,v^k,\tilde\rho^k)(x^k-\bar x)
+
e^k-
\sigma^k
\in \mathcal{C}^*.
\end{equation}
Dividing \eqref{eq:prelimit-dual-inclusion} by \(a_k\), using \eqref{eq:improved-residual-estimates}, \eqref{eq:taylor-remainder-small}, \eqref{eq:dk-converges}, and \eqref{eq:mu-rho-limit}, we obtain
\begin{equation}
\label{eq:limit-dual-inclusion}
\nabla_{xx}^2\Phi_A(\bar x,v^\infty,\rho^\infty)d^\infty
\in \mathcal{C}^*.
\end{equation}
Finally, it remains to prove orthogonality. Recall that in \eqref{eq:rho-sequence-complementarity} and \eqref{eq:s-sequence-complementarity}, we established that for every \(i\in\{1,\ldots,\bar m\}\) and all sufficiently large \(k\), we have $\tilde\rho_i^k\nabla G_i(\bar x)^\top d^\infty=0$ and $\langle s^k,d^\infty\rangle=0$. Since \(d^\infty\in \mathcal{C}\), we also have $\langle \nabla f(\bar x),d^\infty\rangle=0$ and $\nabla h(\bar x)d^\infty=0$. Therefore,
\begin{equation}
\label{eq:zero-pairing-LA-s}
\langle d^\infty,\nabla_{x}\Phi_A(\bar x,v^k,\tilde\rho^k)+s^k\rangle=0.
\end{equation}
Taking the inner product of \eqref{eq:stationarity-dualcone-rhs} with \(d^\infty\), and using \eqref{eq:zero-pairing-LA-s}, gives
\begin{equation}
\label{eq:orthogonality-prelimit}
\left\langle d^\infty,
\nabla_{xx}^2\Phi_A(\bar x,v^k,\tilde\rho^k)(x^k-\bar x)
+
e^k-
\sigma^k
\right\rangle=0.
\end{equation}
Dividing \eqref{eq:orthogonality-prelimit} by \(a_k\), and then passing to the limit using \eqref{eq:improved-residual-estimates}, \eqref{eq:taylor-remainder-small}, \eqref{eq:dk-converges}, and \eqref{eq:mu-rho-limit}, we obtain
\begin{equation}
\label{eq:limit-orthogonality}
\langle d^\infty,\nabla_{xx}^2\Phi_A(\bar x,v^\infty,\rho^\infty)d^\infty\rangle=0.
\end{equation}
Combining \eqref{eq:limit-dual-inclusion} and \eqref{eq:limit-orthogonality}, we conclude that $d^\infty
\perp
\nabla_{xx}^2\Phi_A(\bar x,v^\infty,\rho^\infty)d^\infty
\in \mathcal{C}^*$. Finally, since \(\lambda^\infty=A^\top \rho^\infty\), we get $\nabla_{xx}^2\Phi_A(\bar x,v^\infty,\rho^\infty)d^\infty
=
\nabla_{xx}^2\Phi(\bar x,v^\infty,\lambda^\infty)d^\infty$. Thus, we can conclude that
\begin{equation}
\label{eq:homogeneous-system-obtained}
d^\infty
\perp
\nabla_{xx}^2\Phi(\bar x,v^\infty,\lambda^\infty)d^\infty
\in
\cC(\bar x;\feas,\nabla f)^*.
\end{equation}
To conclude the proof, note that by \eqref{eq:dinfty-critical}, \eqref{eq:dk-converges}, and \eqref{eq:homogeneous-system-obtained}, we have constructed $d^\infty$ such that
\begin{equation*}
d^\infty\in \cC(\bar x;\feas,\nabla f),
\qquad
\|d^\infty\|=1,
\qquad
d^\infty
\perp
\nabla_{xx}^2\Phi(\bar x,v^\infty,\lambda^\infty)d^\infty
\in
\cC(\bar x;\feas,\nabla f)^*.
\end{equation*}
This contradicts \eqref{eq:homogeneous-system-no-q}; therefore, the desired local error bound in~\eqref{eq:error-bound-conclusion} holds.

\end{document}